\documentclass[12pt]{amsart}
\usepackage{amssymb,mathrsfs,amsmath}
\usepackage{amscd,amsmath,amsthm,amssymb,amsfonts,enumerate}
\usepackage{natbib}
\usepackage{graphicx, color}

\usepackage{epstopdf}

\usepackage{hyperref}
\hypersetup{
    colorlinks=true,
    linkcolor=blue,
    filecolor=magenta,
    urlcolor=cyan,
    citecolor=blue,
}
\usepackage{subfig}

\def\({\left(}
\def\]{\right]}
\def\[{\left[}
\def\){\right)}

\newtheorem{thm}{Theorem}[section]

\newtheorem{asm}[thm]{Assumption}
\newtheorem{prop}[thm]{Proposition}

\newtheorem{rem}[thm]{Remark}

\newtheorem{defn}[thm]{Definition}
\newtheorem{exm}[thm]{Example}

\def\P{{\mathbb P}}
\def\E{{\mathbb E}}
\def\Cov{{\mathbb Cov}}
\def\R{{\mathbb R}}

\newcommand{\disp}{\displaystyle}
\newcommand{\bea}{$$\begin{array}{ll}}
\newcommand{\eea}{\end{array}$$}
\newcommand{\bed}{\begin{displaymath}}
\newcommand{\eed}{\end{displaymath}}
\newcommand{\ad}{&\!\!\!\disp}
\newcommand{\aad}{&\disp}
\newcommand{\barray}{\begin{array}{ll}}
\newcommand{\earray}{\end{array}}
\newcommand{\beq}[1]{\begin{equation} \label{#1}}
\newcommand{\eeq}{\end{equation}}

\newcommand{\bedd}{\bed\begin{array}{l}}
\newcommand{\eedd}{\end{array}\eed}

\newcommand{\sg}{\sigma}

\newcommand{\e}{\varepsilon}

\newcommand{\one}{{1}}

\newcommand{\wdt}{\widetilde}
\newcommand{\wdh}{\widehat}
\newcommand{\diag}{{\rm diag}}
\newcommand{\cd}{(\cdot)}
\newcommand{\rr}{\Bbb R}
\newcommand{\lbar}{\overline}

\renewcommand{\L}{\mathcal {L}}
\def\one{{\hbox{1{\kern -0.35em}1}}}

\def\ph{\varphi}

\def\({\left(}
\def\]{\right]}
\def\[{\left[}
\def\){\right)}
\def\one{{\hbox{1{\kern -0.35em}1}}}

\makeatletter \@addtoreset{equation}{section}

\def\para#1{\vskip 0.4\baselineskip\noindent{\bf #1}}

\def\tr{\hbox{tr}}

\def\xh{\xi^h}

\def\ei{{\bf e_i}}

\def\ej{{\bf e_j}}
\def\argmax{\hbox{argmax}}

\begin{document}

\title[Harvesting and seeding of stochastic populations]{Harvesting and seeding of stochastic populations: analysis and numerical approximation}

\author[A. Hening]{Alexandru Hening }
\address{Department of Mathematics\\
Tufts University\\
Bromfield-Pearson Hall\\
503 Boston Avenue\\
Medford, MA 02155\\
United States
}
\email{alexandru.hening@tufts.edu}

\author[K. Tran]{Ky Quan Tran}
\address{Department of Applied Mathematics and Statistics\\
The State University of New York, Korea\\
119-2 Songdo Moonhwa-ro\\
Yeonsu-Gu, Incheon 21985\\
 Korea
}
\email{ky.tran@stonybrook.edu}

\maketitle

\begin{abstract}
It is well known that excessive harvesting or hunting has driven species to extinction both on local and global scales. This leads to one of the fundamental problems of conservation ecology: how should we harvest a population so that economic gain is maximized, while also ensuring that the species is safe from extinction? Our work analyzes this problem in a general setting. We study an ecosystem of interacting species that are influenced by random environmental fluctuations. At any point in time, we can either harvest or seed (repopulate) species. Harvesting brings an economic gain while seeding incurs a cost. The problem is to find the optimal harvesting-seeding strategy that maximizes the expected total income from harvesting minus the cost one has to pay for the seeding of various species. In \cite{Ky18} we considered this problem when one has absolute control of the population (infinite harvesting rates are possible) as well as absolute repopulation options (infinite seeding rates are possible). In many cases, these approximations do not make biological sense and one must consider what happens when one, or both, of the seeding and harvesting rates are bounded. The focus of this paper is the analysis of these three novel settings: bounded seeding and infinite harvesting, bounded seeding and bounded harvesting, and infinite seeding and bounded harvesting.

Even one dimensional harvesting problems can be hard to tackle. Once one looks at an ecosystem with more than one species analytical results usually become intractable. In our setting, the fact that we have both harvesting and seeding and that the seeding and/or harvesting rates are bounded, significantly complicate the problem. We are able to prove some analytical results regarding the optimal yield and the optimal harvesting--seeding strategies. In order to gain more information regarding the qualitative behavior of the system we develop rigorous numerical approximation methods. This is done by approximating the continuous time dynamics by Markov chains and then showing that the approximations converge to the correct optimal strategy as the mesh size goes to zero. By implementing these numerical approximations, we are able to gain qualitative information about how to best harvest and seed species in specific key examples.

We are able to show through numerical experiments that in the single species setting the optimal seeding-harvesting strategy is always of threshold type. This means there are thresholds $0<L_1<L_2<\infty$ such that: 1) if the population size is `low',  so that it lies in $(0, L_1]$, there is seeding using the maximal seeding rate; 2) if the population size `moderate', so that it lies in $(L_1,L_2)$, there is no harvesting or seeding; 3) if the population size is `high', so that it lies in the interval $[L_2, \infty)$, there is harvesting using the maximal harvesting rate. Once we have a system with at least two species, numerical experiments show that constant threshold strategies are not optimal anymore. Suppose there are two competing species and we are only allowed to harvest or seed species 1. The optimal strategy of seeding and harvesting will involve lower and upper thresholds $L_1(x_2)<L_2(x_2)$ which depend on the density $x_2$ of species $2$.
\end{abstract}

\maketitle

\setlength{\baselineskip}{0.22in}

\tableofcontents

\section{Introduction}\label{sec:int}

On one hand, species usually interact in complex ways within their ecosystems. On the other hand, environmental fluctuations have been shown to strongly influence the population dynamics of species (\cite{ACG87}). There are examples where the environmental fluctuations can drive a species extinct as well as examples where the environmental fluctuations create a rescue effect that saves species from extinction. In order to get a realistic idea of to the long term fate of species it is of fundamental importance to consider the combined effects of biotic interactions and environmental fluctuations.  Starting with the illuminating work of Peter Chesson (\cite{CW81, C82, C94, CH97}), and building on deterministic persistence theory (\cite{H81, H84, HJ89, HS98, ST11}), there is now a powerful theory of stochastic persistence (\cite{SBA11, B18, BS18, HN18, CHN19}).

Many species are not only influenced by their interactions and the environment -- they are also harvested by humans. Excessive harvesting and hunting can lead species to become locally or globally extinct \cite{LES95, LES}. If one looks at the harvesting problem strictly from a conservation point of view, it makes sense to harvest less in order to minimize the extinction risk. This can lead to a significant economic loss due to underharvesting.

As explained by \cite{Ky18} in specific situations one can repopulate (or seed) a species which is at risk of extinction. This happens for example in fisheries or other restricted conservation habitats where one can control the population. From an economic point of view, there is a cost whenever one seeds and a gain when one harvests.

Taking into account all these factors one is faced with the following fundamental problem. Suppose we have an ecosystem of $d$ species that interact, possibly nonlinearly, due to competition for resources, predation, cooperation, mutualism etc, are influenced by random environmental fluctuations and can be controlled through seeding and harvesting.  How should we harvest/seed in order to maximize revenue (gain from harvesting minus loss from seeding) while ensuring species do not go extinct? The various factors (biotic interactions, random environmental fluctuations, economic gain, extinction risk) have to be carefully taken into account if one wants to find a viable exploitation strategy.

We model the populations in continuous time under the assumption that there is environmental stochasticity and no demographic stochasticity. Mathematically this means we look at systems of stochastic differential equations (SDE). There is evidence that SDE are often good approximations of discrete time biological systems (\cite{LES95, T77}). Intuitively, in our setting one can imagine that the random fluctuations in the small time time $dt$ look like $X_tdW_t$ where $(W_t)$ is a Brownian motion. This type of noise has the property that, if there is no harvesting, extinction can only occur asymptotically as time goes to infinity. In contrast, demographic stochasticity is usually modelled by fluctuations of the form $\sqrt{X_t}dW_t$ in a small time $dt$ and implies finite time extinctions. Even though it is biologically clear that extinction is always inevitable, there are settings, where extinction happens after long periods of time and neglecting demographic stochasticity is a good first approximation.

Our analysis builds on the significant results that are available in the stochastic harvesting literature. If there is only one species, the state of the art is contained in results by \cite{Alvarez98, Alvarez00, Lungu, Zhu11, HNUW18, AH18}. Significantly fewer results are available if one is interested in multiple interacting species (\cite{Lungu01, Ky17, Ky18}).

We initiated a rigorous analysis of the multispecies harvesting-seeding problem in a previous paper (\cite{Ky18}). As a result we were able to get analytical and numerical results when one assumes that the seeding and harvesting rates are unbounded. In many interesting scenarios this assumption is not realistic. For example, one will usually not be able to seed a population at extremely high rates - it would therefore be more natural to assume that the seeding rate has an upper threshold which cannot be exceeded. Similarly, in other settings it might make sense to assume that the harvesting rate is bounded above. We study the following three novel scenarios:
\begin{itemize}
\item  Bounded seeding and unbounded harvesting rates.
\item  Bounded seeding and bounded harvesting rates.
\item  Unbounded seeding and bounded harvesting rates.
\end{itemize}

In order to study this stochastic singular control problem, the standard approach is to look at the associated Hamilton-Jacobi-Bellman (HJB) partial differential equations. We were able to do this when we assumed that the seeding and harvesting rates are unbounded (\cite{Ky18}). We prove a similar result in the setting of bounded seeding and harvesting rates. If one rate is bounded and the other one is unbounded, due to significant additional technical difficulties, we were not able to show the HJB equation holds. In order to gain some qualitative information, we develop numerical algorithms to approximate the value function (maximal discounted revenue) and the optimal harvesting-seeding strategy.
This is accomplished by making use of the Markov chain approximation methodology developed by \cite{Kushner92}.

 The main contributions of our work are the following:

\begin{enumerate}
\item We analyze the harvesting-seeding problem for a system of interacting species living in a stochastic environment, when the seeding and/or harvesting rates are bounded.
\item We prove analytical results and develop rigorous approximation schemes. We show that these approximation schemes converge to the correct optimal harvesting-seeding strategy (and value function) as the mesh size goes to zero.
\item We apply the approximation schemes to illuminating examples with one or two species in order to see what qualitatively new phenomena emerge due to the interspecies and intraspecies interaction terms, the environmental fluctuations and the boundedness of the seeding/harvesting strategies. In particular we show that the well-known threshold harvesting strategies are not optimal anymore when one can harvest multiple species.
\end{enumerate}

Harvesting species that are part of complex food webs has led to overexploitation and in some cases to extinctions. This happens, in part, because when one picks harvesting strategies the complex interactions of the species and the environmental fluctuations are not taken into account. In some instances, one harvests one specific species from the ecosystem, and ignores the rest. This can disrupt the ecosytem and lead to conservation problems. The fundamental work by \cite{M79} has shown that harvesting at a constant rate and maximizing the MSY (maximum sustainable yield) for specific species in an ecosystem with multiple species is insufficient for conservation purposes. Harvesting at a constant rate has been shown to have many shortcoming even if the harvested stock can be regarded as an isolated population \cite{M78, M79, LES95}. In order to solve this issue, threshold harvesting, where one harvests only the fraction of the population above a fixed threshold has been shown to mitigate the risk of extinction (\cite{LSE97}). Multiple studies have proved rigorously that threshold harvesting of a single isolated species living in a stochastic environment is also optimal from an economic point of view (\cite{AS98, LES95, AH18, HNUW18}. Nevertheless, it is not clear how well threshold harvesting works for multispecies systems. By looking at ecosystems with two species we show that, if one is allowed to harvest both species, threshold harvesting for each species is not optimal anymore. Instead, there exists a complicated surface $S(x_1,x_2)$ such that whenever the population sizes $(x_1,x_2)$ are above the surface we harvest at the maximal rate, while if we are below the surface we never harvest. The interaction of the species make constant threshold strategies suboptimal. Even if we are only allowed to harvest and seed species 1, due to the interaction of the two species, the optimal seeding-harvesting strategyfor species 1 will depend on the density of species 2.

 The rest of our work is organized as follows. In Section \ref{sec:for} we describe our model and the main results.
 Particular examples are explored using the newly developed numerical schemes in Section \ref{sec:fur}.  Finally, all the technical proofs appear in the appendices.

\section{Model and Results}\label{sec:for}
Assume we have a probability space $(\Omega, \mathcal{F}, \P)$ and a filtration $(\mathcal{F}_t)_{t\geq 0}$ satisfying the usual conditions. We consider $d$ species interacting nonlinearly in a stochastic environment. We model the dynamics as follows. Let $\xi_i(t)$ be the
population abundance of the $i$th species at time $t\geq 0$, and
denote by $\xi(t)=(\xi_1(t), \dots, \xi_d(t))'\in \rr^d$ (where $z'$ denotes
the transpose of $z$) the column vector recording all the population abundances.

Based on the assumption that the environment mainly affects the growth/death rates of the populations and the approach in \cite{T77, B02, G88, ERSS13, SBA11, G84}, we consider the dynamics given by
\beq{e1.2.1} d \xi(t)=b(\xi(t))
dt+\sigma(\xi(t)) d w(t),\eeq
where $w\cd=\(w_1\cd, ..., w_d \cd\)'$ is a
$d$-dimensional standard Brownian motion adapted to $(\mathcal{F}_t)_{t\geq 0}$ and $b,\sigma:[0,\infty)^d\to \R^d$ are locally Lipschitz continuous functions.  Let $S=(0, \infty)^d$ and $\lbar S =[0, \infty)^d$.  We assume that $b(0)=\sg(0)=0$ so that $0$ is an equilibrium point of \eqref{e1.2.1}. This makes sense because if our populations go extinct, they should not be able to get resurrected without external intervention (like a repopulation/seeding event). If $\xi_i(t_0)=0$ for some $t_0\ge 0$, then $\xi_i(t)=0$ for any $t\ge t_0$. Thus, $\xi(t)\in \lbar S$ for any $t\ge 0$.

  Let $Y_i(t)$ denote the amount of species $i$ that has been harvested up to time $t$ and set
$Y(t)=(Y_1(t), \dots, Y_d(t))'\in \rr^d$.
Let $Z_i(t)$ denote the amount of species $i$ seeded into the system up to time $t$.
If we add the harvesting and seeding effects to \eqref{e1.2.1} we note that the dynamics of the $d$ species becomes
\beq{e1.2.2}X(t)=x+\int\limits_0^t b(X(s))ds +  \int\limits_0^t  \sigma(X(s)) dw(s) - Y(t) +Z(t),\eeq
where  $X(t)=(X_1(t), \dots, X_d(t))'\in \rr^d$ are the species populations at time $t\geq 0$. We assume the initial population abundances, before any seeding or harvesting, are
\beq{e1.3}X(0-)=x\in \lbar S.\eeq
\para{Notation.}
For $x, y\in \rr^d$, with $x=(x_1, \dots, x_d)'$ and $y=(y_1, \dots, y_d)'$,
we define the scalar product $x\cdot y=\sum_{i=1}^d x_i y_i$ and the norm $|x|=\sqrt{x\cdot x}$.
Let $\ei\in \rr^d$ denote the unit vector in the $i$th direction for $i=1, \dots, d$. If $x=(x_1, \dots, x_d)'\in \rr^d$ and $y=(y_1, \dots, y_d)'\in \rr^d$ and $x_i\le y_i$ for each $i$, we write $x\le y$ and we define $[x, y]=\{\xi=(\xi_1, \dots, \xi_d)': x_i\le \xi_i\le y_i, i=1, \dots, d\}$. For a real number $r$ let $r^+ := \max\{r, 0\}$ and $r^-:=\max\{-r, 0\}$.
Let $\mathcal{L}$ be the infinitesimal generator of the process $\xi(t)$ from \eqref{e1.2.1}. This linear operator acts as
	\begin{equation}\label{e1.ge}
	\L \Phi(x)=b(x)\nabla \Phi(x)+\dfrac{1}{2}\tr\big(\sigma(x)\sigma'(x)\nabla^2 \Phi(x)\big),\end{equation}
		on twice continuously differentiable functions $\Phi\cd: \mathbb{R}^d\to \mathbb{R}$. We write
	$\nabla \Phi(\cdot)$ and $\nabla^2 \Phi(\cdot)$ for the gradient and the Hessian matrix of $\Phi(\cdot)$.

We suppose that  the instantaneous marginal yield accrued
from exerting the harvesting strategy $Y_i$ for the species $i$ is $f_i: \overline S\to (0, \infty)$. This is also known as the price of species $i$. Let $g_i:  \overline{S}\mapsto (0, \infty)$ represent the marginal cost we need to pay for the seeding of species $i$ under the strategy $Z_i$. We will set $f=(f_1, \dots, f_d)'$ and   $g=(g_1, \dots, g_d)'$.
For a harvesting-seeding strategy $(Y, C)$ we define the \textit{performance function} as
\beq{e1.2.4}
J(x, Y, C):= \E_{x}\bigg[\int\limits_0^{\infty} e^{-\delta s} f (X(s-))\cdot dY(s)-\int\limits_0^{\infty} e^{-\delta s} g\( X(s-) \) \cdot dZ(s)\bigg],
\eeq
where
 $\delta> 0$ is the discounting factor, $\E_{x}$ is the expectation with respect
to the probability law when the initial populations are $X(0-)=x$,
 and $f(X(s-))\cdot dY(s):=\sum_{i=1}^n f_i (X(s-))dY_i(s)$. One can see that the performance function looks at the expected current value of the total gain from harvesting minus the current value of the total cost paid to seed species into the system.

\para{Control strategy.}
Let $\mathcal{A}_{x}$ denote the collection of all admissible controls with initial condition $X(0-)=x\in \overline{S}$. A harvesting-seeding strategy $(Y, Z)$ is in $\mathcal{A}_{x}$ if it satisfies
 the following
conditions:
\begin{itemize}
\item[{\rm (a)}] The processes $Y(t)$ and $Z(t)$ are right continuous, nonnegative, and nondecreasing with respect to $t$,
\item[{\rm (b)}] The processes $Y(t)$ and $Z(t)$ are adapted to the filtration $(\mathcal{F}(t))_{t\geq 0}$,
\item[{\rm (c)}] The system given by \eqref{e1.2.2} and \eqref{e1.3} has a unique solution with $X(t)\ge 0$ for all $t\ge 0$,
\item[{\rm (d)}] For any $x\in \overline{S}$ one has $0\le J(x, Y, Z)<\infty$.
\end{itemize}

\textbf{The optimal harvesting-seeding problem.} The problem we will be interested in is to maximize the
performance function and find an optimal harvesting-seeding strategy $(Y^*, Z^*)\in \mathcal{A}_{x}$ such that
\beq{e.5}
J(x, Y^*, Z^*)=V(x):= \sup\limits_{(Y, Z)\in \mathcal{A}_{x}}J(x, Y, Z).
\eeq
The function $V(\cdot)$ is called the \textit{value function}.

\begin{asm}\label{a:1}
We will make the following standing assumptions throughout the paper.
\begin{itemize}
\item[{\rm (a)}]  The functions $b(\cdot)$ and $\sigma(\cdot)$ are locally Lipschitz continuous. Moreover, for any initial condition $x\in  \lbar S$, the uncontrolled system \eqref{e1.2.1}  has a unique global solution in $S$.
\item[{\rm (b)}]
For any $i=1, \dots, d$, $x, y\in \rr^d$, $f_i(x)<g_i( x)$; $f_i\cd$, $g_i\cd$ are continuous and non-increasing functions.
\end{itemize}
\end{asm}
\begin{rem}
 Note that Assumption \ref{a:1} (a) is very general and includes most common ecological models, for example Lotka-Volterra competition and predator-prey models as well as general Kolmogorov systems \cite{Dang, Li09, Mao2006, HN18}.  Assumption \ref{a:1} (b)
 is natural: it just means that the gain from harvesting should always be strictly less than the cost of seeding.
In \cite{Ky18}, we analyzed the general case when both $Y\cd$ and $Z\cd$ are singular controls, i.e., they are not absolutely continuous with respect to time. In this paper, we focus on the case when at least one of these two controls is absolutely continuous and has a bounded rate.   We refer to \cite{Ky18} for further details regarding Assumption \ref{a:1}, the optimal harvesting-seeding problem, and properties of the value function.
\end{rem}

We analyze the following three scenarios.
\begin{itemize}
\item  Bounded seeding and unbounded harvesting rates: there is $\lambda=(\lambda_1, \dots, \lambda_d)'\in [0, \infty)^d$ such that $dZ(t)=C(t)dt$ for an adapted process $(C(t))_{t\geq 0}$ such that $0\le C(t)\le \lambda$. We call $\lambda$ the \textit{maximum seeding rate}. For convenience, we also denote by $\mu=(\mu_1, \dots, \mu_d)'$ the maximum harvesting rate and in this scenario we have $\mu=\infty$; that is, $\mu_i=\infty$ for any $i=1, \dots d$.

\item  Bounded seeding and bounded harvesting rates: we have $dY(t)=R(t)dt$ for an adapted process $(R(t))_{t\geq 0}$ such that $0\le R(t)\le \mu, t\geq0$ and $dZ(t)=C(t)dt$ for an adapted process $(C(t))_{t\geq 0}$ with $0\le C(t)\le \lambda, t\geq 0$.

\item  Unbounded seeding and bounded harvesting rates: we have $dY(t)=R(t)dt$ for an adapted process $(R(t))_{t\geq 0}$ such that $0\le R(t)\le \mu, t\geq0$.
\end{itemize}

In each of these three settings we will prove that there is a numerical approximation scheme that converges to the correct value function as the step size goes to zero. In addition, if both the seeding and harvesting rates are bounded, we show that the value function solves the Hamilton-Jacbi-Bellman equation in a weak sense.

\subsection{Bounded seeding and unbounded harvesting rates}
\

Since the seeding is bounded we have $$dZ(t)=C(t)dt$$ where $$0\le C(t)\le \lambda, t\geq 0.$$ Without loss of generality, we identify the process $(Y, Z)$ with $(Y, C)$. We will furthermore assume in this section that the price functions are constant so that $f_i(x)=f_i, x\in \lbar S$. The dynamics of the populations affected by harvesting and seeding will be
\beq{e2.2.2}X(t)=x+\int\limits_0^t \big[b(X(s)) + C(s)\big]ds +  \int\limits_0^t  \sigma(X(s)) dw(s) - Y(t),\eeq
while the performance function takes the form
\beq{e2.2.4}
J(x, Y, C):= \E_{x}\bigg[\int\limits_0^{\infty} e^{-\delta s} f \cdot dY(s)-\int\limits_0^{\infty} e^{-\delta s} g\( X(s) \) \cdot C(t) ds\bigg].
\eeq

Pick a large number $U>0$ and define the class $\mathcal{A}^U_{x}\subset \mathcal{A}_{x}$ that consists of strategies $(Y,C)\in \mathcal{A}_{x}$ such that the resulting process $X$ stays in $[0,U]^d$ for all times. The class $\mathcal{A}^U_{x}$ can be constructed using Skorokhod stochastic differential equations (\cite{B98, F16, LS84, Kushner92}) which force the process to stay in $[0,U]^d$ for all $t>0$.

We let $V^U(x)$ be the value function when we restrict the problem to the class $\mathcal{A}^U_{x}\subset \mathcal{A}_{x}$. In other words
\begin{equation}\label{e:VU}
V^U(x):=\sup\limits_{(Y, C)\in \mathcal{A}_{x}^U} J(x, Y, C).
\end{equation}

In earlier work (\cite{Ky18}) we conjectured that, generically, the optimal strategy will live in $\mathcal{A}_{x}^U$ for $U$ large enough. In the current formulation, we can restate the conjecture as follows: there exists $U>0$ such that for all $x\in [0,U]^d$ we have
	$$
	V(x):= \sup\limits_{(Y, C)\in \mathcal{A}_{x}}J(x, Y, C) =V^U(x):= \sup\limits_{(Y, C)\in \mathcal{A}_{x}^U} J(x, Y, C).
	$$
We are able to prove this conjecture under a natural assumption.

\begin{prop}\label{prop1}
	Suppose that there exists a number $U>0$ such that
	\beq{e2.5}\sum\limits_{i=1}^d \big[ b_i(x) -\delta(x_i-U)\big]f_i<0 \quad \text{for}\quad |x|> U.\eeq Then there exists $x^*\in [0, U]^d$ such that
	$$V(x)= V(x^*) + f\cdot (x-x^*)\quad \text{for}\quad x\in \overline{S}\setminus [0, U]^d.$$
Moreover,
	$$
	V(x)=V^U(x) \quad \text{for}\quad x\in [0, U]^d.
	$$
\end{prop}

It should be noted that the inequality \eqref{e2.5} is easily verified and holds in most ecological systems. In dimension $d=1$, \eqref{e2.5} becomes $$\big[b(x)-\delta(x-U)\big]f <0  \quad \text{for}\quad  x>U.$$ If $b(x)\le 0$ for sufficiently large $x$ we can therefore find the required $U$. Similarly, if the dimension is at least $d\geq 2$ and $$\sum\limits_{i=1}^d \big[ b_i(x) +\delta |x|\big]f_i<0$$
for sufficiently large $|x|$ then \eqref{e2.5} holds.

\begin{rem} If the value function $V$ is continuous, one can apply the dynamic programing principle to show that the value function is a viscosity solution of the quasi-variational inequalities
\beq{e2.hjb}
\max\limits_{x\in S}\Big\{(\L-\delta) \phi(x) +\max\limits_{\xi\in [0, \lambda]}\big[ \xi\cdot (\nabla \phi-g)\(x\)\big], f-\nabla \phi(x)\Big\}=0.
\eeq
However, in this setting it is hard to establish the continuity of the value function. Alternatively, one can try to prove a singular control version of the weak dynamic programing principle developed by \cite{BT11} and then characterize the value function as a discontinuous viscosity solution of \eqref{e2.hjb}.
Because of the technical nature of these problems, we leave them as open questions.
\end{rem}

In order to gain important qualitative information about the optimal harvesting-seeding strategies and the value function we develop a numerical approximation scheme. We construct a controlled Markov chain that approximates the controlled diffusion $X\cd$ from \eqref{e2.2.2}.
Assume without loss of generality that $U$
is an integer multiple of $h$ and define
$$S_{h}: = \{x=(k_1 h, \dots, k_d h)'\in \rr^d: k_i\in\mathbb{Z}_{\geq 0}\}\cap [0, U]^d.$$
The set $S_h$ is a lattice where the components are positive integer multiples of $h$.
We will approximate $X$ by $\{X^h_n: n\in\mathbb{Z}_{\geq 0}\}$ -- a discrete-time controlled Markov chain with state space $S_{h}$.

At any time  step $n$, the control is first specified by the choice of an action: harvesting or seeding. We use $\pi^h_n$ to denote the action at step $n$:
\begin{itemize}
  \item $\pi^h_n=i$  if there is harvesting of species $i$
  \item $\pi^h_n=0$ if there is seeding.
\end{itemize}
In the case of
a seeding,  the magnitude of the seeding component must be specified. We denote this by $C^h_n$. The space of possible controls is therefore $\mathcal{U} = \{0, 1,\dots, d\}\times [0, \lambda]$.
Let $u^h=\{u^h_n\}_n$ with $u^h_n= (\pi^h_n, C^h_n), n\in\mathbb{Z}_{\geq 0}$ be a sequence of controls.

   We denote by $p^h\(x, y|u=(\pi, c)\)$ the transition probability from state $x$ to another state $y $ under the control $u=(\pi, c)$.
We will choose $p^h\(x, y|u=(\pi, c)\)$ together with interpolation intervals $\Delta t^h (x, u)$ so that the piecewise constant interpolation of  $\{X^h_n\}$ approximates $X\cd$ well for small $h$. A control sequence $u^h$ is called admissible if under this control sequence, $\{X^h_n\}$ is a Markov chain with state space $S_h$. The class of all admissible control sequences $u^h$ with initial state $x$ will be denoted by
$\mathcal{A}^h_{x}$.

For $x\in S_h$ and $u^h\in \mathcal{A}^h_{x}$, the performance function for the controlled Markov chain is defined as
\beq{e2.4.4}
J^h(x, u^h) =  \E\bigg[\sum_{m=1}^{\infty} e^{-\delta t_m^h} f\cdot \Delta Y_{m}^{h}- \sum_{m=1}^{\infty} e^{-\delta t_m^h} g(X^h_m) \cdot C^h_m \Delta t^h(X^h_m, u^h_m)\bigg],
\eeq
where $\Delta Y_m^h$ is the harvesting amount at step $m$.
The value function of the controlled Markov chain is
\beq{e2.4.5}
V^h(x) = \sup\limits_{u^h\in \mathcal{A}^h_{x}} J^h (x, u^h).
\eeq
The similarity between \eqref{e2.4.4} and \eqref{e2.2.4}  suggests that the
optimal values $V^h(x)$ and $V^U(x)$ will be close for small $h$, and this will turn out
to be the case.
The following theorem tells us that the value function of the Markov chain approximations converges to the correct value function as $h$ goes to zero.
\begin{thm} \label{thm:conv1}
	Suppose Assumptions \ref{a:1},\ref{a:2}, and \eqref{e2.5} hold.  Then for any $x\in [0, U]^d$,
	 $V^h(x)\to V(x)$ as $h\to 0$.
	As a result, for sufficiently small $h$, a near-optimal harvesting-seeding strategy of the controlled Markov chain $X^h_n$ is also a near-optimal harvesting-seeding policy of the continuous-time system $X$ given by \eqref{e2.2.2}.
\end{thm}

\subsection{Bounded harvesting and seeding rates}  \label{sec:for3}
\

In most practical situations it is impossible to have an infinite harvesting rate (\cite{Alvarez98}). In this subsection we look at the case when both the harvesting and seeding rates are bounded. This means for all $t\geq 0$
\begin{equation*}
\begin{aligned}
dY(t)&=R(t)dt,\\
dZ(t)&=C(t)dt
\end{aligned}
\end{equation*}
with
\begin{equation*}
\begin{aligned}
0 &\leq R(t)\le \mu,\\
0 &\leq C(t)\le \lambda
\end{aligned}
\end{equation*}
where $\lambda$ is the maximum seeding rate and $\mu$ is the maximum harvesting rate.
We will identify $(Y, Z)$ with $(R, C)$. The dynamics of the population system with harvesting and seeding is given by
 \beq{e3.2.2}X(t)=x+\int\limits_0^t \Big[b(X(s)) + C(s)-R(s)\Big] ds +  \int\limits_0^t  \sigma(X(s)) dw(s),\eeq
and the performance function is
 \beq{e3.2.4}J(x, R, C):= \E_{x}\bigg[\int\limits_0^{\infty} e^{-\delta s} f(X(s)) \cdot R(s) ds-\int\limits_0^{\infty} e^{-\delta s} g\( X(s) \) \cdot C(s)ds\bigg].
 \eeq
It would be never optimal if both $R(t)$ and $C(t)$ were positive for all $t$ on a set of positive measure. We can therefore suppose that $R(t)=0$ whenever $C(t)>0$ and $C(t)=0$ whenever $R(t)>0$. Equation \eqref{e3.2.2} becomes
 \beq{e3.2.2a}X(t)=x+\int\limits_0^t \Big[b(X(s)) + Q(s)\Big] ds +  \int\limits_0^t  \sigma(X(s)) dw(s),\eeq
 where $Q(s)=C(s)-R(s)=\(Q_1(s), \dots, Q_d(s)\)'$. Note that $-\mu\le Q(s)\le \lambda$.
 The performance function \eqref{e3.2.4} becomes
 \beq{e3.2.2x} J(x, Q):= \E_{x}\int\limits_0^{\infty} e^{-\delta s}\Big[ Q^-(s) \cdot f\(X(s)\) - Q^+(s)\cdot g\( X(s) \)\Big]ds.
 \eeq

In this setting we can characterize the value function as a viscosity solution of the associated quasi-variational inequalities \beq{e3.3.17}
(\L-\delta) \phi(x) + \max\limits_{\xi\in [-\mu, \lambda]}\Big[\xi^-\cdot \big(f-\nabla \phi)\(x\) - \xi^+ \cdot (g-\nabla \phi)\(x\)\Big]=0, \quad x\in S.
\eeq
We will make use of standard viscosity solution approach (\cite{Ky18}).

	\begin{thm}
Suppose Assumption \ref{a:1} is satisfies. Then the following properties hold.

{\rm (a)} The value function
$V$ is finite and continuous on $\overline S$.

{\rm (b)} The value function
$V$ is a viscosity subsolution of  \eqref{e3.3.17}; that is,
	for
	any $x^0\in S$ and any function $\phi\in C^2(S)$ satisfying
	$$(V-\phi)(x)\ge (V-\phi)(x^0)=0,$$
	for all $x$ in a neighborhood of $x^0$, we have
	\beq{e3.3.17}(\L-\delta) \phi(x^0) + \max\limits_{\xi\in [-\lambda, \mu]}\Big[\xi^-\cdot \big(f-\nabla \phi)\(x^0\) - \xi^+ \cdot (g-\nabla \phi)\(x^0\)\Big]\le 0.\eeq

{\rm (c)} The value function
	$V$ is a viscosity supersolution of
	\eqref{e3.3.17}; that is,
	 for
	 any $x^0\in S$ and any function $\ph\in C^2(S)$ satisfying
	 \beq{e3.3.27j}(V-\ph)(x)\le (V-\ph)(x^0)=0,\eeq for all $x$ in a neighborhood of $x^0$, we have
	 \beq{e3.3.27k}(\L-\delta) \ph(x^0) + \max\limits_{\xi\in [-\lambda, \mu]}\Big[\xi^-\cdot \big(f-\nabla \ph)\(x^0\) - \xi^+ \cdot (g-\nabla \ph)\(x^0\)\Big]\ge 0.\eeq
	
	 \noindent {\rm(d)} The value function $V$ is a viscosity solution of  \eqref{e3.3.17}.	
\end{thm}

 We develop numerical approximation methods for computing the value function in this setting. We will need to approximate the control problem by an analogous control problem with a bounded state space $[0, U]^d$. This is done by replacing the original dynamical system with one which evolves exactly as before in the interior of some compact domain but is instantaneously reflected back when the controlled process is about to exit the domain. The modified constrained dynamics   of the $d$ species \eqref{e3.2.2} now becomes
\beq{e3.2.2.2}X(t)=x+\int\limits_0^t \Big[b(X(s)) +Q(s)\Big] ds +  \int\limits_0^t  \sigma(X(s)) dw(s)-dN(t),\eeq
where  $N(t)=\big(N_1(t), \dots, N_d(t)\big)'$ is the reflection component, which is a componentwise nondecreasing, right continuous, $\{\mathcal{F}(t)\}$-adapted process satisfying
$$\int_0^\infty I_{\{X_i(t)< U\}}d N_i(t)=0, \quad i=1, 2, \dots, d.$$
We refer to \cite{B98, F16, LS84, Kushner92} for  reflected diffusions and Skorokhod stochastic differential equations.
The corresponding value function is denoted by $V^U$.

Let $h>0$ be a discretization parameter. We proceed to construct a controlled Markov chain in discrete time to approximate the controlled diffusion $X\cd$.  Assume without loss of generality that $U$
is an integer multiple of $h$. Due to the reflection terms in the dynamics of the controlled process
we consider a slightly enlarged state space $$S_{h+}=\{x=(k_1 h, \dots, k_d h)'\in \rr^d: k_i\in\mathbb{Z}_{\geq 0}\}\cap [0, U+h]^d.$$

Let $\{X^h_n: n\in\mathbb{Z}_{\geq 0}\}$
be a discrete-time controlled Markov chain with state space $S_{h+}$.
At any time  step $n$, the control is first specified by the choice of an action: controlled diffusion or reflection. We use $\pi^h_n$ to denote the action at step $n$
\begin{itemize}
  \item $\pi^h_n=0$  if the $n$th step is a controlled diffusion step
  \item $\pi^h_0=i$ if  the $n$th step is a reflection step on species $i$.
\end{itemize}
In the case of
a controlled diffusion step,  the magnitude of the harvesting-seeding component, which is $Q^h_n$, must also be
specified. The space of controls in this setting is given by $\mathcal{U}=\{0, 1, \dots, d\}\times [-\mu, \lambda]$.
Let $u^h=\{u^h_n\}_n$  defined by $u^h_n= (\pi^h_n, Q^h_n)$ for $n\in\mathbb{Z}_{\geq 0}$ be a sequence of controls.

  We denote by $p^h\(x, y|u=(\pi, q)\)$ the transition probability from state $x$ to another state $y $ under the control $u=(\pi, q)$.
  We will choose $p^h\(x, y|u=(\pi, q)\)$ together with interpolation intervals $\Delta t^h (x, u)$ so that the piecewise constant interpolation of  $\{X^h_n\}$ approximates $X\cd$ well for small $h$.
  A control sequence $u^h$ is admissible if under this policy, $\{X^h_n\}$ is a Markov chain with state space $S_{h+}$.

For $x\in S_{h+}$ and $u^h=(\pi^h, Q^h)\in \mathcal{A}^h_{x}$, the performance function for the controlled Markov chain is defined as
\beq{e3.4.4}
J^h(x, u^h) =  \E\sum_{m=1}^{\infty} e^{-\delta t_m^h} \Big[(Q^h_m)^+\cdot f(X^h_m)  - (Q^h_m)^- \cdot g(X^h_m)  \Big]\Delta t^h(X^h_m, u^h_m).
\eeq
The value function of the controlled Markov chain is
\beq{e3.4.5}
V^h(x) = \sup\limits_{u^h\in \mathcal{A}^h_{x}} J^h (x, u^h).
\eeq

The convergence theorem for this scenario is given below.

\begin{thm}\label{thm:conv2}
	Suppose Assumptions \ref{a:1} and \ref{a:2} hold. Then for any $x\in [0, U]^d$,
	$V^h(x)\to V^U(x)$ as $h\to 0$.
	For sufficiently small $h$, a near-optimal harvesting-seeding strategy of the controlled Markov chain is also a near-optimal harvesting-seeding policy of the continuous-time system \eqref{e3.2.2.2}.
\end{thm}

\subsection{Unbounded seeding and bounded harvesting rates}  \label{sec:for2}
\

If we assume the seeding can be unbounded and the harvesting is bounded we have $$dY(t)=R(t)dt$$ with $0\le R(t)\le \mu, t\geq 0$. We identify $(Y, Z)$ with $(R, Z)$. Suppose that the seeding functions $g_i\cd$ are constant. The dynamics of the ecosystem is given by
\beq{e4.2.2}X(t)=x+\int\limits_0^t \big[b(X(s)) - R(s)\big]ds +  \int\limits_0^t  \sigma(X(s)) dw(s) + Z(t),\eeq
and the performance function is
\beq{e4.2.4}
J(x, R, Z):= \E_{x}\bigg[\int\limits_0^{\infty} e^{-\delta s} f (X(s))\cdot R(s)ds-\int\limits_0^{\infty} e^{-\delta s} g \cdot dZ(s)\bigg].
\eeq

Similar to the preceding case,  in order to develop numerical methods for computing the value function, we will need to approximate the problem by a related control problem with a bounded state space $[0, U]^d$.  The modified constrained dynamics of the $d$ species \eqref{e4.2.2} becomes
\beq{e4.2.2}dX(t)=\big[b(X(t))-R(t)\big] dt +    \sigma(X(t)) dw(t) + dZ(t)-dN(t), \quad X(0-)=x,\eeq
where  $N(t)=\big(N_1(t), \dots, N_d(t)\big)'$ is the reflection component, which is a componentwise nondecreasing, right continuous, $\{\mathcal{F}(t)\}$-adapted process satisfying
$$\int_0^\infty I_{\{X_i(t)< U\}}d N_i(t)=0, \quad i=1, 2, \dots, d.$$
As usual, the corresponding value function is denoted by $V^U$.

Let $h>0$ be a discretization parameter. We proceed to construct a controlled Markov chain in discrete time to approximate the controlled diffusions. Assume without loss of generality that $U$
is an integer multiple of $h$. We look at the enlarged state space  $$S_{h+}=\{x=(k_1 h, \dots, k_d h)'\in \rr^d: k_i\in\mathbb{Z}_{\geq 0}\}\cap [0, U+h]^d.$$
Let $\{X^h_n: n\in\mathbb{Z}_{\geq 0}\}$
be a discrete-time controlled Markov chain with state space $S_{h+}$.
At any time  step $n$, the control is first specified by the choice of an action: controlled diffusion, seeding, or reflection. We use $\pi^h_n$ to denote the action at step $n$
\begin{itemize}
  \item $\pi^h_n=0$  if the $n$th step is a controlled diffusion step
  \item  $\pi^h_n=-i$  if the $n$th step is a seeding step on species $i$
  \item $\pi^h_0=i$ if  the $n$th step is a reflection step on species $i$.
\end{itemize}
In the case of
a controlled diffusion step,  the magnitude of the harvesting, which is $R^h_n$, must also be
specified. The space of controls will be $\mathcal{U}=\{0, \pm 1, \pm 2,\dots, \pm d\}\times [0, \mu]$.

  We denote by $p^h\(x, y|u=(\pi, r)\)$ the transition probability from state $x$ to another state $y $ under the control $u=(\pi, r)$. We will choose $p^h\(x, y|u=(\pi, r)\)$ together with interpolation intervals $\Delta t^h (x, u)$ so that the piecewise constant interpolation of  $\{X^h_n\}$ approximates $X\cd$ well for small $h$.
  Formally, a control sequence $u^h$ is admissible if under this policy, $\{X^h_n\}$ is a Markov chain with state space $S_{h+}$.

For $x\in S_{h+}$ and $u^h=(\pi^h, R^h)\in \mathcal{A}^h_{x}$, the performance function for the controlled Markov chain is defined as
\beq{e4.4.4}
J^h(x, u^h) =  \E\bigg[\sum_{m=1}^{\infty} e^{-\delta t_m^h} f(X^h_m)\cdot  R^h_m  \Delta t^h(X^h_m, u^h_m) - \sum_{m=1}^{\infty} e^{-\delta t_m^h} g\cdot \Delta Z^h_m \bigg],
\eeq
where $\Delta Z^h_m$ is the seeding amount at step $m$.
The value function of the controlled Markov chain is
\beq{e4.4.5}
V^h(x) = \sup\limits_{u^h\in \mathcal{A}^h_{x}} J^h (x, u^h).
\eeq
We get the following convergence result.
\begin{thm}\label{thm:conv3}
	Suppose Assumptions \ref{a:1} and \ref{a:2} hold. Then for any $x\in [0, U]^d$,
	$V^h(x)\to V^U(x)$ as $h\to 0$.
	For sufficiently small $h$, a near-optimal harvesting-seeding strategy of the controlled Markov chain is also a near-optimal harvesting-seeding policy of the continuous-time system \eqref{e4.2.2}.
\end{thm}

\section{Numerical Examples}\label{sec:fur}
In this section we explore various relevant scenarios and see how our numerical approximation scheme can provide fundamental insights into the optimal harvesting and seeding of populations.
\subsection{Single species system.}
We first look at a system which has one single species that is driven by a logistic stochastic differential equation (\cite{AS98, EHS15, HNUW18}). The dynamics that includes harvesting and seeding will be given by
\beq{e5.6.1}
d X(t) = X(t)\big(b_1-b_2X(t) \big)dt + \sigma X(t)dw(t)-dY(t) + d Z(t).
\eeq
Here $b_1$ is the per-capita growth rate, $b_2$ is the per-capita competition rate and $\sigma^2>0$ is the per-capita variance of the environmental fluctuations.

Let $\lambda$ and $\mu$ be the maximum seeding and harvesting rates, so that $0\le\lambda\le \infty$ and $0\le \mu\le \infty$.

We first look at the case $\lambda<\infty$ and $\mu=\infty$.
For an admissible strategy $(Y, C)$ we have \beq{e5.6.2}J(x, Y, C)=\E\left[\int_0^\infty e^{-\delta s} fdY(s)-\int_0^\infty e^{-\delta s} gC(s)ds\right].\eeq
Based on the algorithm constructed above and in Appendix \ref{sec:alg}, we carry out the computation by using the methods in \cite[Chapter 6]{Kushner92}. At each level $x=h,2h, \dots, U$ and $n$th iteration, denote by $u(x, n)=\big(\pi(x, n), c(x, n)
\big)$ the control one chooses, where $\pi(x, n)=1$ if there is harvesting, $\pi(x, n)=0$ if there is seeding.
We initially let $\pi(x, 0)=1$ and $c(x, 0)=0$ for all $x$ and we try to find better harvesting-seeding strategies. The initial harvesting-seeding policy is
$(Y_0, C_0)$, the policy which drives the system to extinction
immediately and has no seeding. Note that
$$J(x,  Y_0, C_0)=fx$$ for all $x$
and $$V_0^h(x)=fx, \quad x = 0, h, 2h, \dots, U.
$$

 \begin{figure}[h!tb]
 	\begin{center}
 		\includegraphics[height=2in,width=6in]{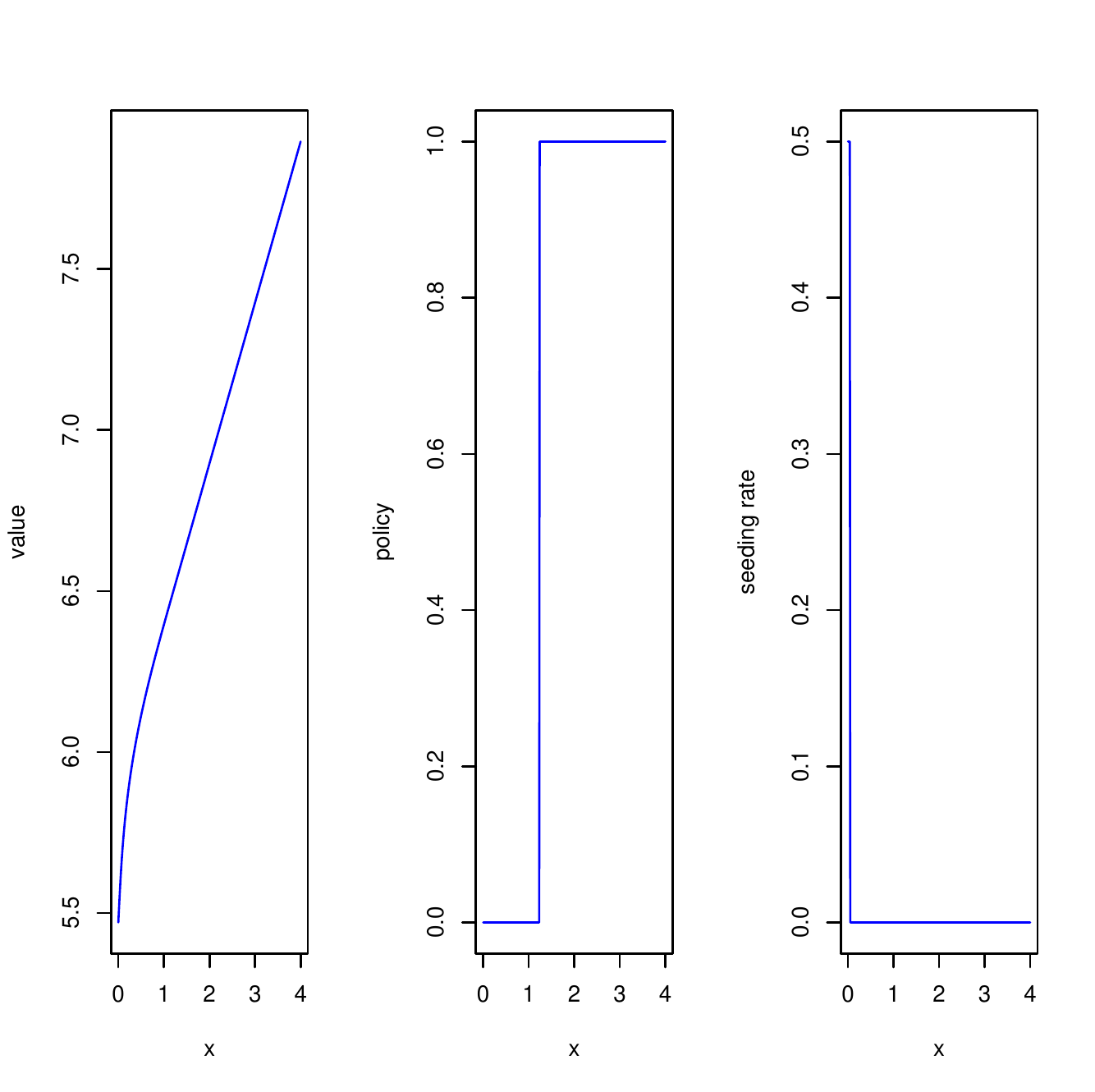}
 		\caption{Value function (left), optimal policy (middle, $1$:  harvesting,  $0$:
seeding), and optimal seeding rate (right) when $\lambda =0.5, \mu =\infty$} \label{fig1}\end{center}
 \end{figure}

 \begin{figure}[h!tb]
 	\begin{center}
 		\includegraphics[height=2in,width=6in]{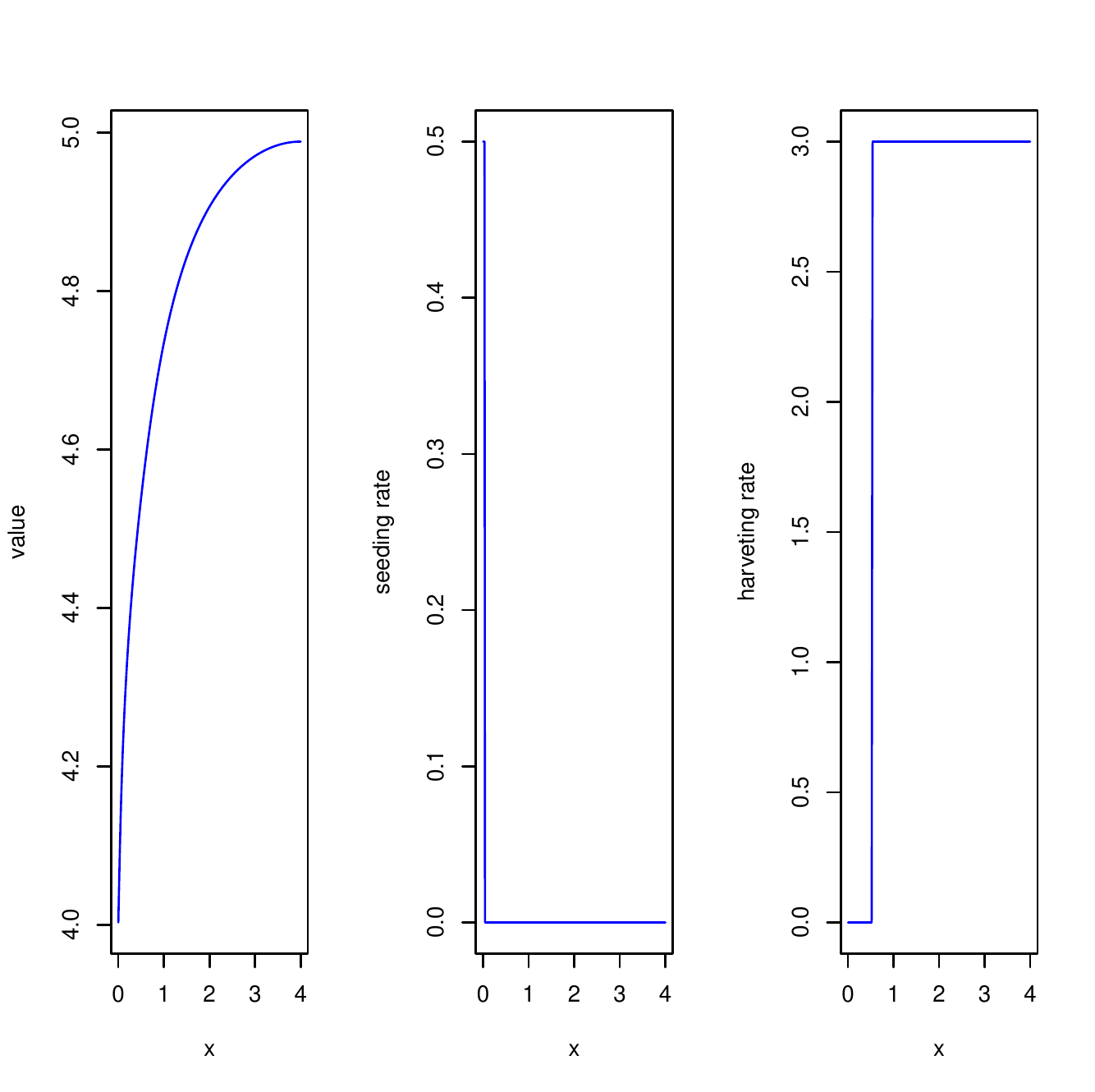}
 		\caption{Value function (left), optimal harvesting rate (middle), and optimal seeding rate (right) when $\lambda =0.5, \mu =3$} \label{fig2}\end{center}
 \end{figure}

We find an improved value $V^h_{n+1}(x)$
and record the updating optimal control by
$$u(x, n)=\argmax \left\{(i, c): V^{h, i, c}_{n+1} (x)\right\} ,\quad
V^h_{n+1}(x)= V^{h, u(x,  n)}_{n+1} (x),$$
where
\bea
\aad V^{h, 1, c}_{n+1}(x)=V^h_n(x-h) +  fh,\\
\aad V^{h, 0, c}_{n+1} (x) =e^{-\delta \Delta t^h(x, 0, c) }\Big[ V^{h}_n (x+h) p^h \big(x, x+h| (0, c) \big)\\
\aad \qquad\quad  \qquad +V^{h}_n  (x-h) p^h \big( x, x-h| (0, c) \big) +  V^{h}_n (x
) p^h \big(x, x | (0, c) \big) - gc  \Delta t^h(x, 0, c)
\Big].
\eea
The numerical algorithm alternates between policy iterations and value iterations
until the
increment
$V^h_{n+1}\cd-V^h_n\cd$
reaches
some tolerance level. The error tolerance is chosen to be $10^{-7}$. We pick the parameters
\begin{equation*}
 b_1=3,\quad  b_2=2, \quad \sigma=2, \quad \delta=0.05,
\quad  f(x)\equiv0.5,\quad g(x)\equiv 2.5,  \quad U=4.
\end{equation*}
Note that $b_1-\frac{\sigma^2}{2}>0$, so that the species survives in the absence of harvesting and seeding.

For the first numerical experiment, take
   $\lambda=0.5$ and $\mu=\infty$.
Figure \ref{fig1} shows the value function $V(x)$ as a function of the population size $x$, gives the optimal harvesting-seeding policies, and also provides the optimal seeding rates.
It can be seen from Figure \ref{fig1}  that  the optimal policy is a barrier strategy. There are thresholds $L_1$ and $L_2$, where $L_1=0.04$ and $L_2=1.25$ such that $[0, L_1]$ is the seeding region (the seeding rate is positive and maximal), $(L_1, L_2 )$ is the no-control region (no seeding and no harvesting), and $[L_2, U]$ is the harvesting region.

\begin{figure}[h!tb]
	\begin{center}
		\includegraphics[height=2in,width=4in]{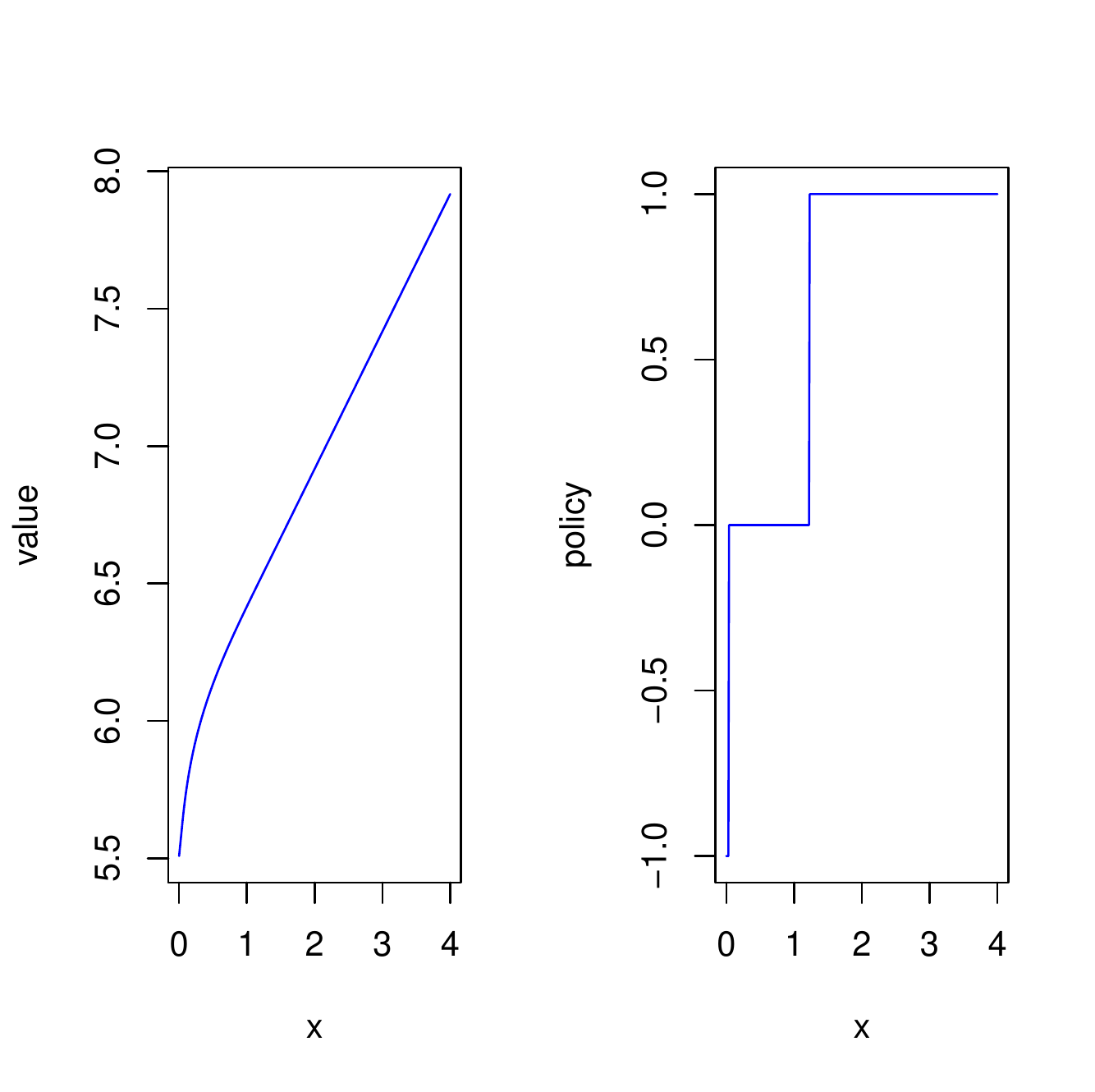}
		\caption{Value function (left) and optimal policy (right, 1: harvesting, 0: no control, -1: seeding) when $\lambda=\infty, \mu=\infty$} \label{fig3}\end{center}
\end{figure}

\begin{figure}[h!tb]
	\begin{center}
		\includegraphics[height=2in,width=6in]{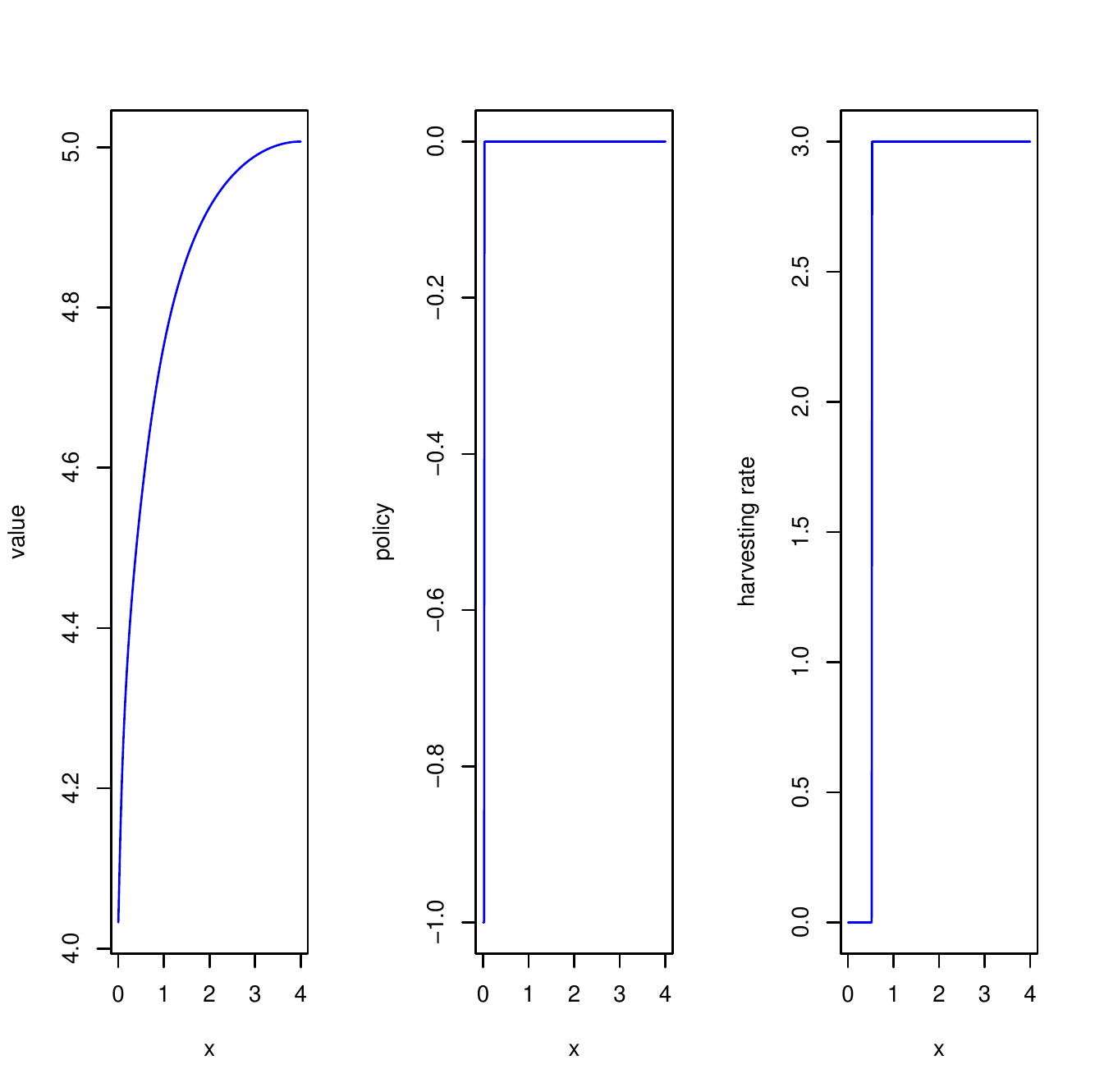}
		\caption{Value function (left), optimal policy (middle, -1: seeding, 0: harvesting), and optimal harvesting rate (right) when $\lambda=\infty, \mu=3$} \label{fig4}\end{center}
\end{figure}

Next, let $\lambda=0.5$ and $\mu =3$ and keep the other parameters as above. The numerical results are shown in Figure \ref{fig2}.  Similar to the preceding scenario, the optimal policy is a barrier strategy. In particular, we have $L_1=0.03$ for the seeding threshold and $L_2=0.54$ for the harvesting threshold. Note that this implies that one needs to harvest sooner if the harvest rate is bounded.
Moreover, it turns out that it is always optimal to harvest and seed with the maximal possible rates.

Figure \ref{fig3} shows the numerical experiment when both harvesting and seeding rates are infinite, i.e., $\lambda=\infty$ and $\mu =\infty$. For the policies in
Figure \ref{fig3}, $1$ denotes harvesting,  $-1$ denotes
seeding, and $0$ denotes no action. In this scenario, $L_1=0.03, L_2=1.23$. Figure \ref{fig4} looks at unbounded seeding $\lambda=\infty$ and bounded harvesting $\mu =3$. In Figure \ref
{fig4}, since the harvesting rate is bounded, we see that it is optimal to start harvesting at the lower threshold $L_2=0.54$ (compared to $L_2=1.23$) with the maximal rate. Moreover, $L_1=0.03$.

\begin{figure}[h!tb]
	\begin{center}
		\includegraphics[height=2.7in,width=4.7in]{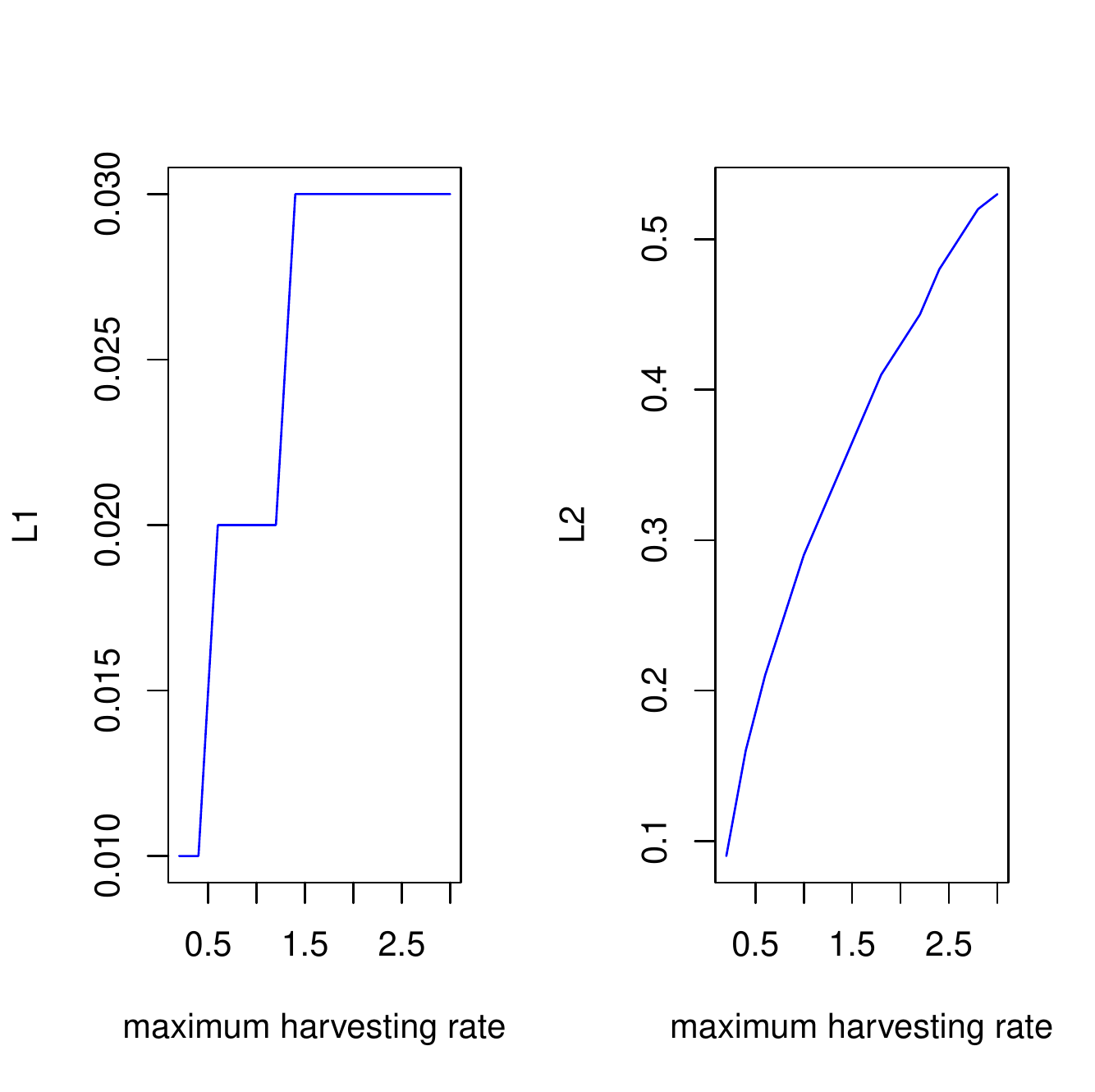}
		\caption{The thresholds $L_1$ (left) and $L_2$ (right) for $\lambda=\infty$ and $\mu\in [0, 3]$} \label{fig4A}\end{center}
\end{figure}

\begin{figure}[h!tb]
	\begin{center}
		\includegraphics[height=2.7in,width=4.7in]{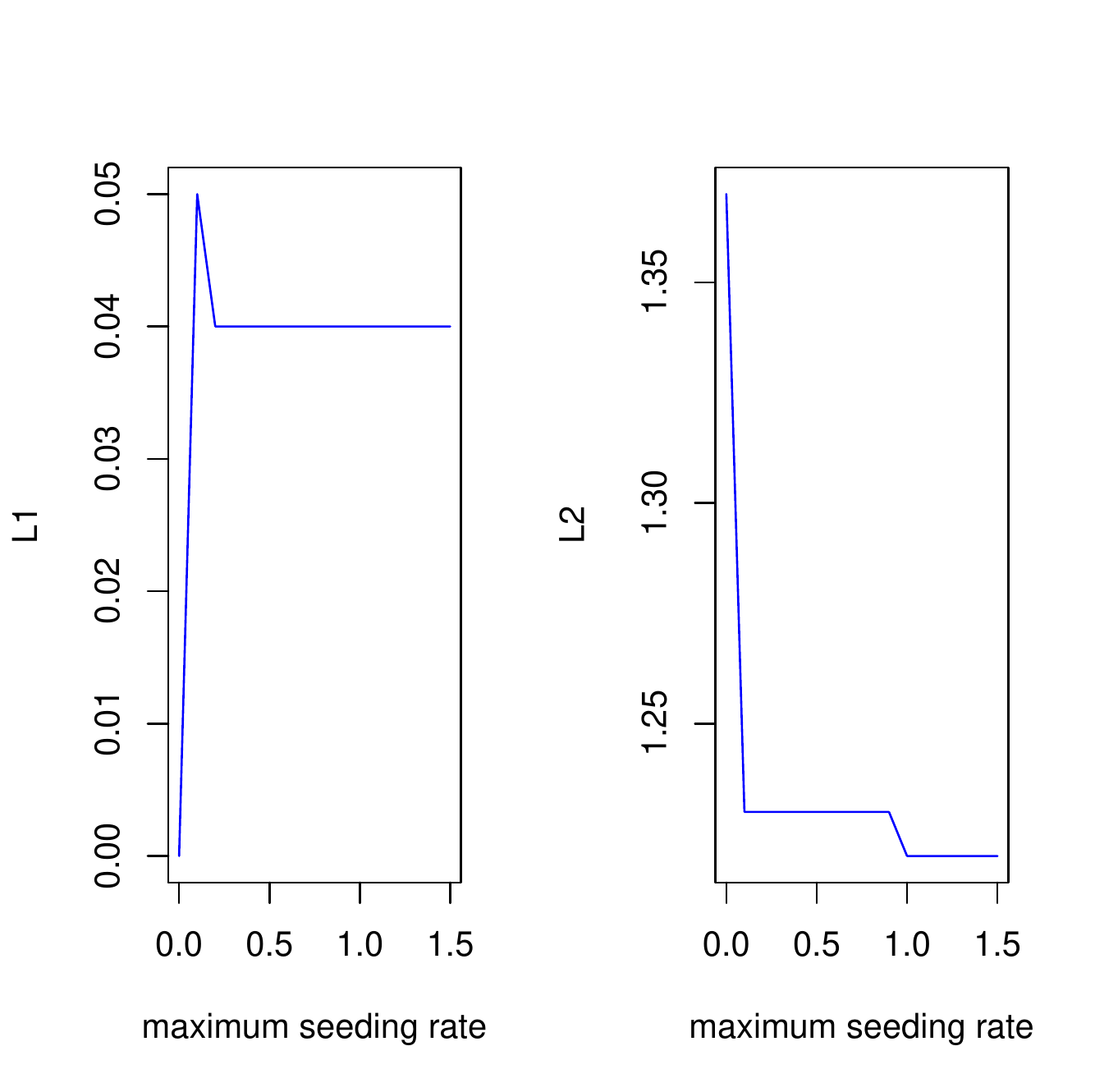}
		\caption{The thresholds $L_1$ (left) and $L_2$ (right) for $\lambda\in [0, 1.5]$ and $\mu=\infty$} \label{fig4B}\end{center}
\end{figure}

\begin{figure}[h!tb]
	\begin{center}
		\includegraphics[height=2.7in,width=4.7in]{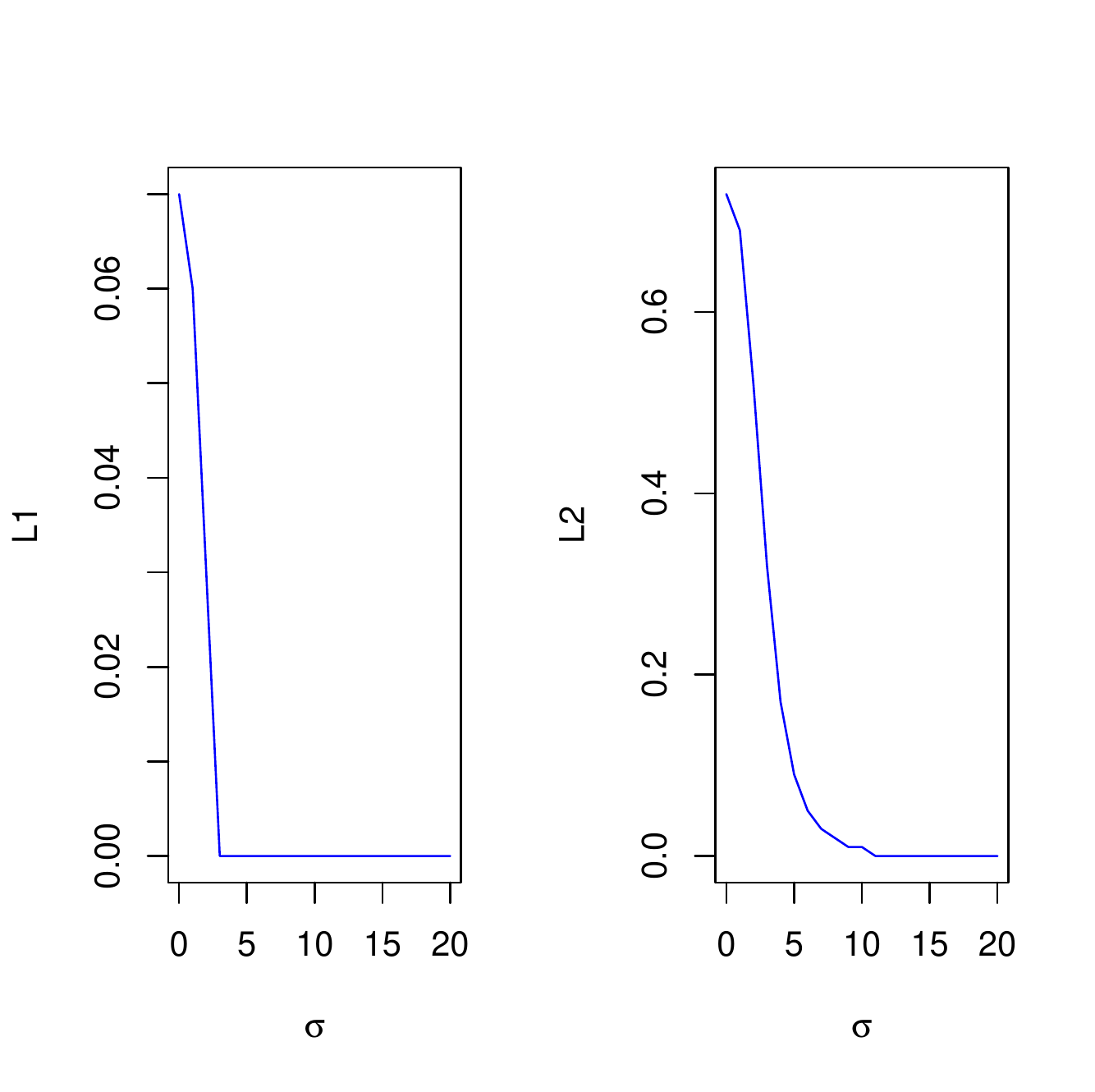}
		\caption{ The thresholds $L_1$ (left) and $L_2$ (right) for $\lambda=0.5$ and $\mu=3$ and $\sigma\in [1, 20]$} \label{fig4C}\end{center}
\end{figure}

\textit{\textbf{Biological interpretation}: In general if there is just one species, the optimal seeding-harvesting strategy will be of threshold type. There is a lower threshold $L_1>0$ and an upper threshold $L_2>L_1$. If the population size is below $L_1$ we seed at the maximal rate $\lambda\leq \infty$. In particular, if the seeding rate is infinite this means that the population gets to a level above $L_1$ immediately at $t=0$ and then never goes below $L_1$ - the seeding happens infinitely fast at $L_1$ so that the process reflects from $L_1$ into $(L_1, L_2)$. When the population size is between $L_1$ and $L_2$ we do not seed nor do we harvest. Once we are above the threshold $L_2$ we harvest at the maximal rate $\mu$. If the harvest rate is infinite, the population gets to a level below $L_2$ immediately at $t=0$ and then never goes above it again - the harvesting happens infinitely fast at $L_2$ so that the process reflects from $L_2$ into $(L_1,L_2)$. If both harvesting and seeding rates are infinite the process immediately enters $(L_1, L_2)$ at $t=0$ and stays there forever. If one rate is finite, the corresponding point ($L_1$ if finite seeding and $L_2$ is finite harvesting) wont be reflecting and the population can pass that threshold at a time $t>0$.
The thresholds $L_1, L_2$ depend on the seeding and harvesting rates as well as on the variance of the environmental fluctuations. Figure \ref{fig4A} provides the graph of the thresholds $L_1$ and $L_2$ as functions of the harvesting rate $\mu\in [0, 3]$ when $\lambda = \infty$. Both $L_1$ and $L_2$ increase with $\mu$ - as the harvesting rate increases we can wait longer until we start seeding or harvesting. Figure \ref{fig4B} provides the graph of the thresholds $L_1$ and $L_2$ as functions of $\lambda\in [0, 1.5]$ when $\mu = \infty$ - the seeding threshold $L_1$ first increases linearly after which it decreases and then becomes constant. When the seeding rate is very close to zero, it is hard to keep the species away from extinction and the seeding has to happen for a longer time (higher $L_1$). As the seeding rate increases, extinction becomes less likely and the threshold $L_1$ decreases. The harvesting threshold $L_2$ decreases with the seeding rate - a higher seeding rate makes extinction less likely and one can start harvesting at lower population levels. Figure \ref{fig4C} provides the graph of the thresholds $L_1$ and $L_2$ as functions of $\sigma\in [0, 20]$ when $\lambda = 0.5$ and $\mu=3$. The thresholds $L_1$ and $L_2$ are non-increasing functions of $\sigma$. It can be seen that when the noise intensity $\sigma$ is large, and the species goes extinct fast, it becomes optimal to harvest at the maximal possible rate at any population level and it is never optimal to seed anymore. This observation fits with the results by \cite{Alvarez98, Ky17} for harvesting problems without seeding.
}

\subsection{Two-species ecosystems}

\begin{exm}{\rm
		Consider two species competing according to the following stochastic Lotka-Volterra system
		\beq{}\barray
		\aad d X_1(t)  = X_1(t)\Big(b_1-a_{11}X_1(t) - a_{12}X_2(t)   \Big)dt + \sg_1X_1(t)dw_1(t)-dY_1(t)+dZ_1(t)\\
		\aad d X_2(t)  = X_2(t)\Big(b_2-a_{21}X_1(t) - a_{22}X_2(t)   \Big)dt +  \sg_2X_2(t)dw_2(t)-dY_2(t) + dZ_2(t),
		\earray
		\eeq

Here $b_1, b_2$ are the per-capita growth rates, $a_{12}, a_{21}$ the per-capita interspecific competition rates, $a_{11}, a_{22}$ the per-capita intraspecific competition rates and $\sigma_1^2, \sigma_2^2>0$ the per-capita variances of the environmental fluctuations. If there is no seeding or harvesting the dynamics of the above ecosystem has been studied extensively in the literature (\cite{TG80, KO81, SBA11, EHS15, HN18}).
Let $\lambda=(\lambda_1, \lambda_2)', \mu=(\mu_1, \mu_2)'$ be the maximum seeding rates and the maximum harvesting rates for the two species.
We set
\beq{}
\barray
\aad \delta=0.05, \quad f_1(x)\equiv1,\quad  f_2(x)\equiv 1.5, \quad g_1(x)\equiv4, \quad g_2(x)\equiv3,\\
\aad b_1=3,  a_{11}=2, a_{12}=1.5,  \sg_1=3, b_2=2,  a_{21}=2, a_{22}=2,  \sg_2=4,	U =4.
\earray
\eeq
Since the stochastic growth rates of the species are negative, $b_1-\sigma_1^2/2, b_2 -\sigma_2^2/2<0$ both species go extinct in the absence of seeding.
 \begin{figure}
 	\centering
 	\subfloat{{\includegraphics[width=6.5cm]{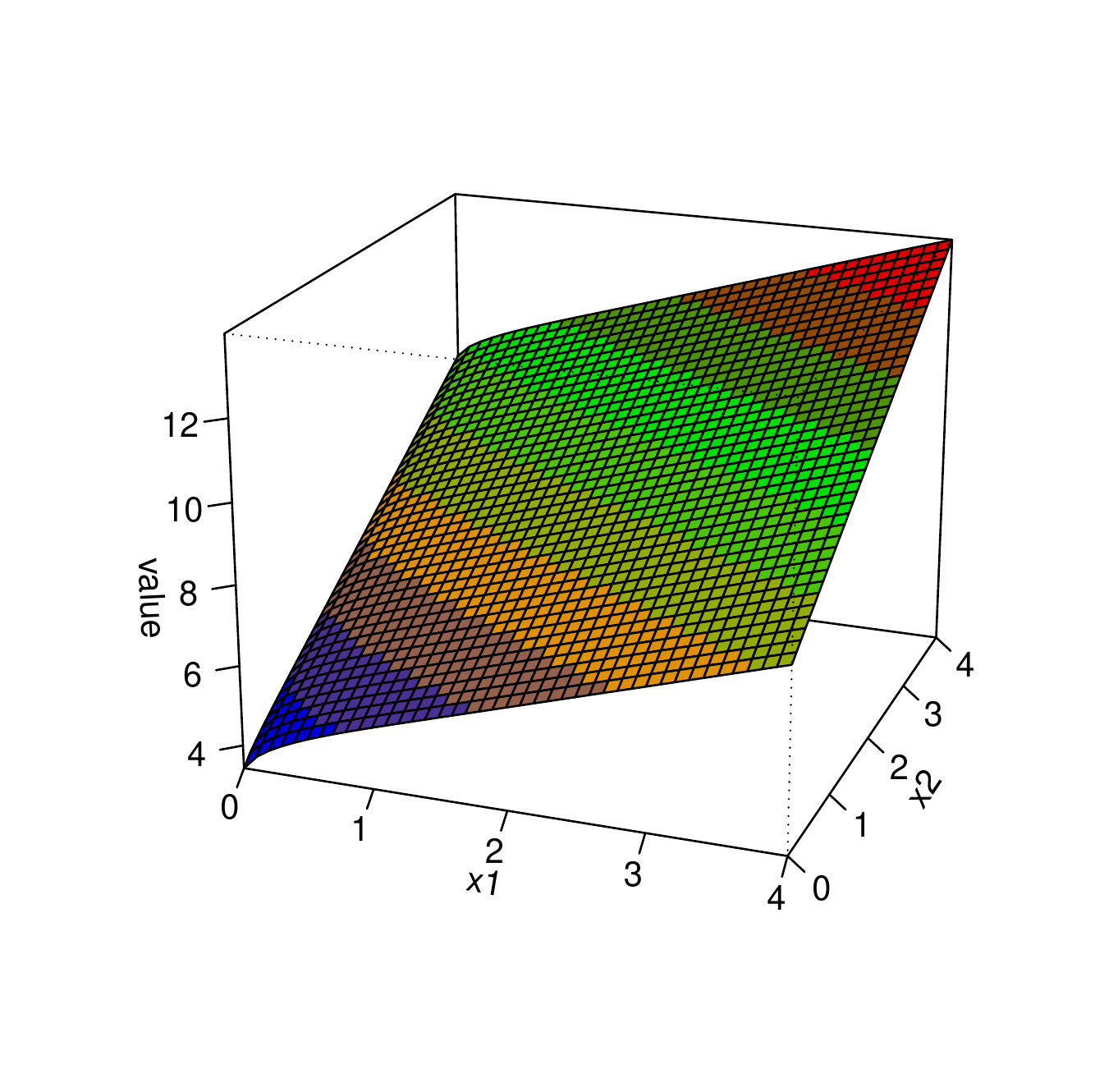} }}%
 	\quad
 	\subfloat{{\includegraphics[width=6.5cm]{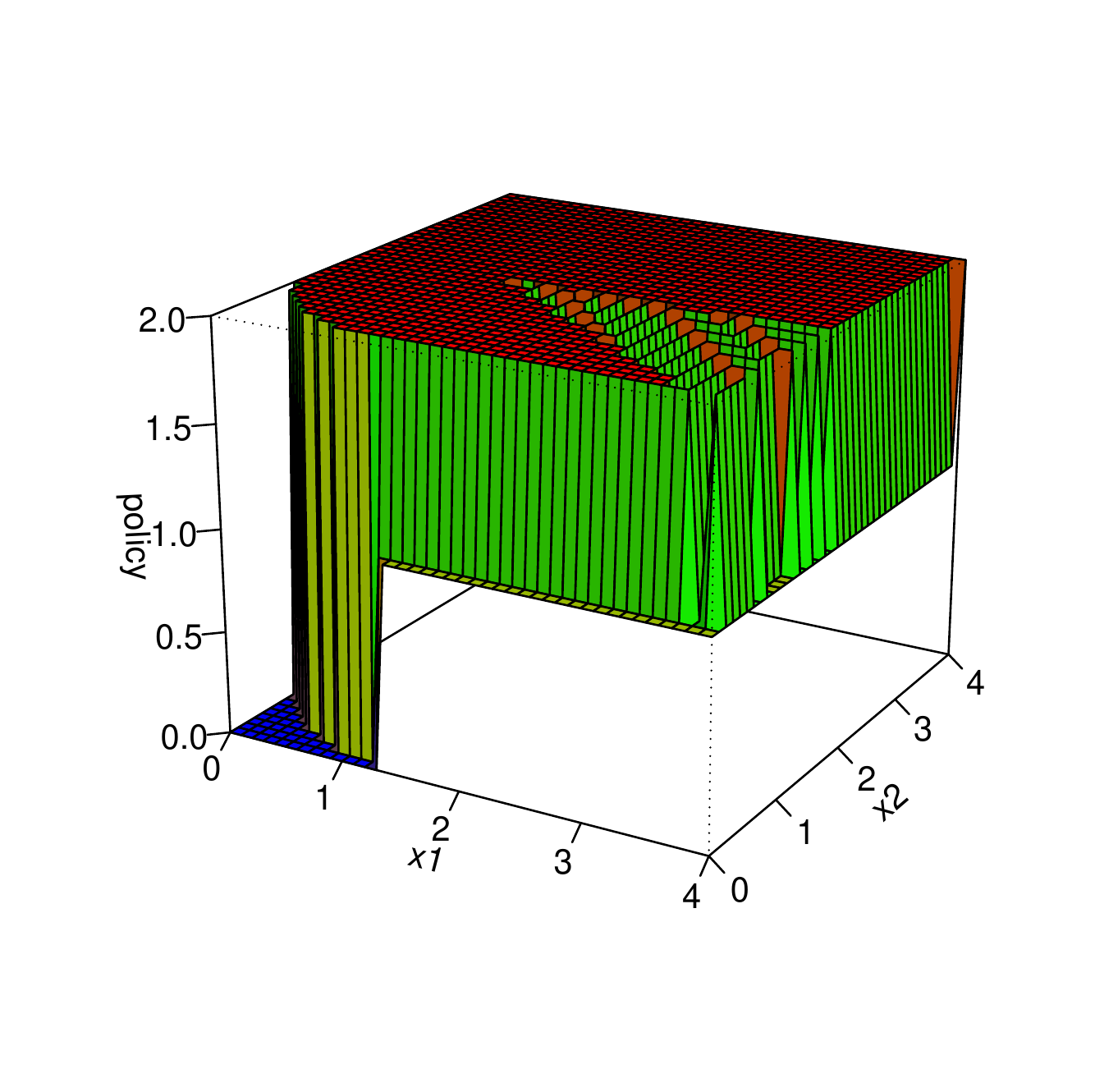} }}%
	\quad
	\subfloat{{\includegraphics[width=8.5cm]{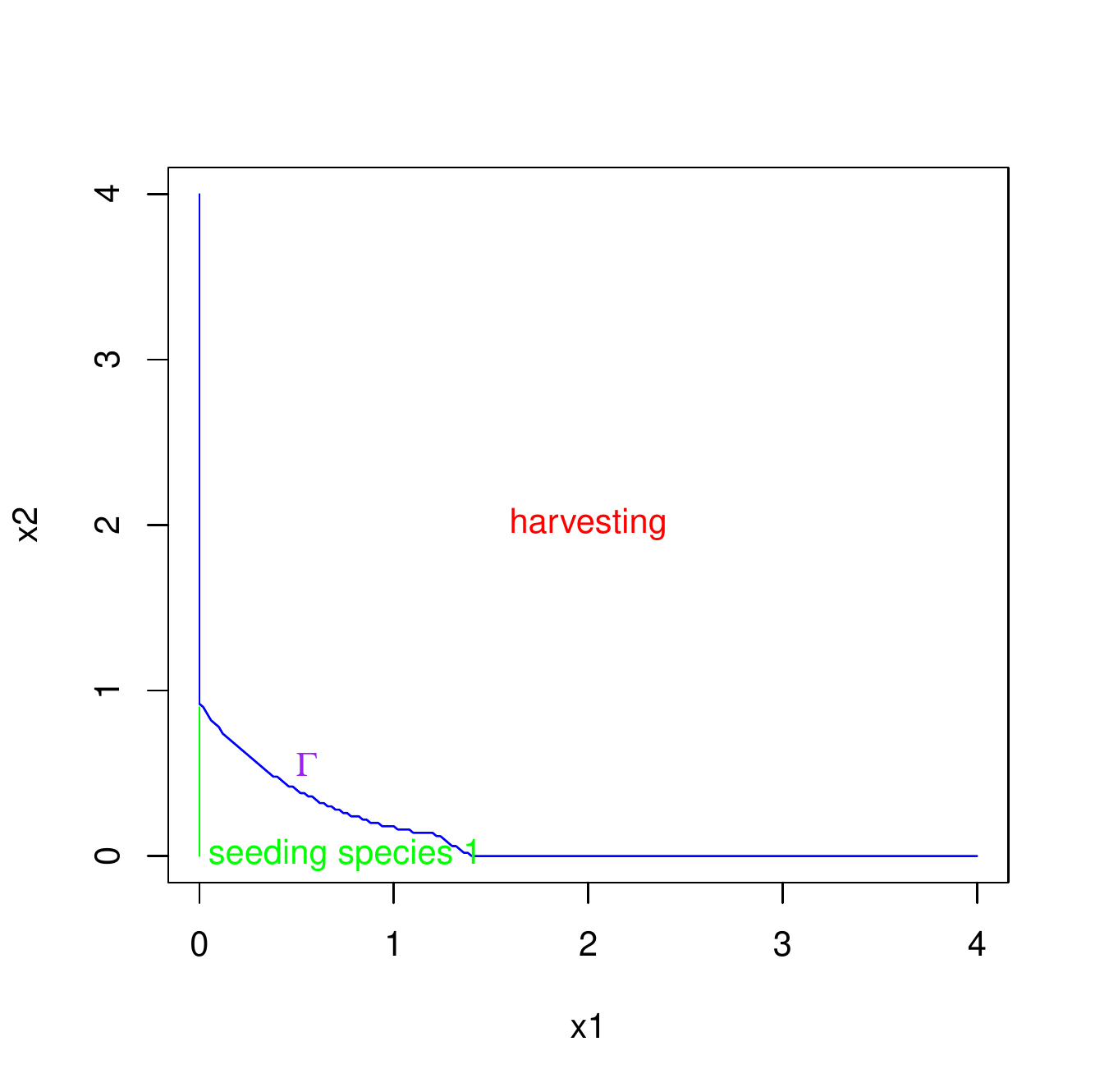} }}%
 	\quad
	
 	\caption{The value function and the optimal policy (1: harvesting of species 1, 2: harvesting of species 2,   0: seeding)}
 	\label{fig5}
 \end{figure}

 \begin{figure}
	\centering
	\subfloat{{\includegraphics[width=6.5cm]{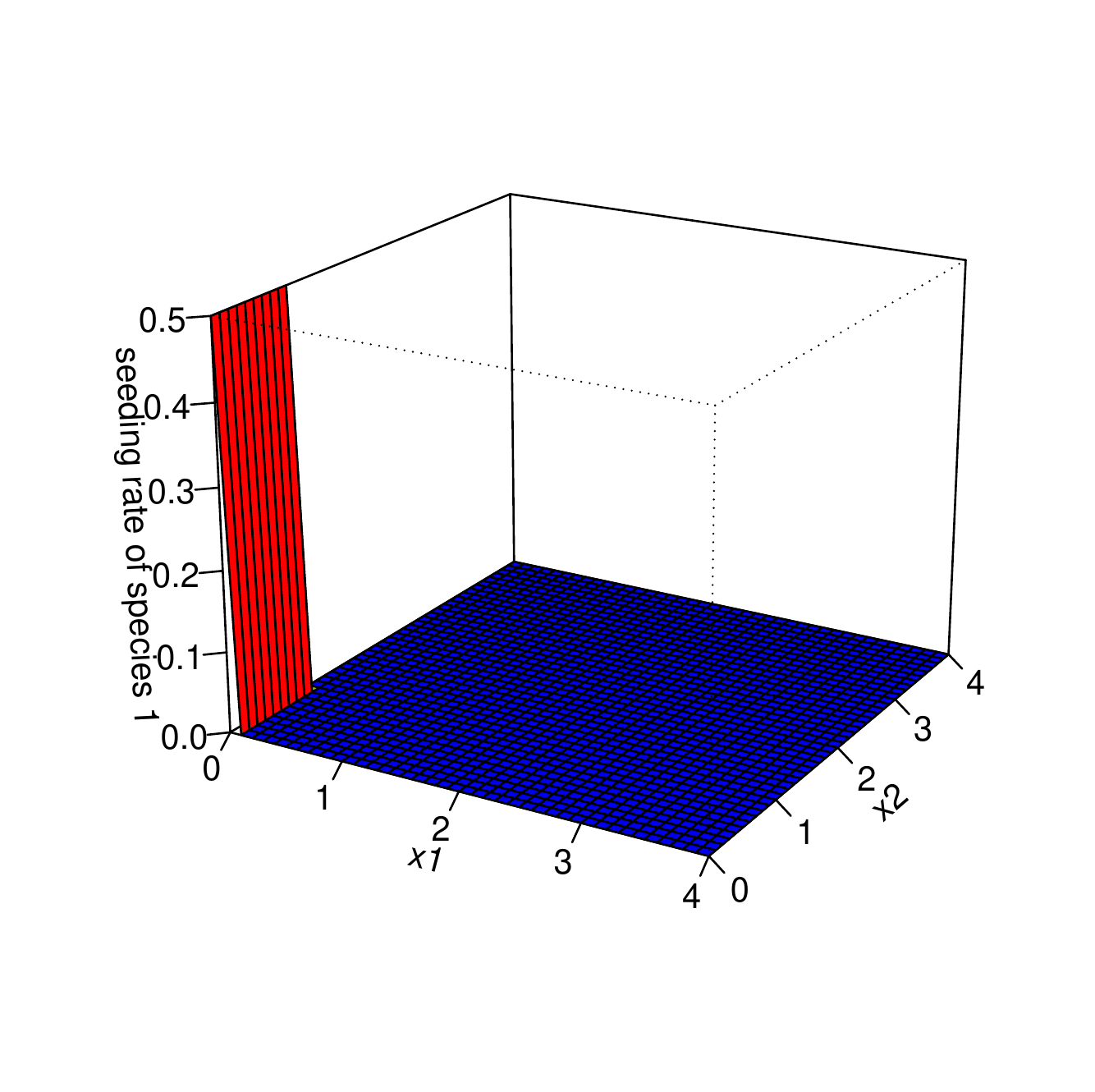} }}%
	\qquad
	\subfloat{{\includegraphics[width=6.5cm]{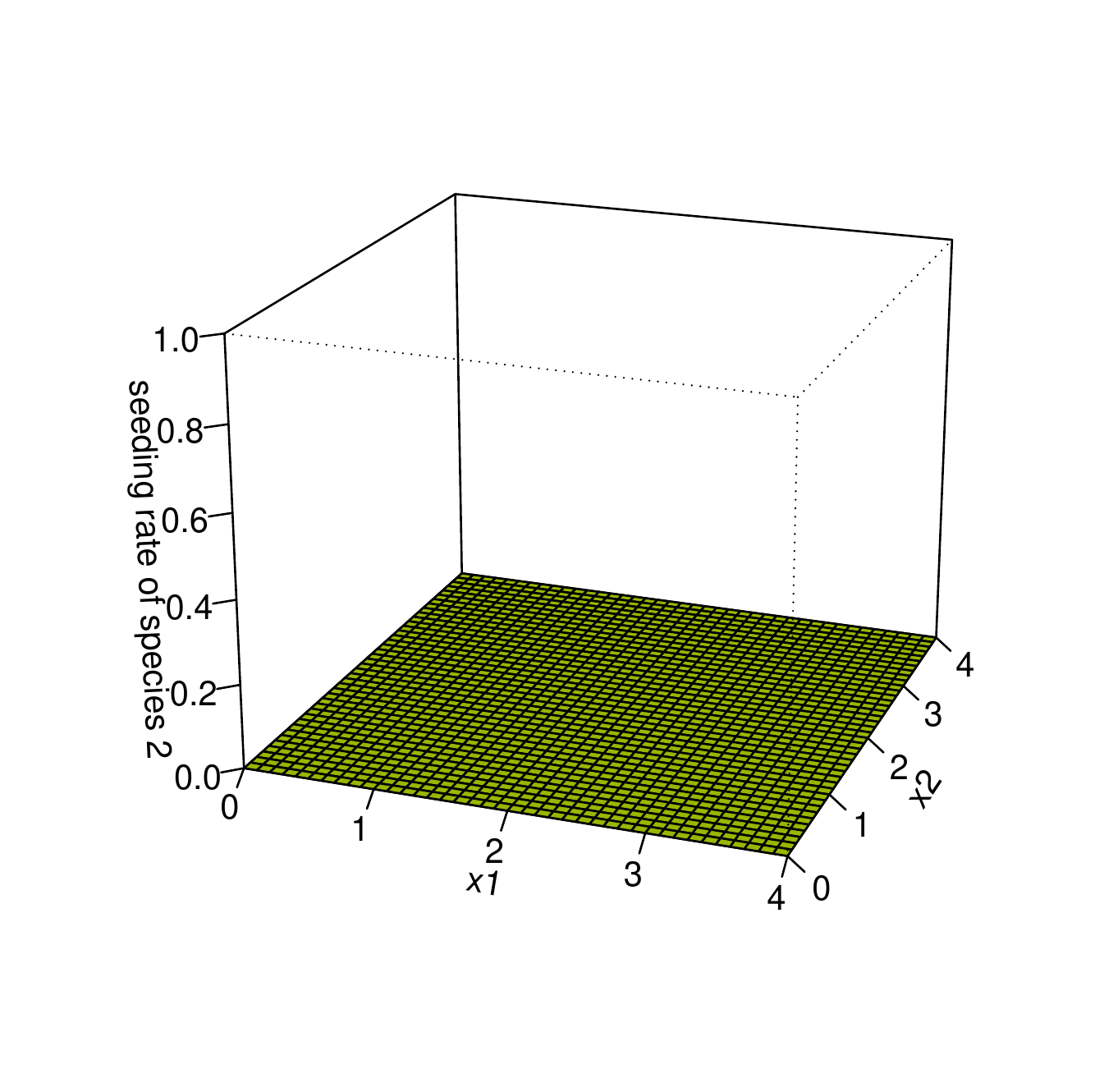} }}%
	\caption{The optimal seeding rates}
	\label{fig6}
\end{figure}

 	\begin{figure}[h!tb]
 	\centering		\subfloat{{\includegraphics[width=5.05cm]{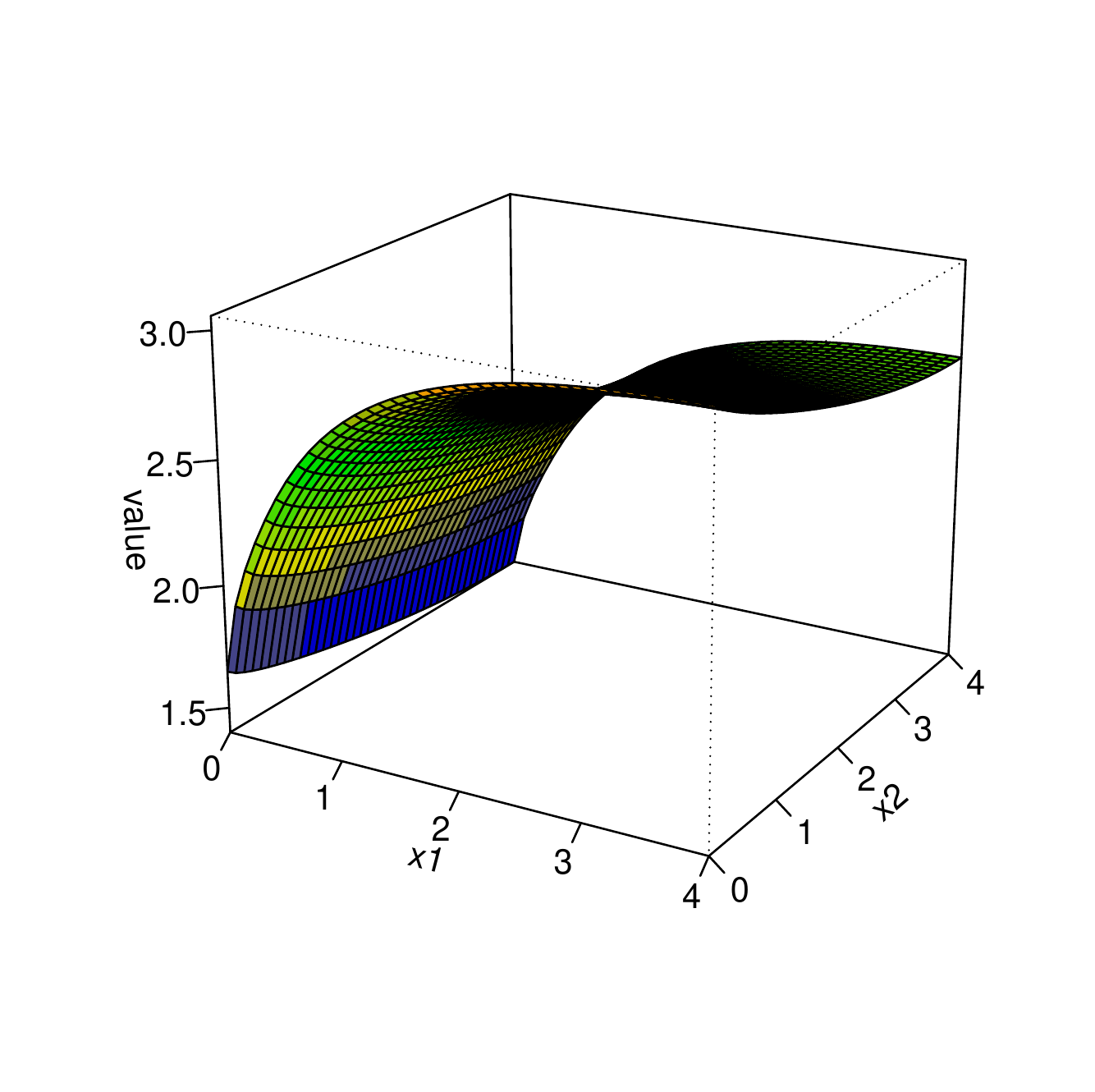} }}%
			\quad
\subfloat{{\includegraphics[width=5.05cm]{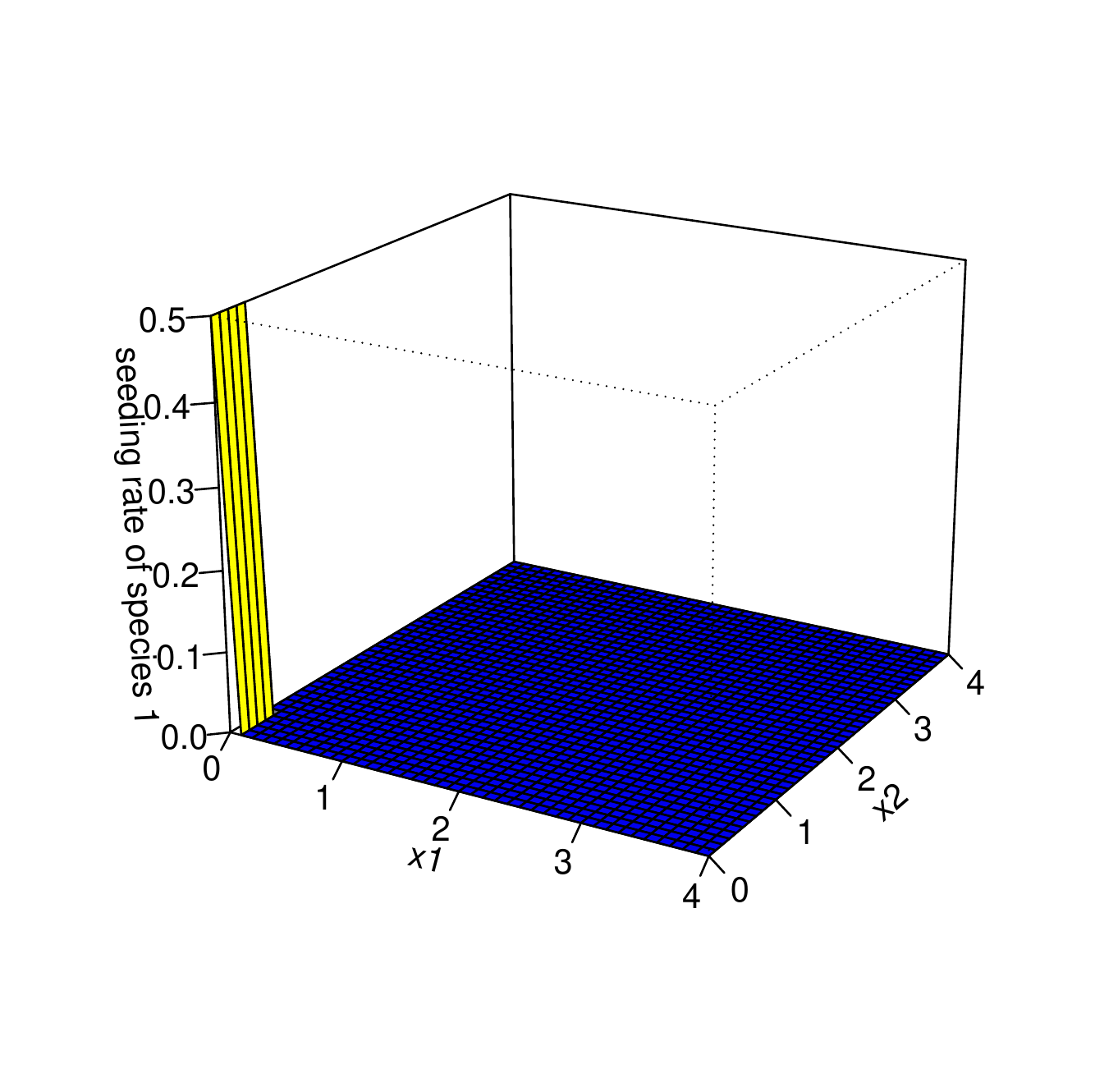} }}%
	\quad
\subfloat{{\includegraphics[width=5.05cm]{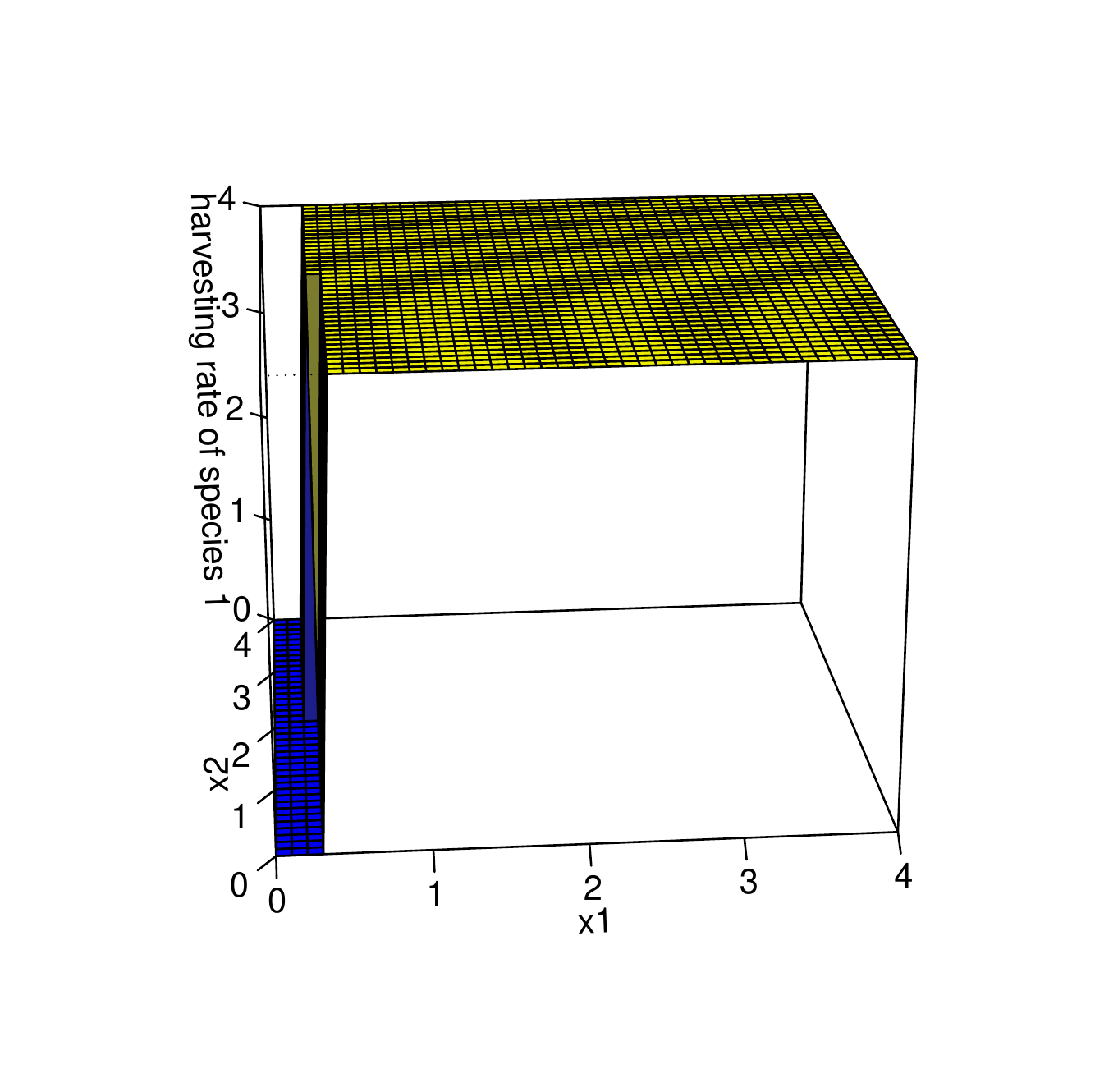} }}%
	\caption{The value function, optimal seeding, and harvesting rates of species 1}
	\label{fig7}
\end{figure}	

For the first experiment we take $$\lambda_1=\lambda_2=0.5, \quad \mu_1=\mu_2=\infty.$$
In Figure \ref{fig5} one can see the value function and the optimal harvesting-seeding policy as functions of the population sizes $(x_1,x_2)$. Here ``1" denotes the harvesting of species 1, ``2" the harvesting of species 2, and ``0" the seeding (including seeding zero). Figure \ref{fig6} provides the optimal seeding rates of the two species.

\textit{\textbf{Biological interpretation}:
We note that it is never optimal to seed species 2 -- the optimal seeding rate of species 2 is identically zero. There is a nonlinear curve $\Gamma$ (see Figure \ref{fig5} and Figure \ref{fig6}) such that it is optimal to harvest whenever the population sizes $(X_1(t), X_2(t))$ lie above $\Gamma$. Seeding takes place only when $(X_1(t), X_2(t))$ is in the green domain, which is close to 0. In particular, only species 1 should be seeded and we should seed with the maximal rate. This observation is well connected with the chosen system parameters. Note that $\frac{a_{11}}{b_1}=\frac{2}{3}<\frac{2}{2}=\frac{a_{21}}{b_2}$ and $\frac{a_{12}}{b_1}=\frac{1.5}{3}<\frac{2}{2}=\frac{a_{22}}{b_2}$. The intraspecific competition within species $1$, given by $\frac{a_{11}}{b_1}$, is smaller than the interspecific competition effect of species $1$ on
species $2$, given by $\frac{a_{21}}{b_2}$, and the interspecific competition effect of species $2$ on
species $1$, given by $\frac{a_{12}}{b_1}$, is smaller that the intraspecific competition within species $2$, given by $\frac{a_{22}}{b_2}$. Moreover, the stochastic growth rate of species $1$ is larger than the stochastic growth rate of species $2$: $b_1-\frac{\sigma_1^2}{2}> b_2-\frac{\sigma_2^2}{2}$. The environment is more favorable to species 1 than to species 2.
}

For the second example, we take $$\lambda_1=0.5, \quad \lambda_2=0, \quad \mu_1=4, \quad \mu_2=0.$$ We are not allowed to seed or harvest species 2. However,
because of the interactions between the two species, the optimal harvesting-seeding policy for the system will depend on the population sizes of both species.
In Figure \ref{fig7} one can see the value function, the optimal seeding rate, and the optimal harvesting rate of species 1.

\textit{\textbf{Biological interpretation}:
There exist lower and upper thresholds $0\leq L_1(x_2)\leq L_2(x_2)$ which depend on the population size of species $2$. Whenever the size of population $1$ is under $L_1(x_2)$ we seed species $1$ at the maximal rate. If the population size of species $1$ is above $L_2(x_2)$ we harvest this species at the maximal rate. Even in this case when we are only allowed to seed or harvest species 1, the optimal harvesting-seeding strategy is not a simple threshold strategy. Due to the interaction of the two species, the optimal policy will depend on the population sizes of both species.
One interesting observation (see Figure \ref{fig7}) is that for a fixed population size $x_1$ of species $1$, the value function is a decreasing function of $x_2$. When the size of $x_2$ increases, due to competition and the fact that we cannot harvest species $2$, the value function will decrease.
}

For the last experiment, we take $$\lambda_1=0, \quad \lambda_2=0.5, \quad \mu_1=0, \quad \mu_2=4, \quad \sg_2 = 2.5.$$
We are not allowed to seed or harvest species 1.
Figure \ref{fig8} provides the value function, the optimal seeding rate, and the optimal harvesting rate of species 2. Similarly to the preceding case (Figure \ref{fig7}), there are levels $L_1^*(x_1)$ and $L_2^*(x_1)$ depending on $x_1$ such that if the abundance of species 2 is larger than $L_2^*(x_1)$, one should harvest species 2 at the maximal rate. If the abundance of species 2 is below $L_1^*(x_1)$, one should seed species 2 using the maximal seeding rate.

\begin{figure}
	\centering
	\subfloat{{\includegraphics[width=5.05cm]{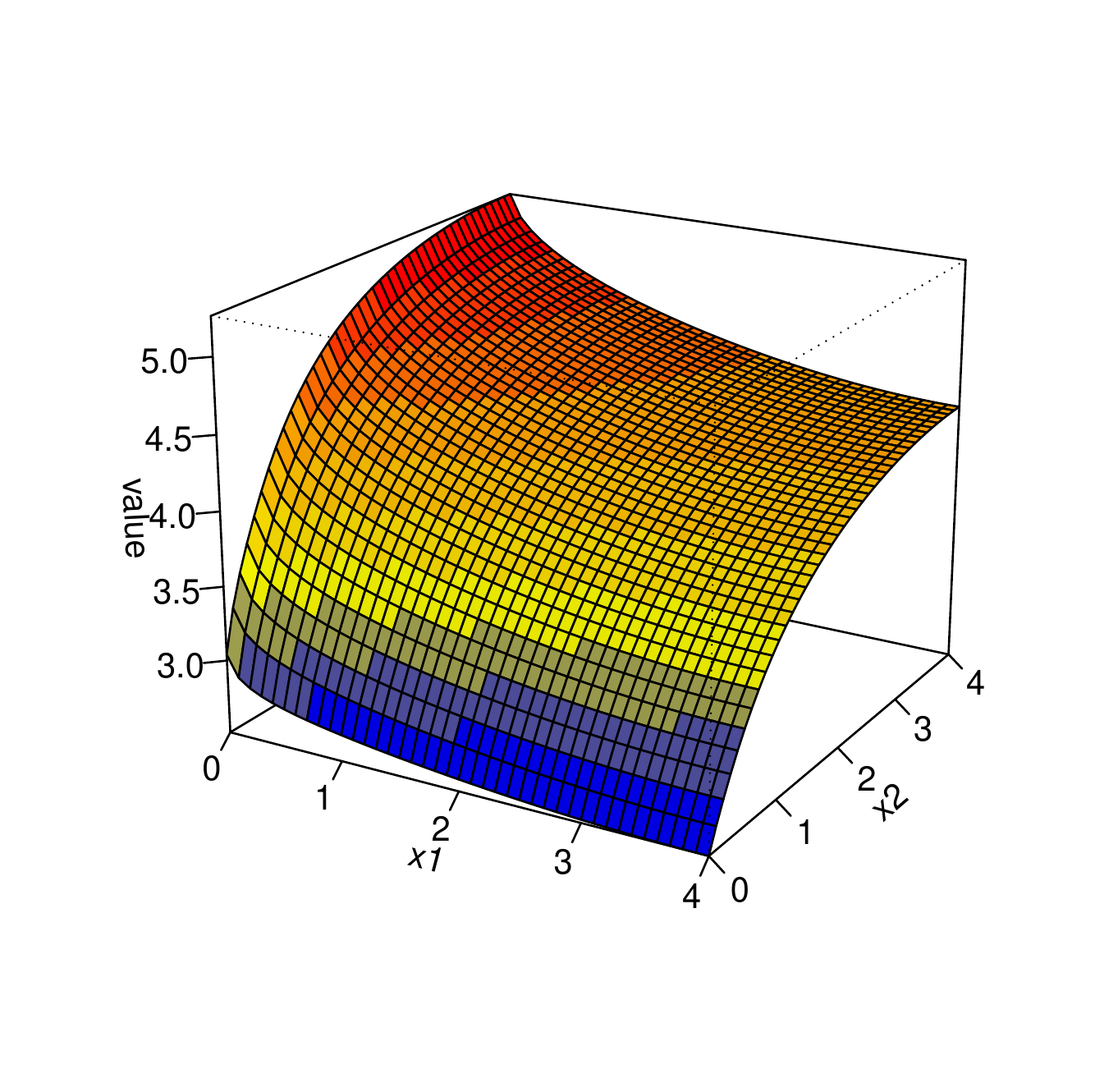} }}%
	\quad
	\subfloat{{\includegraphics[width=5.05cm]{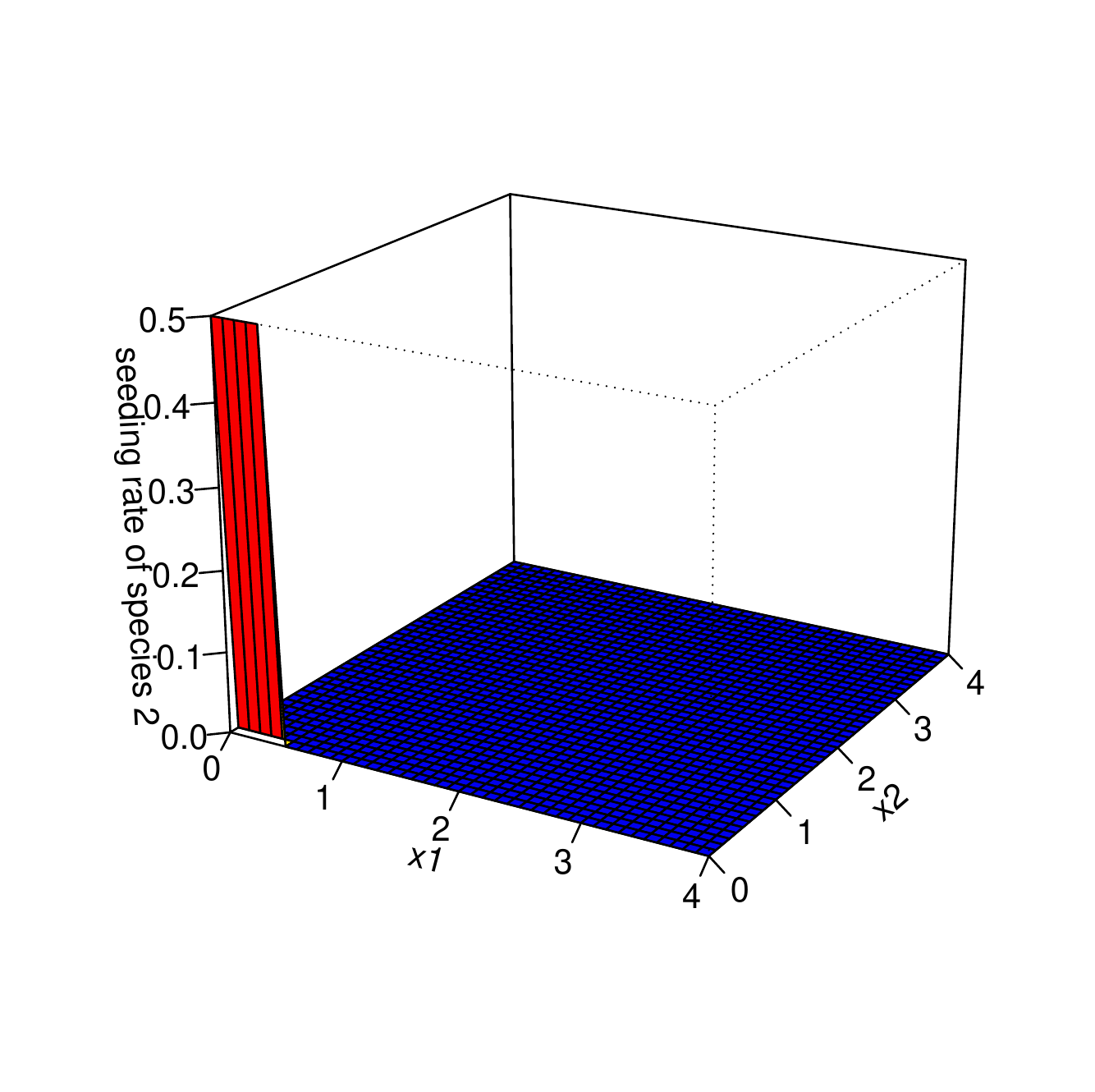} }}%
	\quad
	\subfloat{{\includegraphics[width=5.05cm]{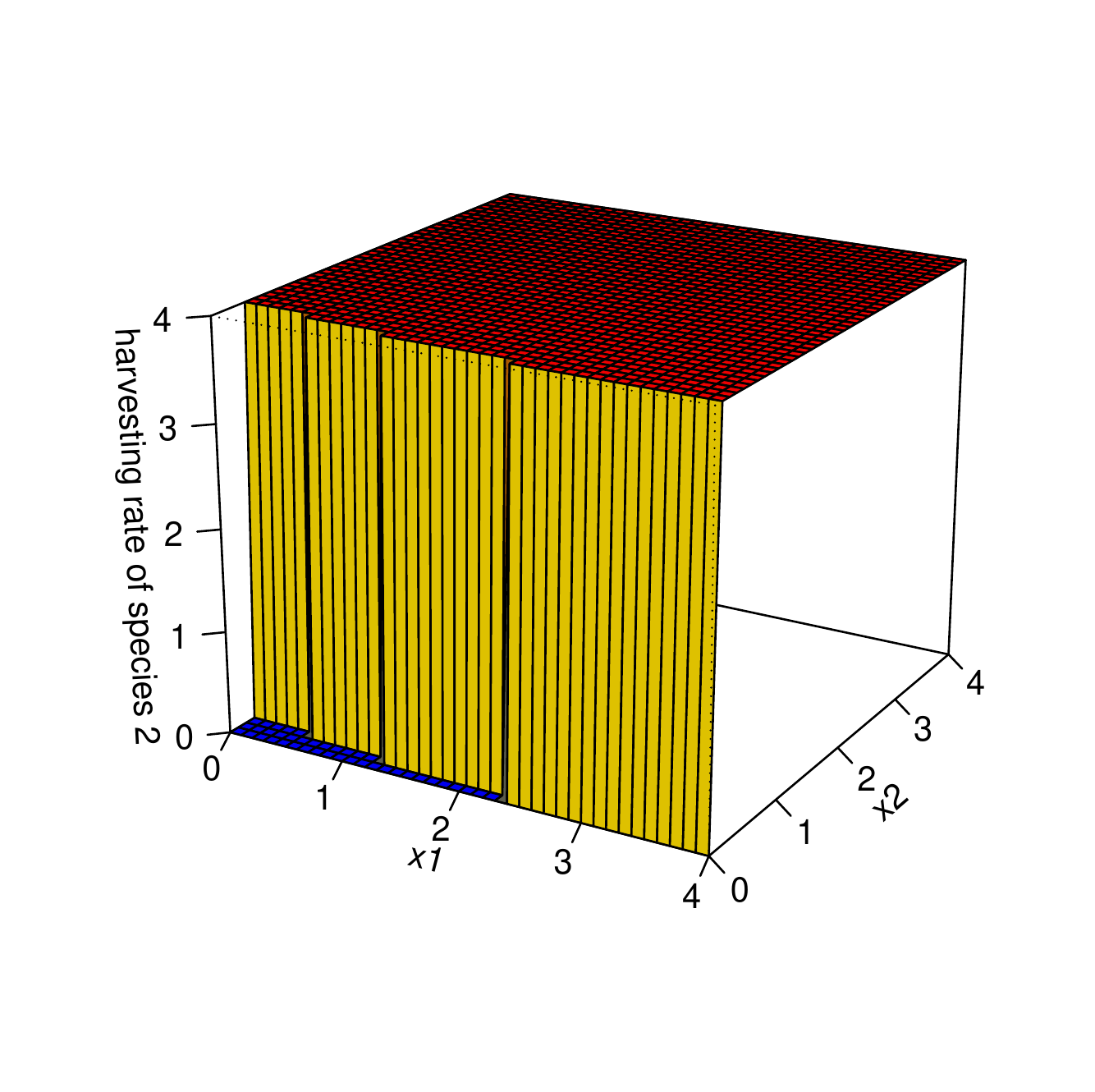} }}%
	\caption{The value function, optimal seeding, and harvesting rates of species 2}
	\label{fig8}
\end{figure}

}
\end{exm}

\begin{exm}{\rm
		Consider a predator-prey model where the predator has a Holling type 2 response and the prey satisfies a logistic equation. The dynamics is given by	
        \beq{}\barray
		\aad d X_1(t)  = X_1(t)\left(b_1-a_{11}X_1(t)- \dfrac{a_{12} X_1(t)}{b_3+X_1(t)} \right)dt + \sg_1X_1(t)dw_1(t) -dY_1(t)+  dZ_1(t),\\
		\aad d X_2(t)  = X_2(t)\left(-b_2 + \dfrac{a_{21} X_1(t)}{b_3+X_1(t)} - a_{22} X_2(t)   \right)dt + \sg_2X_2(t)dw_2(t) -dY_2(t)+  dZ_2(t),
		\earray
		\eeq
		where $X_1(t)$ and $X_2(t)$ denote the population sizes of the prey and that of the predator.  Let $\lambda=(\lambda_1, \lambda_2)'$ and $\mu=(\mu_1, \mu_2)'$ be the maximum seeding rates and the maximum harvesting rates of the prey and predator. We pick the coefficients to be $$\delta=0.05, f_1(x)\equiv 0.5, f_2(x)\equiv 0.75,  g_1(x)\equiv 3, g_2(x)\equiv 4, U=4,$$
and
		\begin{equation*}
		b_1=2,  a_{11}=1.2, a_{12}=1,  \sg_1=1.6,  b_2=1, b_3=1,   a_{21}=4, a_{22}=2,  \sg_2=1.8.
		\end{equation*}
		For the first numerical experiment, we take $$\lambda_1=\lambda_2=0.5, \quad \mu_1=\mu_2=\infty.$$ Figure \ref{fig9} shows the value function  and the optimal policy as a function of the population abundances $(x_1,x_2)$. Here ``1" denotes harvesting of species 1,  ``-1" the
		seeding of species 1, ``2" the harvesting of species 2, and ``0" the seeding of species 2 (the seeding rates are given in Figure \ref{fig10}).

\textit{\textbf{Biological interpretation}: The optimal seeding rate of the predator is identically zero -- it is never optimal to seed the predator. Moreover, one starts harvesting the predator at a low density -- it is optimal to keep the predator size low. This makes sense as the driving force of the dynamics is given by the prey species. The predator will always go extinct on its own. If one keeps the predator population low, the prey species can grow and one can then harvest this population as well.
There is a curve $\Gamma$ (Figure \ref{fig9}) such that $(X_1(t), X_2(t))$ is above this curve, it is optimal to harvest. There is seeding of the prey species when the $(X_1(t), X_2(t))$ is in the green domain, which is close to 0. We only seed the prey species when it is close to extinction (or is initially extinct).}	
	\begin{figure}
 	\centering
 \subfloat{{\includegraphics[width=6.5cm]{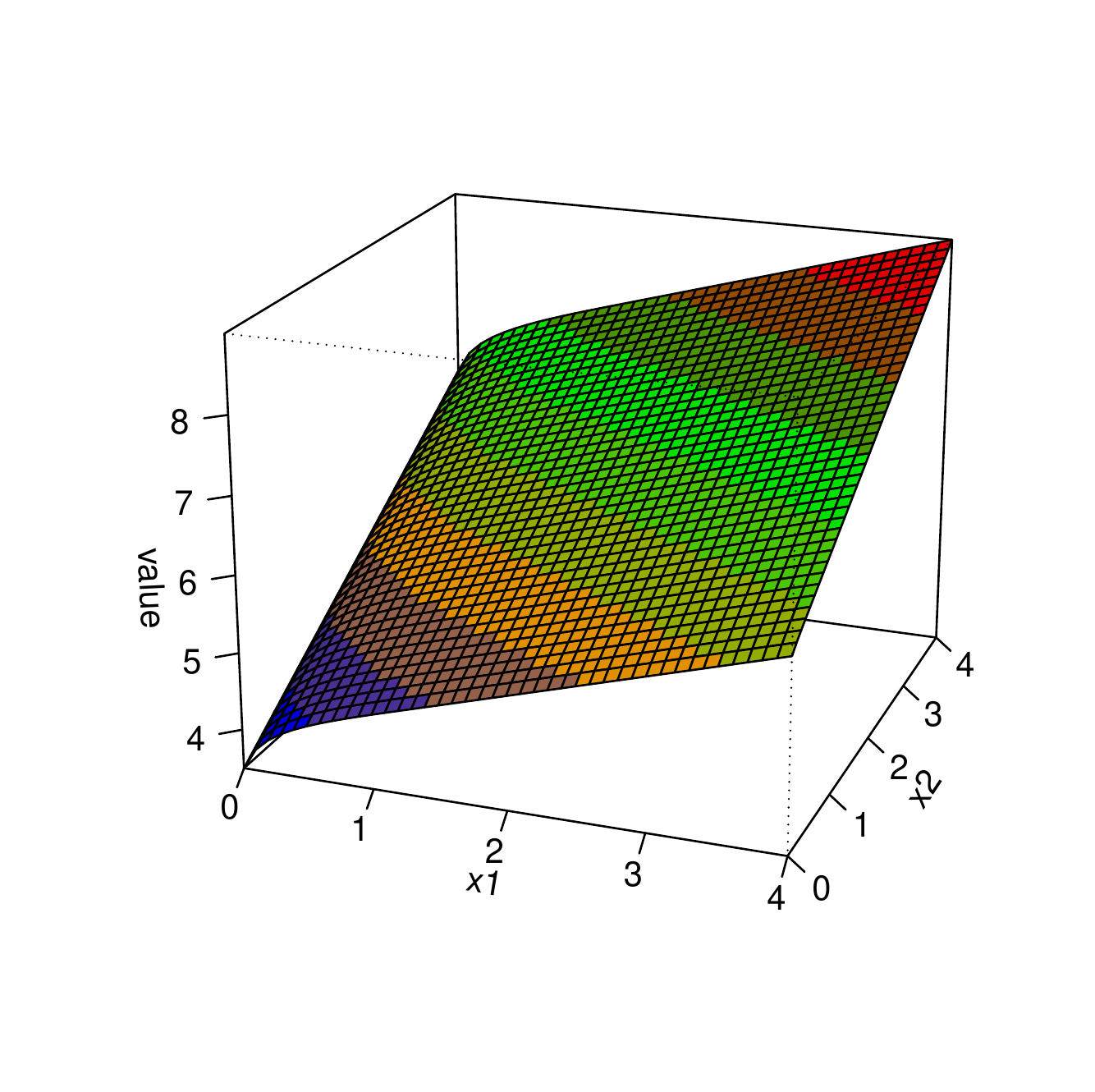} }}%
 	\quad
 	\subfloat{{\includegraphics[width=6.5cm]{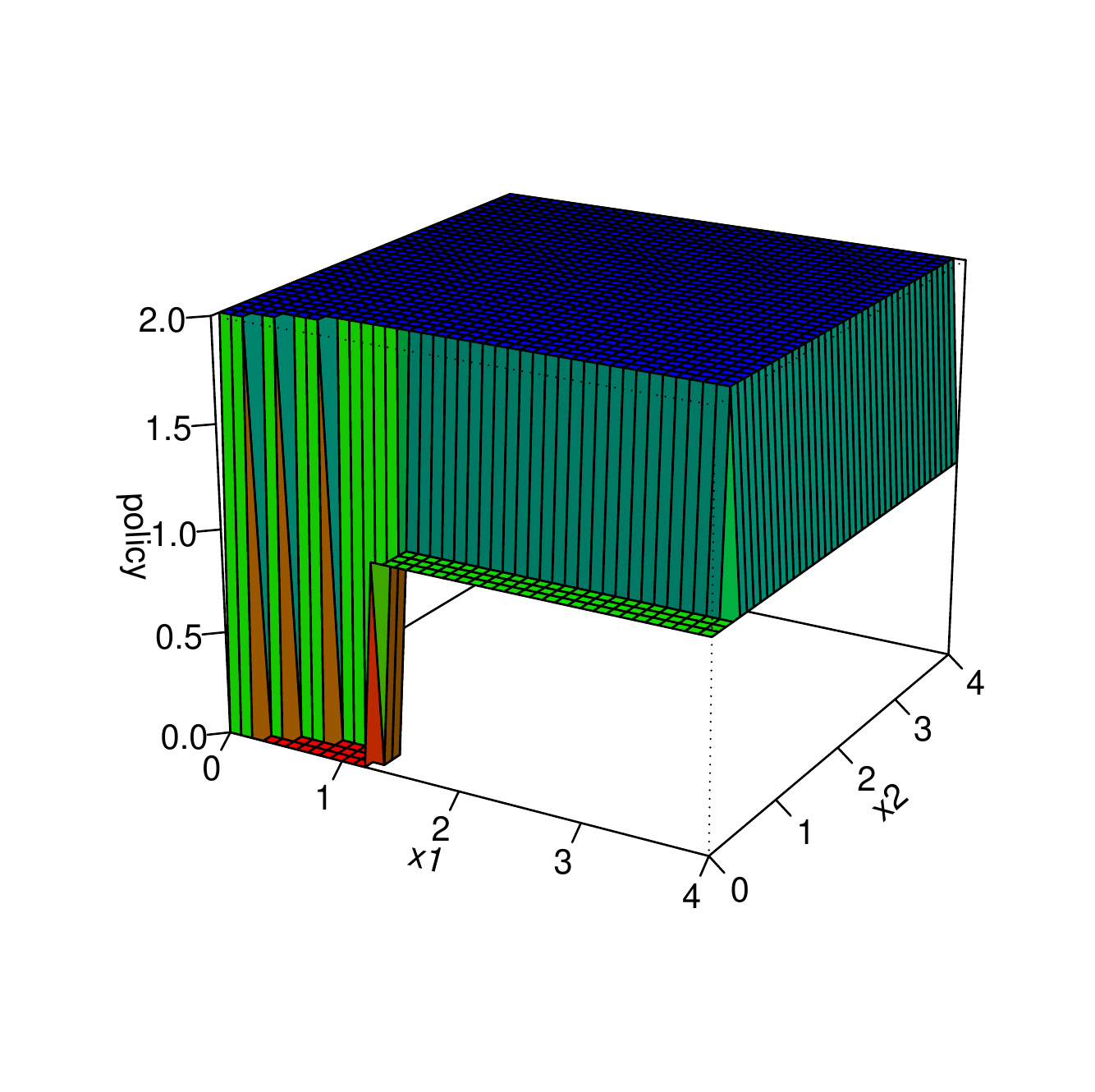} }}%
\quad
\subfloat{{\includegraphics[width=8.5cm]{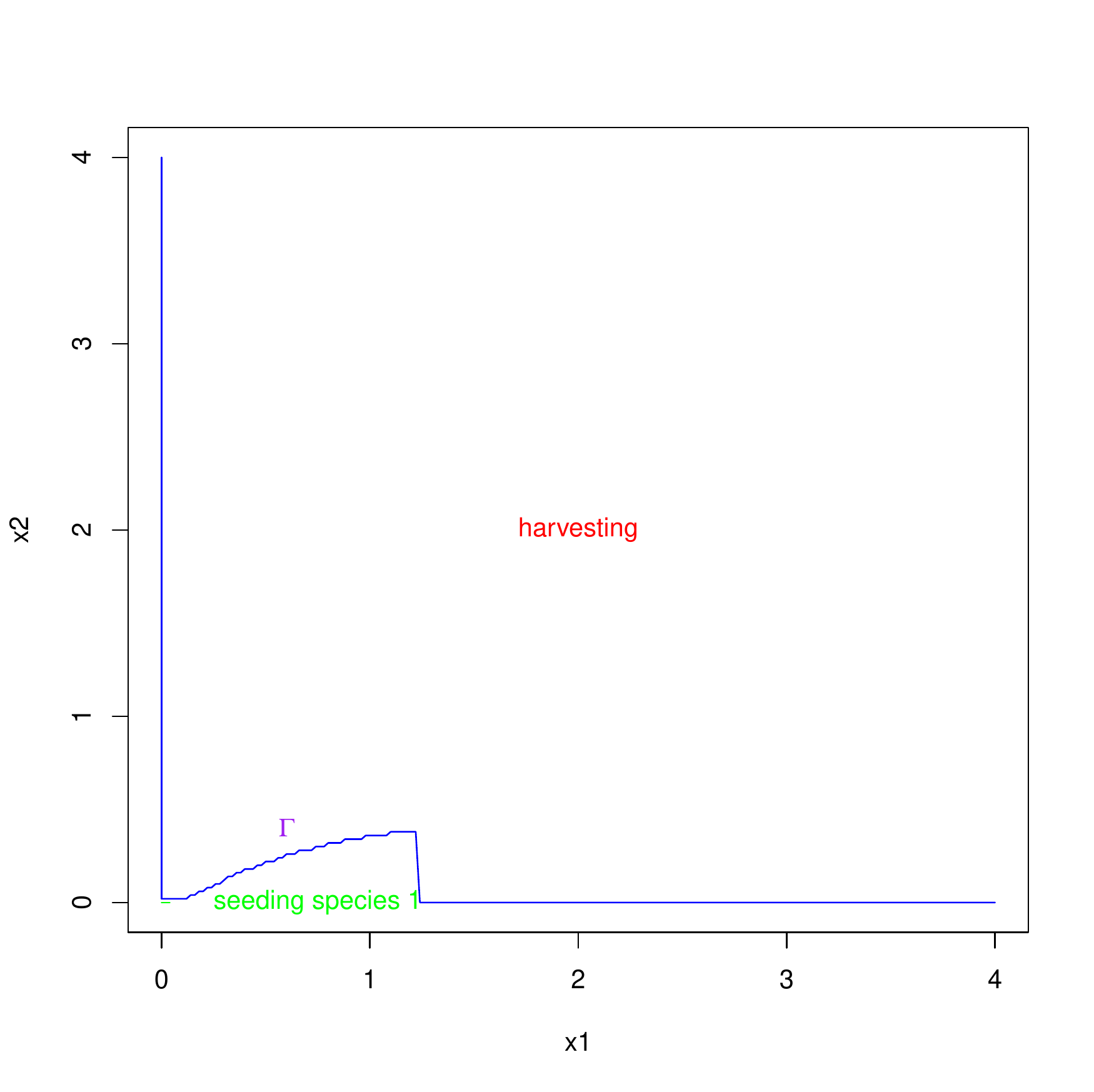} }}%
 	\caption{The value function and the optimal policy (1: harvesting of the prey, 2: harvesting of the predator,   0: seeding)}
 \label{fig9}
 \end{figure}

 \begin{figure}
	\centering
	\subfloat{{\includegraphics[width=6.5cm]{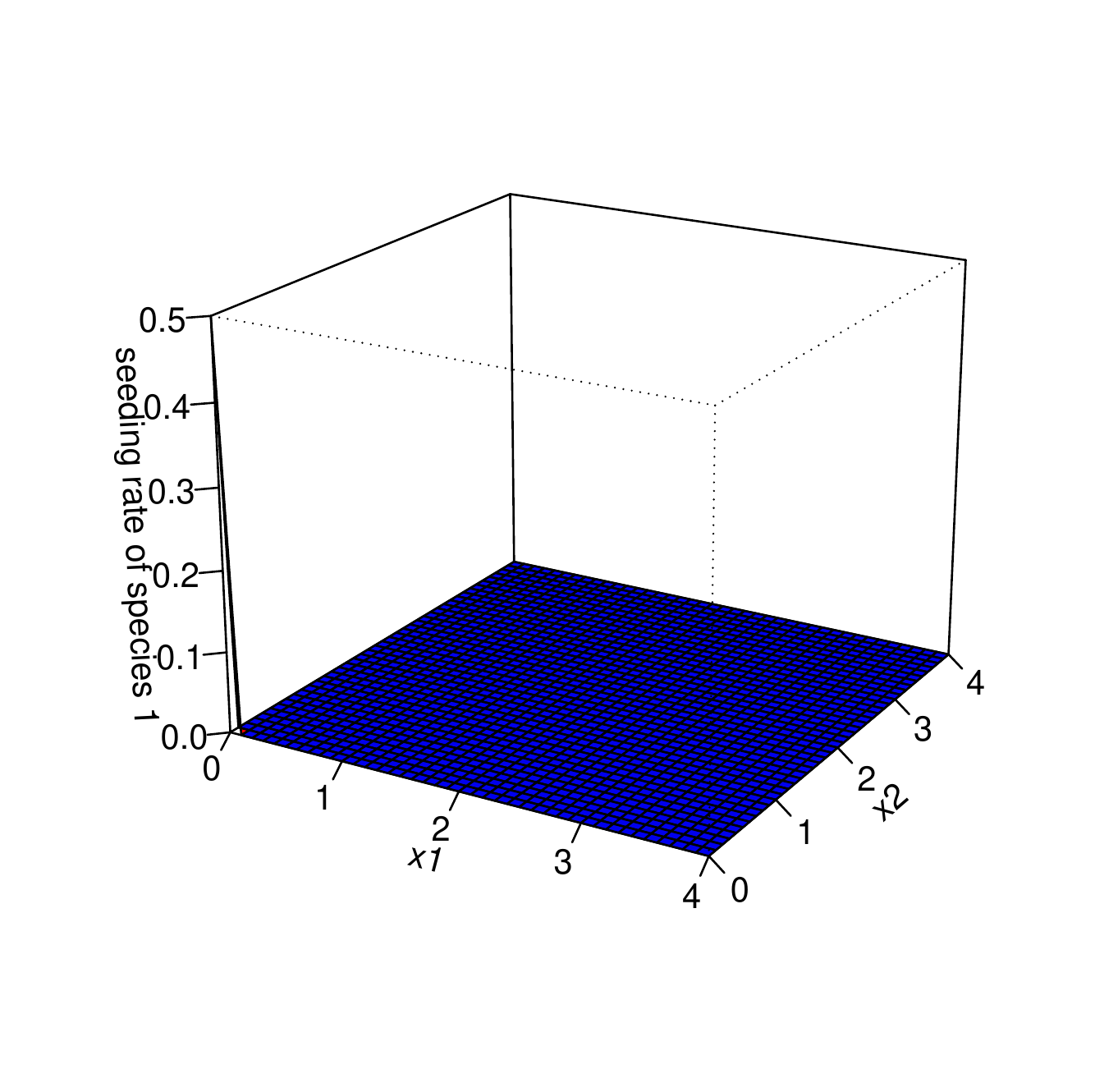} }}%
	\qquad
	\subfloat{{\includegraphics[width=6.5cm]{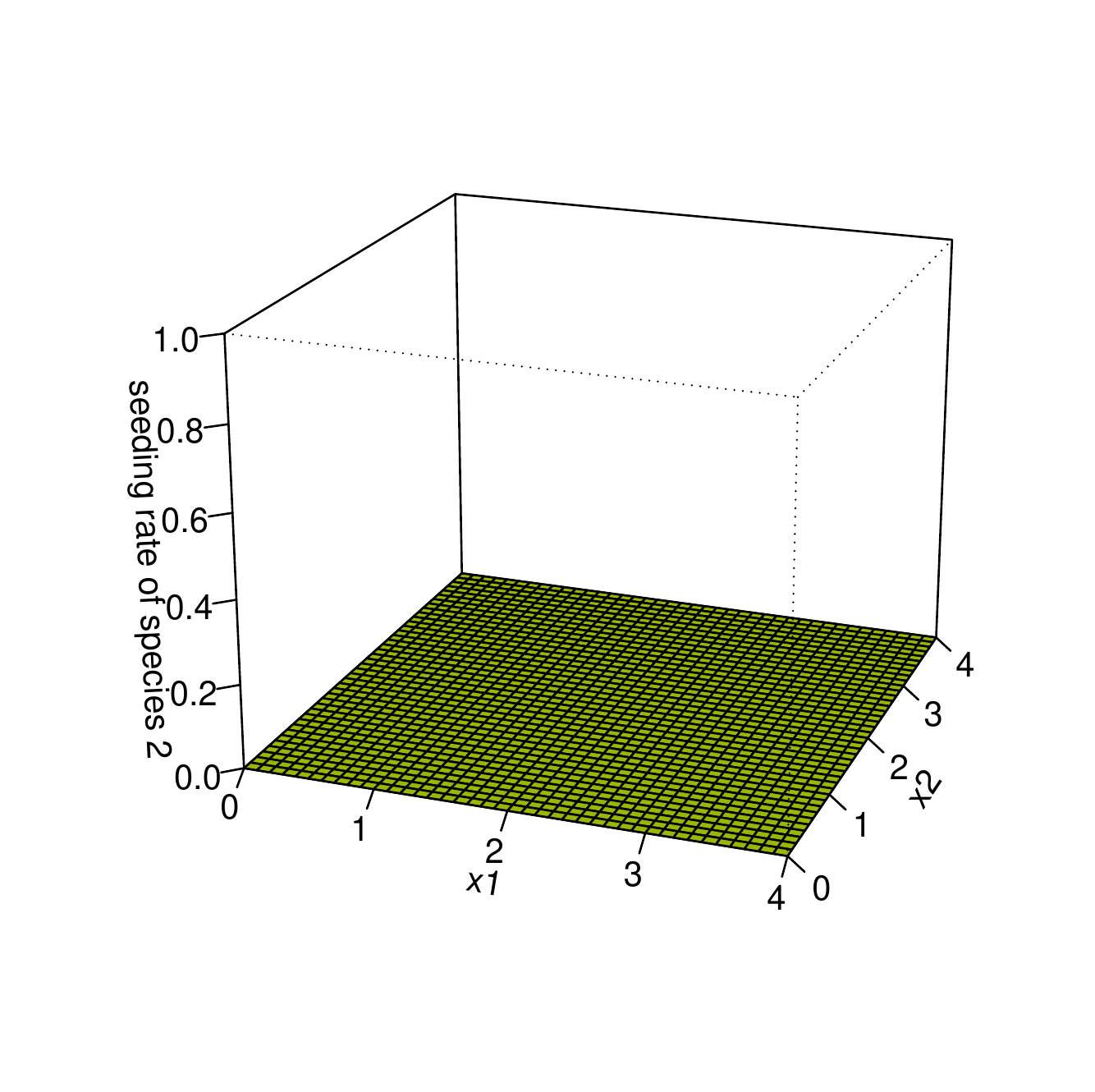} }}%
	\caption{The optimal seeding rates}
	\label{fig10}
\end{figure}

		For the second numerical experiment, we take $$\lambda_1=0, \quad \lambda_2=0.5, \quad \mu_1=0, \quad \mu_2=5,$$
so that only the predator can be seeded or harvested.
Figure \ref{fig11} provides the value function, the optimal seeding rate, and the optimal harvesting rate of
the predator.

\textit{\textbf{Biological interpretation}:
Just as in the first numerical experiment, it turns out that it is never optimal to seed the predator. Even if both species are extinct, and we are not allowed to seed the prey species, it is not optimal to seed the predator. Since the predator goes extinct without the prey, the optimal strategy is to harvest all of it immediately if there is no prey to sustain the dynamics. There is a level $L(x_1)$ which depends on the size of the prey population such that if the predator population is above $L(x_1)$ it is optimal to harvest it at the maximal rate. }

\begin{figure}[h!tb]
 	\centering		\subfloat{{\includegraphics[width=5.05cm]{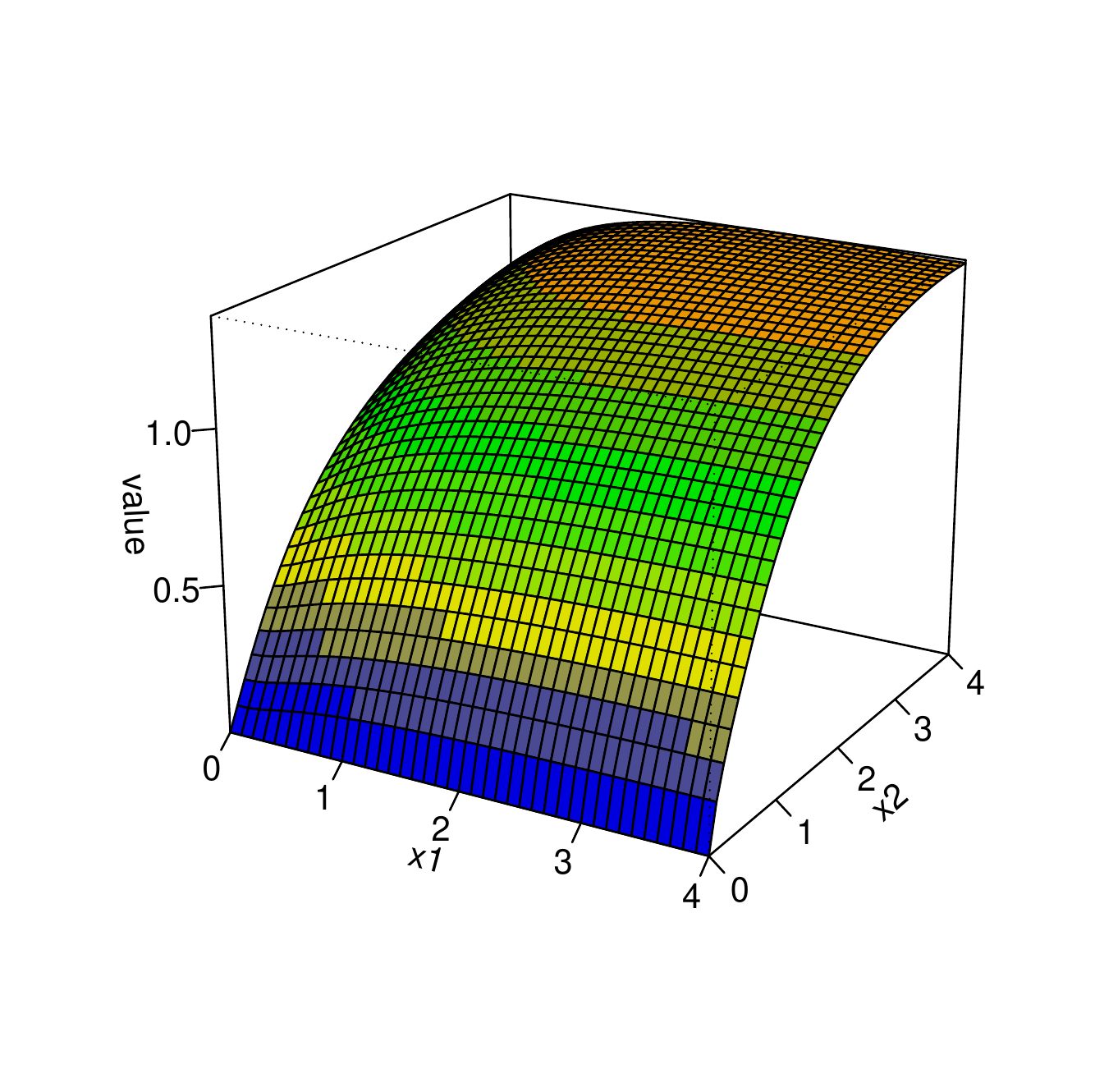} }}%
			\quad
\subfloat{{\includegraphics[width=5.05cm]{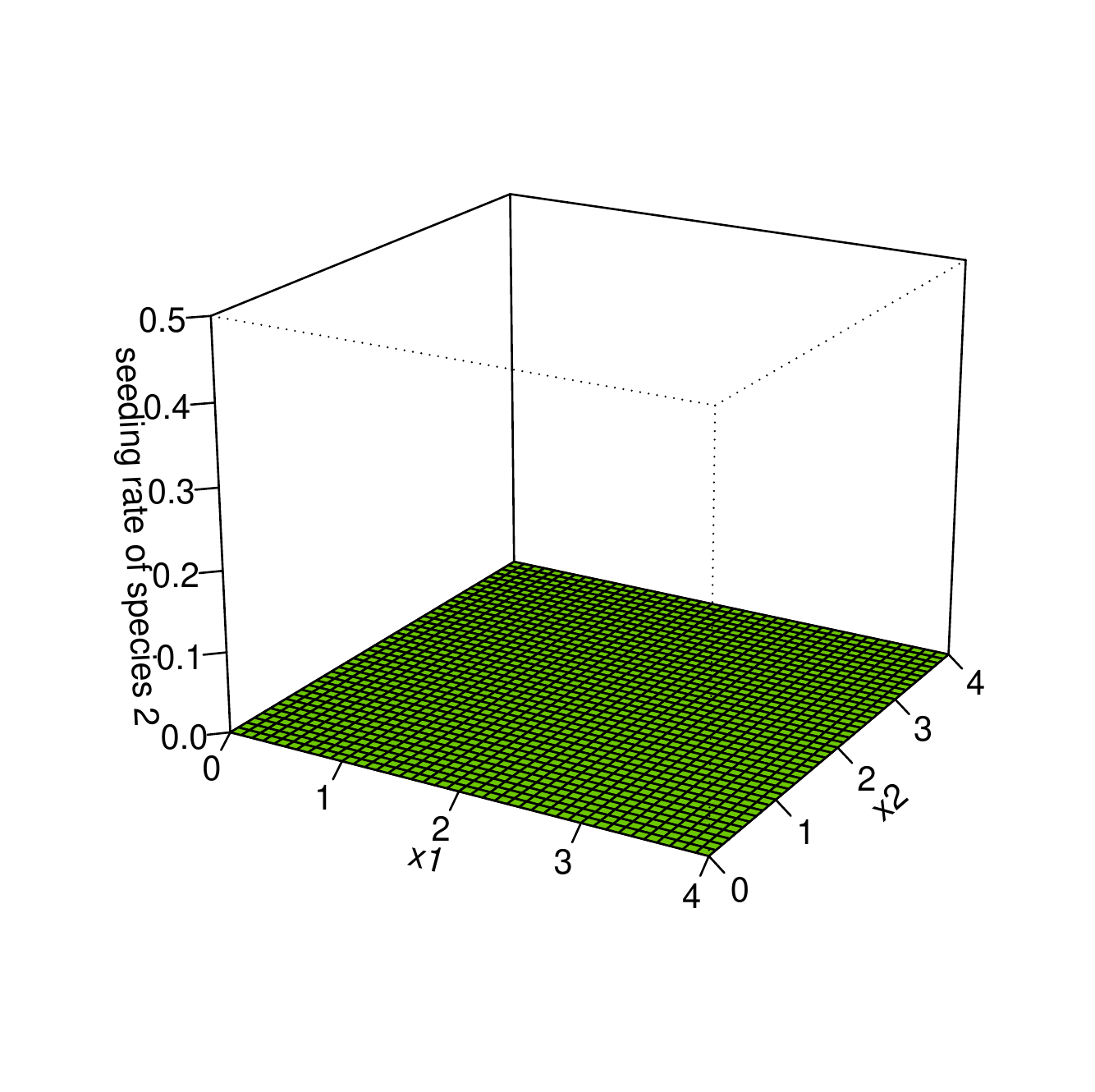} }}%
	\quad
\subfloat{{\includegraphics[width=5.05cm]{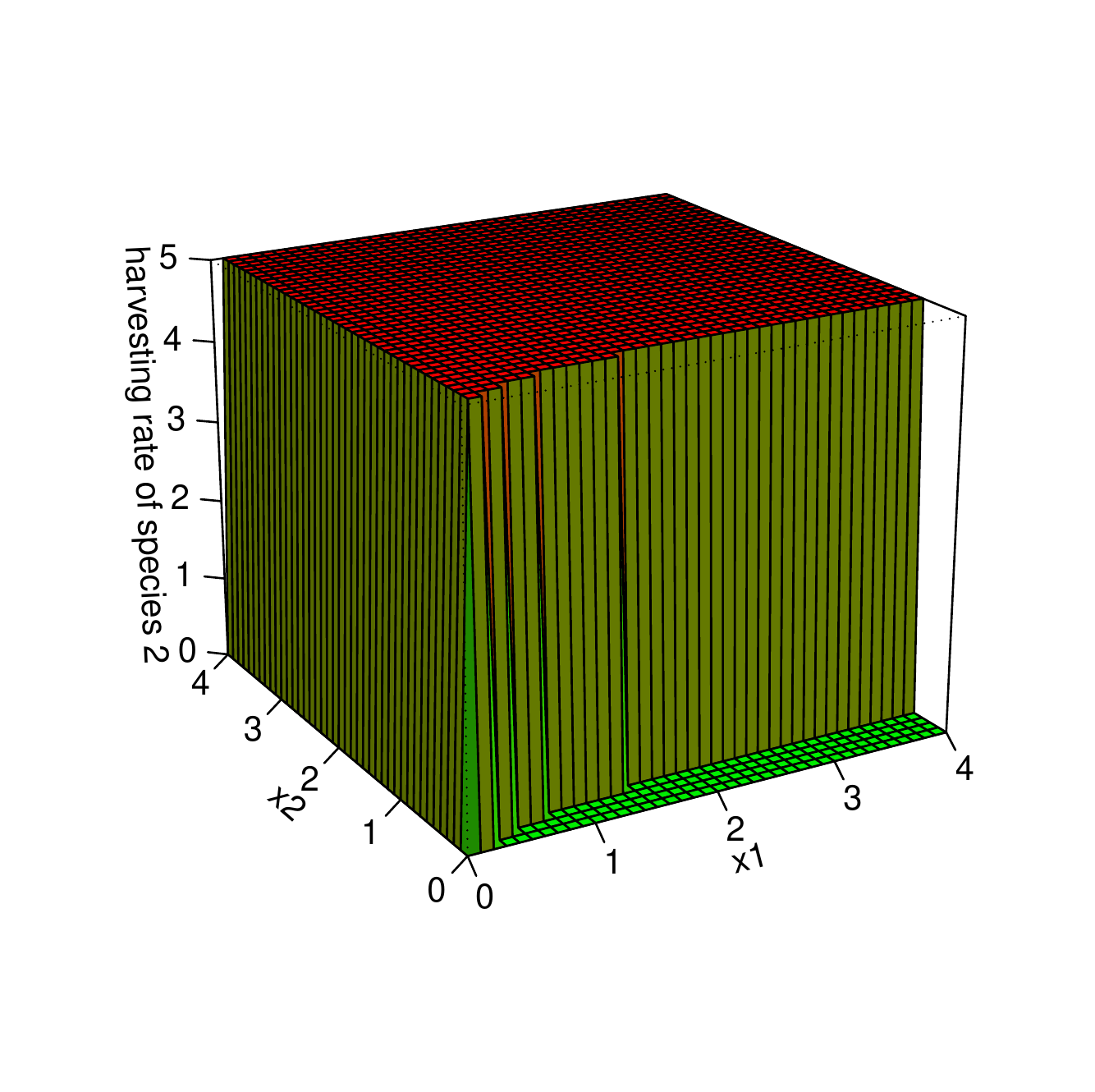} }}%
	\caption{The value function, optimal seeding, and harvesting rates of species 2}
	\label{fig11}
\end{figure}

	}
\end{exm}

\textbf{Acknowledgements:} Alexandru Hening has been supported by the NSF through the grant DMS-1853463.
\clearpage
\bibliographystyle{agsm}
\bibliography{harvest}

\appendix

\section{Properties of the value function}

\begin{prop}\label{prop2} Assume we are in the setting of bounded seeding and unbounded harvesting rates.
	Suppose that there exists number $U>0$ such that
	$$\sum\limits_{i=1}^d \big[ b_i(x) -\delta(x_i-U)\big]f_i<0 \quad \text{for}\quad |x|>U.$$ Then there exists $x^*\in [0, U]^d$ such that
	$$V(x)= V(x^*) + f\cdot (x-x^*)\quad \text{for}\quad x\in \overline{S} \setminus [0, U]^d.$$
Moreover,
	$$
	V(x)=V^U(x) \quad \text{for}\quad x\in [0, U]^d.
	$$
\end{prop}

\begin{proof}
	Fix some $x \in \lbar S\setminus [0, U]^d$ and $(Y, C)\in \mathcal{A}_{x}$, and let $X$ denote the corresponding harvested process. Let $x_i^*= \min\{x_i, U\}$ for $i=1, \dots, d$ and $x^*=(x_1^*, \dots, x_d^*)'$.
	
	Let $\e\in (0, 1)$ be a constant and
	define  \beq{5e:2}\Phi_\e(y)=f \cdot (y-x^*) +\e, \quad y\in \lbar S\setminus [0, U]^d.\eeq
	 We can extend $\Phi_\e\cd$ to the entire $\lbar S$ so that $\Phi_\e\cd$ is twice continuously differentiable, $\Phi_\e(y)\ge 0$ and $f\le \nabla \Phi_\e(y)$ for all $y\in \lbar  S$.
	 By assumption, we can check that $$(\mathcal{L}-\delta)\Phi_\e(y)=\sum\limits_{i=1}^d \big[ b_i(y) -\delta(y_i-x^*)\big]f_i-\delta \e< 0 \quad  \text{for} \quad y\in \lbar S\setminus [0, U]^d.$$
	  Choose $N$ sufficiently large so that $|x|< N$. For
	$$
	\beta_N=\inf\{t\ge 0:   |X(t)|\ge N\}, \quad	\gamma_0=\inf\{t\ge 0:    X(t)\in [0, x^*]\}, \quad T_N = N\wedge \beta_N\wedge \gamma_0,
	$$
we have
$T_N \to \gamma_0$ with probability one as $N \to\infty$.	
By Dynkin's formula, 	\bea
	\E_{x} \ad \big[e^{-\delta T_N}\Phi_\e\(X(T_N)\)\big]-\Phi_\e(x)\\
	\ad = \E_{x}\int_{0}^{T_N} e^{-\delta s}(\L-\delta)\Phi_\e\(X(s)\)ds-\E_{x}\int_{0}^{T_N} e^{-\delta s}\nabla \Phi_\e\(X(s)\)\cdot dY^c(s)\\
	\ad\quad +\E_{x}\int_{0}^{T_N} e^{-\delta s}\nabla \Phi_\e\(X(s)\)\cdot C(s)ds + \E_{x}\sum\limits_{0\le s\le T_N}e^{-\delta s}\Big[\Phi_\e\(X(s)\)-\Phi_\e\(X(s-)\)\Big],
	\eea
where $Y^c\cd$ is the continuous part of $Y\cd$. Let $\Delta Y(s)= Y(s)-Y(s-)$. 	Since $\nabla \Phi_\e(X(s))=f$ and $\Phi_\e\(X(s)\)-\Phi_\e\(X(s-)\)=-f\cdot \Delta Y(s)$,
	we obtain
	\beq{3e.3.3}\barray
	\E_{x} \ad \big[e^{-\delta T_N}\Phi_\e\(X(T_N)\)\big]-\Phi_\e(x) \le
\E_{x}\int_{0}^{T_N} e^{-\delta s}(\mathcal{L}-\delta)\Phi_\e(X(s))ds\\
	\ad \quad  -\E_{x}\int_{0}^{T_N} e^{-\delta s}f\cdot dY^c(s)  + \E_{x}\int_{0}^{T_N} e^{-\delta s}f\cdot C(s)ds - \E_{x}\sum\limits_{0\le s\le T_N}e^{-\delta s}f\cdot \Delta Y(s).
	\earray
	\eeq
	Since $\Phi_\e(y)\ge 0$ and $f<g(y)$ for any $y\in \lbar S$, it follows from \eqref{3e.3.3}  that
$$	
	\E_{x}\int_{0}^{T_N} e^{-\delta s}f\cdot dY(s) - \E_{x}\int_{0}^{T_N} e^{-\delta s}g(X(s))\cdot C(s)ds\le \Phi_\e(x)+\E_{x}\int_{0}^{T_N} e^{-\delta s}(\mathcal{L}-\delta)\Phi_\e(X(s)) ds.
	$$
	Letting $N\to\infty$, by the bounded convergence theorem, we obtain
	$$	
	\E_{x}\int_{0}^{\gamma_0} e^{-\delta s}f\cdot dY(s) - \E_{x}\int_{0}^{\gamma_0} e^{-\delta s}g(X(s))\cdot C(s)ds\le  \Phi_\e(x)  +\E_{x}\int_{0}^{\gamma_0} e^{-\delta s}(\mathcal{L}-\delta)\Phi_\e(X(s))ds.
	$$
	As a result
\bea
J(x, Y, C)\ad \le \E_{x}\Big[\int_{0}^{\gamma_0} e^{-\delta s}f\cdot dY(s) - \int_{0}^{\gamma_0} e^{-\delta s}g(X(s))\cdot C(s)ds +  V(X(\gamma_0)) \Big] \\
\ad \le  V(x^*) + \Phi_\e(x)+\E_{x}\int_{0}^{\gamma_0} e^{-\delta s}(\mathcal{L}-\delta)\Phi_\e(X(s))ds.
\eea
The above implies
\beq{4e:1}
J(x, Y, C)\le V(x^*) + f\cdot (x-x^*) + \e+\E_{x}\int_{0}^{\gamma_0} e^{-\delta s}(\mathcal{L}-\delta)\Phi_\e(X(s))ds.
\eeq
Letting $\e\to 0$ in \eqref{4e:1}
\beq{4e:2}
J(x, Y, C)\le V(x^*) + f\cdot (x-x^*)-\E_{x}\int_{0}^{\gamma_0} e^{-\delta s}(\mathcal{L}-\delta)\Phi_0(X(s))ds,
\eeq
where $\Phi_0\cd$ is also defined by \eqref{5e:2} at $\e=0$.
Note that if $\P(\gamma_0=0)<1$, then  \eqref{4e:1} is a strict inequality.
	On the other hand, it is obvious (by harvesting instantaneously $x-x^*$ at time $t=0$) that
	\beq{4e:3}
	V(x)\ge V(x^*) + f\cdot (x-x^*).\eeq
	In view of \eqref{4e:2} and \eqref{4e:3}, if  $x \in \lbar S \setminus [0, U]^d$, $V(x)=V(x^*) + f\cdot (x-x^*)$. Moreover,
	it is optimal to instantaneously harvest an amount of $x-x^*$ to drive the population to the state $x^*$ on the boundary of $[0, U]^d$,
	and then apply an optimal or near-optimal harvesting-seeding policy in $\mathcal{A}_{x^*}$. Therefore, if the initial population $x\in [0, U]^d$, it is optimal to apply a harvesting-seeding policy so that the population process stays in $[0, U]^d$ forever. This completes the proof.
\end{proof}

\begin{prop}
Suppose we are in the setting of bounded seeding and harvesting rates, and that Assumption \ref{a:1} is satisfied.

{\rm (a)} The value function
$V$ is finite and continuous on $\overline S$.

{\rm (b)} The value function
$V$ is a viscosity subsolution of \eqref{e3.3.17}; that is,
	for
	any $x^0\in S$ and any function $\phi\in C^2(S)$ satisfying
	$$(V-\phi)(x)\ge (V-\phi)(x^0)=0,$$
	for all $x$ in a neighborhood of $x^0$, we have
	\beq{e.3.18}(\L-\delta) \phi(x^0) + \max\limits_{\xi\in [-\lambda, \mu]}\Big[\xi^-\cdot \big(f-\nabla \phi)\(x^0\) - \xi^+ \cdot (g-\nabla \phi)\(x^0\)\Big]\le 0.\eeq

{\rm (c)} The value function
	$V$ is a viscosity supersolution of \eqref{e3.3.17}; that is,
	 for
	 any $x^0\in S$ and any function $\ph\in C^2(S)$ satisfying
	 \beq{e.3.27j}(V-\ph)(x)\le (V-\ph)(x^0)=0,\eeq for all $x$ in a neighborhood of $x^0$, we have
	 \beq{e.3.27k}(\L-\delta) \ph(x^0) + \max\limits_{\xi\in [-\lambda, \mu]}\Big[\xi^-\cdot \big(f-\nabla \ph)\(x^0\) - \xi^+ \cdot (g-\nabla \ph)\(x^0\)\Big]\ge 0.\eeq
	
	 \noindent {\rm(d)} The value function $V$ is a viscosity solution of \eqref{e3.3.17}.	
\end{prop}

In the proof,
we use the following notation and definitions. For a point $x^0\in S$ and a strategy $Q\in \mathcal{A}_{x^0}$,
let ${X}$ be the corresponding process with harvesting and seeding. Let $B_\e(x^0)=\{x\in S: |x-x^0|<\e \}$, where $\e>0$ is sufficiently small so that
$\overline{B_\e(x^0)}\subset S$.
Let $\theta=\inf\{t\ge 0: {X}(t)\notin B_\e(x^0) \}$. For a constant $r>0$, we  define $\theta_r=\theta\wedge r$.

\noindent

\begin{proof}\

\noindent (a) Since the functions $f\cd$, $g\cd$ and the rates  $C\cd$, $R\cd$ are bounded, the value function is also bounded. The conclusion then
follows by \cite[Chapter 3, Theorem 5]{Krylov}.

\noindent (b)
For $x^0\in S$, consider
a $C^2$ function $\phi(\cdot)$ satisfying $\phi(x^0)=V(x^0)$ and
$\phi(x)\le V(x)$ for all $x$ in a neighborhood of $x^0$.
	Let  $\e>0$ be sufficiently small so that
 $\overline{B_\e(x^0)}\subset S$ and
 $\phi(x)\le V(x)$ for all $x\in \overline{B_\e(x_0)}$, where $\overline{B_\e(x_0)}=\{x\in S: |x-x^0|\le \e \}$ is the closure of $B_\e(x^0)$.

 Let $\xi\in [-\mu, \lambda]$ and define
$Q\in \mathcal{A}_{x^0}$ to satisfy $Q(t)=\xi$ for all $t\in [0, r]$ for a positive constant $r$. We denote by ${X}$ the corresponding harvested process with initial condition $x^0$.
Then ${X}(t)\in \overline{B_\e(x^0)}$ for all $0\le t\le  \theta$. By virtue of the dynamic programming principle, we have
	\beq{e.3.19}
	\barray
	\phi(x^0)= V(x^0)
	 \ge \E\bigg[ \int_0^{\theta_r} e^{-\delta s}\Big( Q^-(s)\cdot f\({X}(s)\)- Q^+(s)\cdot g\({X}(s)\)\Big)ds + e^{-\delta \theta_r}   \phi({X}(\theta_r))\bigg].
	\earray
	\eeq
	By the Dynkin formula, we obtain
	\beq{e.3.20}
	\barray
	\phi(x^0)\ad = \E e^{-\delta \theta_r} \phi ({X}(\theta_r)) - \E \int_0^{\theta_r} e^{-\delta s} (\L -\delta ) \phi ({X}(s))ds\\
	\ad \qquad + \E\int_0^{\theta_r} e^{-\delta s}\Big(Q^-(s)\cdot \nabla\phi\({X}(s)\) - Q^+(s)\cdot \nabla\phi\({X}(s)\)\Big)ds.
	\earray
	\eeq
	A combination of \eqref{e.3.19} and \eqref{e.3.20} leads to
	\beq{e.3.21}
	\barray
	0 \ad \ge \E \int_0^{\theta_r} e^{-\delta s}\Big(Q^-(s)\cdot f\({X}(s)\) - Q^+(s)\cdot g\({X}(s)\)\Big)ds 	+ \E \int_0^{\theta_r} e^{-\delta s} (\L -\delta ) \phi ({X}(s))ds\\
	\ad \qquad - \E\int_0^{\theta_r} e^{-\delta s}\Big(Q^-(s)\cdot \nabla\phi\({X}(s)\) - Q^+(s)\cdot\nabla\phi\({X}(s)\)\Big)ds,
	\earray
	\eeq
	which in turn implies
	$$\E\int_0^{\theta_r}e^{-\delta s}\Big[ (\L-\delta) \phi(X(s)) + Q^-(s)\cdot\big(f-\nabla \phi)\(X(s)\) - Q^+(s)\cdot (g-\nabla \phi)\(X(s)\)\Big]ds\le 0.$$
	By the continuity of $X\cd$ and the definition of $Q\cd$, we obtain
	$$(\L-\delta) \phi(x^0) + \xi^-\cdot \big(f-\nabla \phi)\(x^0\) - \xi^+ \cdot (g-\nabla \phi)\(x^0\)\le 0.$$
	This completes the proof of (b).
	
\noindent (c)	Let $x^0\in S$ and suppose $\varphi(\cdot)\in C^2(S)$ satisfies \eqref{e.3.27j}
for all $x$ in a neighborhood of $x^0$.
We argue by contradiction. Suppose that \eqref{e.3.27k} does not hold. Then there exists a constant $A>0$ such that
	\beq{e.3.27m}(\L-\delta) \ph(x^0) + \max\limits_{\xi\in [-\lambda, \mu]}\Big[\xi^-\cdot \big(f-\nabla \ph)\(x^0\) - \xi^+ \cdot (g-\nabla \ph)\(x^0\)\Big]\le -2A< 0.\eeq
Let $\e>0$ be small enough so that $\overline{B_\e(x^0)}\subset S$ and for any $x\in \overline{B_\e (x^0  )}$, $\varphi(x)\ge V(x)$ and
	\beq{e.3.27n}
	(\L-\delta) \ph(x) + \max\limits_{\xi\in [-\lambda, \mu]}\Big[\xi^-\cdot \big(f-\nabla \ph)\(x\) - \xi^- \cdot (g-\nabla \ph)\(x\)\Big]\le -A <0.
	\eeq
		Let
	 $Q\in \mathcal{A}_{x^0}$ and ${X}\cd$ be the corresponding process. Recall that $\theta=\inf\{t\ge 0: {X}(t)\notin B_\e(x^0) \}$ and $\theta_r=\theta\wedge r$ for any $r>0$.
 It follows from the Dynkin formula that
	\beq{e.3.27p}\barray \ad \E e^{-\delta \theta_r} \varphi({X}(\theta_r)-\varphi(x^0)) \\
	\ad \quad = \E \int_0^{\theta_r} e^{-\delta s} \Big[(\L -\delta ) \varphi ({X}(s))
	-  Q^-(s)\cdot \nabla\varphi (X(s)) + Q^+(s)\cdot \nabla \varphi (X(s))\Big]ds\\
	\ad \quad =
\int_0^{\theta_r} e^{-\delta s} \Big[(\L -\delta ) \varphi ({X}(s))
	+Q^-(s) \cdot (f- \nabla \varphi )(X(s))  - Q^+(s)\cdot (g-\nabla \varphi) (X(s))\Big]ds\\
	\ad\quad - \int_0^{\theta_r} e^{-\delta s} \Big[  Q^-(s)\cdot f(X(s)) - Q^+(s)\cdot g(X(s))\Big]ds
	.\earray
	\eeq
		Equations \eqref{e.3.27n} and \eqref{e.3.27p} show that
	\beq{}\barray \ad \E e^{-\delta \theta_r} \varphi({X}(\theta_r))-\varphi(x^0)) \\
	\ad \quad \le  \E \int_0^{\theta_r } e^{-\delta s} (-A)ds - \int_0^{\theta_r} e^{-\delta s} \Big( Q^-(s)\cdot f(X(s)) - Q^+(s)\cdot g(X(s))\Big)ds
.
	\earray
	\eeq
	Therefore
	\beq{e.3.28}
	\barray
	\varphi(x^0) \ad \ge \E e^{-\delta \theta_r} \varphi ({X}(\theta_r))  +  A \E \int_0^{\theta_r} e^{-\delta s} ds\\
	\ad \qquad \qquad  + \int_0^{\theta_r} e^{-\delta s} \Big( Q^-(s)\cdot f(X(s)) - Q^+(s)\cdot g(X(s))\Big)ds.
	\earray
	\eeq	
Letting $r\to \infty$, we have
	\beq{e.3.31b}
	\barray
	V(x^0)=\varphi(x^0) \ad \ge \E e^{-\delta \theta} \varphi ({X}(\theta))  +  A \E \int_0^{\theta} e^{-\delta s} ds\\
	\ad \qquad \qquad  + \int_0^{\theta} e^{-\delta s} \Big( Q^-(s)\cdot f(X(s)) - Q^+(s)\cdot g(X(s))\Big)ds.
	\earray
	\eeq	
Set $\kappa_0 = A \E \int_0^{\theta} e^{-\delta s} ds>0$. Taking the supremum over $Q\in \mathcal{A}_{x^0}$ we arrive at
	\beq{e.3.31d}
	\barray
	V(x^0) \ad \ge \kappa_0 +  \sup\limits_{Q\in \mathcal{A}_{x^0}}\E\bigg[e^{-\delta \theta} \varphi ({X}(\theta))  + \int_0^{\theta} e^{-\delta s} \Big( Q^-(s)\cdot f(X(s)) - Q^+(s)\cdot g(X(s))\Big)ds\bigg].
	\earray
	\eeq	
In view of the dynamic programming principle, the preceding inequality can be rewritten as $V(x^0)\ge V(x_0)+\kappa_0>V(x^0)$, which is a contradiction. This implies that \eqref{e.3.27k} has to hold and the conclusion follows.

\noindent Part (d) follows from (b) and (c).
	\end{proof}

\section{Numerical Algorithm}
\label{sec:alg}

We will present the detailed convergence analysis of Theorem \ref{thm:conv1}, which is closely based on the Markov chain approximation method developed by \cite{Kushner92, Kushner91}.
 Theorem \ref{thm:conv2} and Theorem \ref{thm:conv3} can be derived using similar techniques and we therefore omit the details.

\subsection{Transition Probabilities for bounded seeding and unbounded harvesting rates}
\

For simplicity, we make use of one more assumption below. This assumption will be used to ensure that the transition probabilities $p^h(x, y|u)$ are well defined. Nevertheless, this is not an essential assumption. There are several alternatives to handle the cases when Assumption \ref{a:2} fails. We refer the reader to \cite[page 1013]{Kushner90} for a detailed discussion. Define for any $x\in \lbar S$ the covariance matrix $a(x)= \sg(x)\sg'(x)$.
\begin{asm}\label{a:2}
	For any $i=1, \dots, d$ and $x\in \lbar S$, $$a_{ii}(x)-\sum\limits_{j: j\ne i}\big|a_{ij}(x)\big|\ge 0.$$
\end{asm}

We define the difference
$\Delta X_n^h = X_{n+1}^h-X_{n}^h.$
Denote by $\Delta Y^h_n$  the harvesting amount for the chain at step $n$.
If $\pi^h_n=i$, we let $\Delta Y^h_n=h\ei$ and then $\Delta X^h_n=-h\ei$.
If $\pi^h_n=0$, we set $\Delta Y^h_n=0$. Define
$$Y^h_0=0, \quad Y^h_n = \sum\limits_{m=0}^{n-1}\Delta Y^h_m.$$
For definiteness, if $X^{h}_{n, i}$ is the $i$th component of the vector $X^h_n$ and $\{j: X^{h}_{n, j}=U\}$ is non-empty, then step $n$ is a harvesting step on species $\min\{j: X_{n, j}^{h}=U\}$. Recall that $u^h_n= (\pi^h_n, C^h_n)$ for $n\in\mathbb{Z}_{\geq 0}$ and $u^h=\{u^h_n\}_n\equiv \{Y^h_n, C^h_n\}_n$ is a sequence of controls. It should be noted that $\pi^h_n = 0$
 includes the case when we seed nothing; that is, $C^h_n = 0$.  Denote by
$\mathcal{F}^h_n=\sigma\{X^h_m,u^h_m, m\le n\}$ the $\sigma$-algebra containing the information from the processes $X^h_m$ and $u^h_m$ between the times $0$ and $n$.

The sequence $u^h= (\pi^h, C^h)\equiv \{Y^h_n, C^h_n\}_n$
is said to be admissible if it satisfies the following conditions:
\begin{itemize}
	\item[{\rm (a)}]
	$u^h_n$ is
	$\sigma\{X^h_0, \dots, X^h_{n},u^h_0, \dots, u^h_{n-1}\}-\text{adapted},$
	\item[{\rm (b)}]  For any $x\in S_h$, we have
	$$\P\{ X^h_{n+1} = x | \mathcal{F}^h_n\}= \P\{ X^h_{n+1} = x | X^h_n, u^h_n\} = p^h( X^h_n, x| u^h_n),$$
	\item[{\rm (c)}] Denote by $X^{h}_{n, i}$ the $i$th component of the vector $X^h_n$. Then
	$$ \P\big(
	 \pi^h_{n}=\min\{j: X^{h}_{n, j} = U\}  | X^{h}_{n, j} = U \text{ for some } j\in \{1, \dots, d \}, \mathcal{F}^h_n\big)=1.
	$$
	\item[{\rm (d)}] $X^h_n\in S_h$ for all $n\in\mathbb{Z}_{\geq 0}$.
\end{itemize}
The class of all admissible control sequences $u^h$ having the initial state $x$ will be denoted by
$\mathcal{A}^h_{x}$.

For each
$(x, u)\in S_h\times \mathcal{U}$,
we define
a family of interpolation intervals $\Delta t^h (x, u)$. The values of $\Delta t^h (x, u)$ will be specified later. Then we define
\beq{typo} t^h_0 = 0,\quad  \Delta t^h_m = \Delta t^h(X^h_m, u^h_m),
\quad  t^h_n = \sum\limits_{m=0}^{n-1} \Delta t^h_m.\eeq

Let $\E^{h, u}_{x, n}$, $\Cov^{h, u}_{x, n}$ denote the conditional expectation and covariance given by
$$\{X_m^h, u_m^h, m\le n, X_n^h=x, u^h_n=u \},$$
respectively. Our objective is to define transition probabilities $p^h (x, y | u)$ so that the controlled Markov chain $\{X^h_n\}$ is locally consistent with respect to the controlled diffusion \eqref{e2.2.2}
in the sense that the following conditions hold at seeding steps, i.e., for $u=(0, c)$
\beq{e.4.2}
\barray
\aad \E^{h, u}_{x, n}\Delta X_n^h = \big({b}(x)+c\big)\Delta t^h(x, u)  + o(\Delta t^h(x, u)),\\
\aad Cov^{h, u}_{x, n}\Delta X_n^h = a(x)\Delta t^h(x, u) + o(\Delta t^h(x, u)),\\
\aad \sup\limits_{n, \ \omega} |\Delta X_n^h| \to 0 \quad \text{as}\quad h \to 0.
\earray
\eeq
Using the procedure used by \cite{Kushner90}, for $(x, u)\in S_h\times \mathcal{U}$ with $u=(0, c)$, define
\beq{e.4.7}
\barray
\aad Q_h (x, u)=\sum\limits_{i=1}^d a_{ii}(x) -\sum\limits_{i, j: i\ne j}\dfrac{1}{2}|a_{ij}(x)| +h\sum\limits_{i=1}^d |b_i(x) + c_i| +h,\\
\aad p^h \(x, x+h\ei |u\) =
\dfrac{a_{ii}(x)/2-\sum\limits_{j: j\ne i}|a_{ij}(x )|/2+\big(b_{i}(x)+c_i\big)^+ h }{Q_h (x, u)}, \\
\aad p^h \(x, x-h \ei | u\) =
\dfrac{a_{ii}(x)/2-\sum\limits_{j: j\ne i}|a_{ij}(x )|/2+\big(b_i(x) +c_i)^- h}{Q_h (x, u)}, \\
\aad p^h \( x, x+h\ei +h \ej | u\) =  p^h \( x, x-h\ei -h\ej| u\) =
\dfrac{a_{{ij}}^+(x)}{2Q_h (x, u)}, \\
\aad p^h \( x, x+h\ei -h \ej | u\) = p^h \( x, x-h\ei + h \ej |  u \)=
\dfrac{a_{{ij}}^-(x)}{2Q_h (x, u)}, \\
\aad   p^h \( x, x | u\) =\dfrac{h  }{ Q_h (x, u)},\qquad  \Delta t^h (x, u)=\dfrac{h^2}{Q_h(x, u)}.
\earray
\eeq
Set $p^h \(x, y|u=(0, c)\)=0$ for all unlisted values of $y\in S^h$.
Assumption \ref{a:2} guarantees that
the transition probabilities in \eqref{e.4.7} are well-defined. At the harvesting steps, we define
\beq{e.4.8}
\barray
\aad p^h\( x, x - h \ei| u = (i, c)\)=1, \quad \Delta t^h(x, u=(i, c))=0, \quad   i=1, 2, \dots, d.
\earray
\eeq
Thus, $p^h \(x, y|u=(i, c)\)=0$ for all unlisted values of $y\in S^h$.
Using the above transition probabilities,
we can check that the locally consistent conditions of $\{X^h_n\}$ in \eqref{e.4.2} are satisfied.

\subsection{Continuous--time interpolation and time rescaling}
\

The convergence
 result is based on a continuous-time interpolation of the chain,
which will be  constructed to be piecewise constant on the time interval $[t^h_n, t^h_{n+1}), n\ge 0$.
We define
$n^h(t)=\max\{n: t^h_n\le t\}, t\ge 0$.
We first define discrete time processes associated with the controlled Markov chain as follows. Let
$B^h_0=M^h_0=0$ and define for $n\ge 1$,
\beq{e.4.11}
 B^h_n = \sum\limits_{m=0}^{n-1} I_{\{\pi_m^h=0\}}
\E^{h}_m \Delta\xh_m,\qquad  M^h_n  = \sum\limits_{m=0}^{n-1} (\Delta \xh_m -
\E^{h}_m \Delta X_m)I_{\{\pi_m^h=0\}}.
\eeq
The piecewise constant interpolation processes, denoted by $(X^h\cd, Y^h\cd, B^h\cd, M^h\cd, C^h\cd)$ are naturally defined as
\beq{e.4.12}
\barray
\aad X^h(t) = X^h_{n^h(t)},\quad C^h(t) = C^h_{n^h(t)}, \\ \aad  Y^h(t) = Y^h_{n^h(t)}, \quad B^h(t) = B^h_{n^h(t)}, \quad M^h(t) = M^h_{n^h(t)}, \quad t\ge 0.
\earray
\eeq
Define $\mathcal{F}^h(t)=\sigma\{X^h(s),  Y^h(s), C^h(s): s\le t\}$.
At each step $n$, we can write
\beq{e2.4.1}
\Delta X_n^h =  \Delta X_n^h I_{\{\text{harvesting step at }n\}}+ \Delta X_n^h I_{\{\text{seeding step at }n\}}.
\eeq
Thus, we obtain
\beq{e.4.13}
\barray
X^h_n = x +  \sum\limits_{m=0}^{n-1} \Delta X_m^h I_{\{ \pi^h_m\ge 1\}}+ \sum\limits_{m=0}^{n-1} \Delta X_m^h I_{\{ \pi^h_m=0\}}.
\earray
\eeq
This implies
\beq{e.4.14}
X^h(t)
= x + B^h(t) + M^h(t) -Y^h(t).
\eeq
Recall that $\Delta t^h_m = h^2/Q_h(X^h_m, u^h_m)$ if $\pi^h_m=0$ and $\Delta t^h_m = 0$ if $\pi^h_m\ge 1$. It follows that
\beq{e.4.15}
\barray
B^h(t) \ad = \sum\limits_{m=0}^{n^h(t)-1} \Big[
b (X^h_m) +C^h_m \Big]\Delta t^h_m\\
\ad=\int_0^t \Big[ b (X^h(s)) + C^h(s) \Big] ds-\int_{t^h_{n^h(t)}}^t \Big[ b (X^h(s)) +C^h(s)\Big] ds\\
\ad = \int_0^t \Big[ b (X^h(s))  +C^h(s) \Big]ds + \e^h_1(t),
\earray
\eeq
with $\{\e_1^h\cd\}$
  being an $\mathcal{F}^h(t)$-adapted process satisfying \begin{equation*}\lim\limits_{h\to 0} \sup\limits_{t\in [0, T_0]}\E|\e_1^h(t)|=0 \quad \text{for any }0<T_0<\infty.\end{equation*} We now attempt to represent $M^h\cd$ in a form similar to the diffusion term in \eqref{e2.2.2}.
Factor
$$a(x)= \sg(x)\sg'(x)=P(x)D^2(x)P'(x),$$
where $P\cd$ is an orthogonal matrix, $D\cd=\diag\{r_1\cd, ..., r_d
\cd\}$.
Without loss of generality, we suppose that
$\inf\limits_{x}r_i(x)>0$ for all $i=1, \dots, d$. Define $D_0\cd=\diag\{1/r_1\cd, ..., 1/r_d
\cd\}$.
\begin{rem} In the argument above, for simplicity, we assume that the diffusion matrix $a(x)$ is nondegenerate. If this is not the case, we can use the trick from \cite[p.288-289]{Kushner92} to establish equation  \eqref{e.4.17}. 	
\end{rem}
Define $W^h\cd$ by
\beq{e.4.16}
\barray
W^h(t)  \ad = \int_0^t D_0 (X^h(s))
P' (X^h(s))dM^h(s)\\
\ad = \sum\limits_{m=0}^{n^h(t)-1} D_0 (X^h_m)
P' (X^h_m)(\Delta \xh_m -\E^{h}_m \Delta\xh_m)I_{\{ \pi^h_m=0\}}.
\earray
\eeq
Then we can write
\beq{e.4.17}
M^h(t) =\int_0^t \sg (X^h(s)) dW^h(s) + \e_2^h(t),
\eeq
with $\{\e_2^h\cd\}$
 being an $\mathcal{F}^h(t)$-adapted process satisfying \begin{equation*}\lim\limits_{h\to 0} \sup\limits_{t\in [0, T_0]}\E|\e_2^h(t)|=0 \quad \text{for any }0<T_0<\infty.\end{equation*}
Using  \eqref{e.4.15} and \eqref{e.4.17}, we can write \eqref{e.4.14} as
\beq{e.4.18}
X^h(t) = x + \int_0^t \Big[ b (X^h(s)) +C^h(s) ]ds +  \int_0^t \sg(X^h(s)) dW^h(s) -Y^h(t)+\e^h(t),
\eeq
where $\e^h\cd$ is an $\mathcal{F}^h(t)$-adapted process satisfying \begin{equation*}\lim\limits_{h\to 0} \sup\limits_{t\in [0, T_0]}\E|\e^h(t)|=0 \quad \text{for any }0<T_0<\infty.\end{equation*}
The objective
function from \eqref{e2.4.4} can be rewritten as
\beq{e.4.19}
J^h(x, Y^h, C^h)= \E\Big[ \int_0^{\infty} e^{-\delta s} f \cdot dY^{h}(s)- \int_0^{\infty} e^{-\delta s} g(X^h(s))\cdot C^h(s) d(s)\Big].
\eeq

\noindent{\bf Time rescaling.} Next we will introduce ``stretched-out'' time scale. This is similar to the approach previously used by \cite{Kushner91} and \cite{Budhiraja07} for singular control problems. Using
the new time scale, we can overcome the possible non-tightness of the family of processes $\{Y^h\cd\}$.

Define the rescaled
time increments $\{\Delta \wdh{t}_n^h: n\in\mathbb{Z}_{\geq 0}\}$ by
\beq{e.4.20}
\barray
\aad \Delta \wdh{t}^h_n = \Delta t^h_n I_{\{ \pi^h_n =0 \}} + h I_{\{ \pi^h_n \ge 1 \}}, \qquad  \wdh{t}_0=0, \qquad\wdh{t}_n = \sum\limits_{k=0}^{n-1}\Delta \wdh{t}^h_k, \quad n\ge 1.\\
\earray
\eeq

\begin{defn}\label{def:1}{\rm
		The rescaled time process
		$\wdh{T}^h\cd$ is the unique continuous nondecreasing process satisfying the following:
		\begin{itemize}
			\item[(a)] $\wdh{T}^h(0)=0$;
			\item[(b)] the derivative of $\wdh{T}^h\cd$ is 1 on $(\wdh{t}^h_n, \wdh{t}^h_{n+1})$ if $\pi^h_n=0$, i.e., $n$ is a seeding step;
			\item[(c)] the derivative of $\wdh{T}^h\cd$ is 0 on $(\wdh{t}^h_n, \wdh{t}^h_{n+1})$ if $\pi^h_n\ge 1$, i.e., $n$ is a harvesting step.	
		\end{itemize}	
	}
\end{defn}
Define the rescaled and interpolated process $\wdh{X}^h(t)= X^h(\wdh{T}^h(t))$ and likewise define
 $\wdh{Y}^h\cd$, $\wdh{C}^h\cd$, $\wdh{B}^h\cd$, $\wdh{M}^h\cd$, and the filtration $\wdh{\mathcal{F}}^h\cd$ similarly.
It follows from \eqref{e.4.14} that
\beq{e.4.21}
\wdh{X}^h(t)=x+\wdh{B}^h(t) + \wdh{M}^h(t) - \wdh{Y}^h(t).
\eeq
Using the same argument we used for \eqref{e.4.18}
we obtain \beq{e.4.22}
\wdh{X}^h(t) = x + \int_0^t \Big[ b (\wdh{X}^h(s)) + \wdh{C}^h(s) \Big]d\wdh{T}^h(s) + \int_0^t \sg(\wdh{X}^h(s)) d \wdh{W}^h(s)  - \wdh{Y}^h(t) + \wdh{\e}^h(t),
\eeq
with $\wdh{\e}^h\cd$ is an $\wdh{\mathcal{F}}^h\cd$-adapted process satisfying \beq{e.4.23}\lim\limits_{h\to 0} \sup\limits_{t\in [0, T_0]}\E|\wdh{\e}^h(t)|=0 \quad \text{for any }0<T_0<\infty.\eeq

Define
\beq{e.4.aa}{A}^h(t)=\int_0^t  {C}^h(s) ds, \quad \wdh{A}^h(t)=\int_0^t  \wdh{C}^h(s) \wdh{T}^h(s), \quad t\ge 0, h> 0.\eeq

\subsection{Convergence}
\

Using weak convergence methods, we can obtain the convergence of the algorithms. Let $D[0, \infty)$ denote the space of functions that are right continuous and have left-hand limits endowed with the Skorohod topology. All the weak analysis
will be on this space or its $k$-fold products $D^k[0, \infty)$ for appropriate $k$.

\begin{thm}\label{thm:thm} Suppose Assumptions \ref{a:1} and \ref{a:2} hold.
	Let the chain $\{X^h_n \}$ be constructed with transition probabilities defined in \eqref{e.4.7}-\eqref{e.4.8},
	$X^h\cd$, $W^h\cd$, $Y^h\cd$, and $A^h\cd$ be the continuous-time interpolation defined in \eqref{e.4.11}-\eqref{e.4.12}, \eqref{e.4.16}, and \eqref{e.4.aa}.
Let $\wdh{X}^h\cd$,  $\wdh{W}^h\cd$, $\wdh{Y}^h\cd$, $\wdh{A}^h\cd$ be the corresponding rescaled processes, $\wdh{T}^h\cd$ be the process from
Definition {\rm\ref{def:1}}, and denote
$$\wdh{H}^h\cd=\Big(\wdh{X}^h\cd, \wdh{W}^h\cd, \wdh{Y}^h\cd, \wdh{A}^h\cd, \wdh{T}^h\cd\Big).$$
Then the family of processes
	$(\wdh{H}^h)_{h>0}$ is tight. 	
	As a result, $(\wdh{H}^h)_{h>0}$ has a weakly convergent subsequence with limit $$\wdh{H}\cd=\Big(\wdh{X}\cd,\wdh{W}\cd, \wdh{Y}\cd, \wdh{A}\cd, \wdh{T}\cd\Big).$$
\end{thm}

\begin{proof}
		We use the tightness criteria used by \cite[p. 47]{Kushner84}. Specifically,  a sufficient condition for tightness of a sequence of processes  $\zeta^h\cd$ with paths in $D^k[0, \infty)$ is that for any constants $T_0, \rho
		\in (0, \infty)$,
		\bea
		\ad  \E_t^h\big|\zeta^h(t+s)-\zeta^h(t)\big|^2\le \E^h_t \gamma(h, \rho) \quad \text{for all}\quad s\in [0, \rho], \quad t\le T_0,\\
		\ad \lim\limits_{\rho\to 0}\limsup\limits_{h\to 0} \E  \gamma(h, \rho) =0.
		\eea
			The proof for the tightness of $\wdh{W}^h\cd$ is standard; see for example \cite{Kushner91, Jin12}.  We show the tightness of
			$\wdh{Y}^h\cd$ to demonstrate the role of time rescaling.
	Following the definition of ``stretched out" timescale, for any constants $T_0, \rho
	\in (0, \infty)$, $s\in [0, \rho]$ and $t\le T_0$,
	\beq{e:23}
	\barray
	\E^h_{t}|\wdh{Y}^h(t+s) - \wdh{Y}^h(t)|^2\ad \le d h^2 \E^h_{t}(\text{number of harvesting steps in}\\
	\aad \hspace{3cm} \text{interpolated interval } [t, t+s) )^2\\
	\ad\le d h^2 \max\{1, \rho^2/h^2 \} \\
	\ad \le d (h^2 + \rho^2).
	\earray
	\eeq
	Thus $\{\wdh{Y}^h\cd\}$ is tight. The tightness of $\{\wdh{T}^h\cd\}$ follows from the fact that
	$$0\le \wdh{T}^h(t+s)-\wdh{T}^h(t)\le \rho.$$	
	Since $|\wdh{A}^h(t+s)-\wdh{A}^h(t)|\le |\wdh{T}^h(t+s)-\wdh{T}^h(t)|\sum_{i=1}^d \lambda
	_i$, it follows that $\{\wdh{A}^h\cd\}$ is tight.
	The tightness of  $\{\wdh{X}^h\cd\}$ follows from \eqref{e.4.21}, \eqref{e:23}. Hence $\{\wdh{X}^h\cd, \wdh{W}^h\cd, \wdh{Y}^h\cd, \wdh{A}^h\cd, \wdh{T}^h\cd\}$ is tight. By virtue of Prohorov's Theorem,
	$\wdh{H}^h\cd$ has a weakly convergent subsequence with the limit  $\wdh{H}\cd$. This completes the proof.
\end{proof}

We proceed to characterize the limit process.

\begin{thm}\label{thm:thm4.4}  Under conditions of Theorem \ref{thm:thm},
	let $\wdh{\mathcal{F}}(t)$ be the $\sigma$-algebra generated by
	$$\{\wdh{X}(s), \wdh{W}(s), \wdh{Y}(s), \wdh{A}(s), \wdh{T}(s):s \le t\}.$$
	Then the following assertions hold.
	\begin{itemize}
		\item[\rm (a)] $\wdh{X}\cd$, $\wdh{W}\cd$, $\wdh{Y}\cd$, $\wdh{A}\cd$, and $\wdh{T}\cd$  have continuous paths with probabilty one,  $\wdh{Y}\cd$ and $\wdh{T}\cd$ are  nondecreasing and nonnegative. Moreover, $\wdh{T}\cd$ is Lipschitz continuous with Lipschitz coefficient 1.
		\item[\rm (b)] There exists an $\{\wdh{\mathcal{F}}\cd\}$-adapted process $\wdh{C}\cd$ with $\wdh{C}(t)\in  [0, \lambda]$ for any $t\ge 0$, such that $\wdh{A}(t)=\int_0^t \wdh{C}(s)d\wdh{T}(s)$ for any $t\ge 0$.
		\item[\rm (c)] $\wdh{W}(t)$ is an $\wdh{\mathcal{F}}(t)$-martingale with quadratic variation process $\wdh{T}(t)I_d$, where $I_d$ is the $d\times d$ identity matrix.
			\item[\rm (d)] The limit processes satisfy
		\beq{e.5.1}
		\wdh{X}(t) = x + \int_0^t \big[ b (\wdh{X}(s)) + \wdh{C}(s) \big]d\wdh{T}(s) + \int_0^t \sg(\wdh{X}(s)) d \wdh{W}(s)  - \wdh{Y}(t).	 \eeq
	\end{itemize}
\end{thm}

\begin{proof}
	(a)
	Since the sizes of the jumps of $\wdh{X}^h\cd$, $\wdh{W}^h\cd$, $\wdh{Y}^h\cd$, $\wdh{A}^h\cd$, $\wdh{T}^h\cd$  go to $0$ as $h\to 0$, the limits of these processes have continuous paths with probability one (see \cite[p. 1007]{Kushner90}). Moreover, $\wdh{Y}^h\cd$  (resp. $\wdh{T}^h\cd$) converges uniformly to $\wdh{Y}\cd$, (resp. $\wdh{T}\cd$) on bounded time intervals. This, together with
	the monotonicity and non-negativity
	of $\wdh{Y}^h\cd$ and $\wdh{T}^h\cd$
	implies that the processes $\wdh{Y}\cd$ and $\wdh{T}\cd$ are nondecreasing and nonnegative.

	 (b) Since $|\wdh{A}^h_i(t+s)-\wdh{A}^h_i(t)|\le \lambda_i |\wdh{T}^h(t+s)-\wdh{T}^h(t)|$
	for any $t\ge 0, s\ge 0, h>0, i=1, 2,\dots, d$ and by virtue of Skorohod representation,
	$|\wdh{A}_i(t+s)-\wdh{A}_i(t)|\le \lambda_i |\wdh{T}(t+s)-\wdh{T}(t)|$
	for any $t\ge 0, s\ge 0, i=1, 2,\dots, d$; that is, each $\wdh{A}_i$ is absolutely continuous with respect to $\wdh{T}$. Therefore, there exists a $[0,\lambda_i]$-valued $\{\wdh{\mathcal{F}}(t)\}$-adapted process  $\wdh{C}_i\cd$ such that $\wdh{A}_i(t)=\int_0^t \wdh{C}_i(s)d\wdh{T}(s)$ for any $t\ge 0$. Then $C\cd=(C_1\cd, \dots, C_d\cd)'$
	is the desired process.

	 (c)
	Let $\wdh{\E}_t^h$ denote the expectation conditioned on $\wdh{\mathcal{F}}^h(t)=\mathcal{F}^h(\wdh{T}^h(t))$.
	Recall that $W^h\cd$ is an $\mathcal{F}^h\cd$- martingale and by the definition of $\wdh{W}^h\cd$, for any $\rho>0$,
	\beq{e.5.2.1}
	\barray
	\aad \wdh{\E}_t^h \big(\wdh{W}^h(t+\rho)-\wdh{W}^h(t)\big)=0,\\
	\aad\wdh{\E}_t^h \big(\wdh{W}^h(t+\rho)\wdh{W}^h(t+\rho)'-\wdh{W}^h(t)\wdh{W}^h(t)'\big)
	= \big(\wdh{T}^h(t+\rho)-\wdh{T}^h(t)\big)I_d +  \wdh{\e}^h(\rho),
	\earray
	\eeq
	where $\E|\wdh{\e}^h(\rho)|\to 0$ as $h\to 0$.
	To characterize $\wdh{W}\cd$, let $q$ be an arbitrary integer, $t>0$, $\rho>0$ and $\{t_k: k\le q\}$ be such that $t_k\le t<t+\rho$ for each $k$.
	Let $\Psi\cd$ be a real-valued and continuous function with compact support. Then in view of \eqref{e.5.2.1}, we have
	\beq{e:251}
	\E\Psi(\wdh{H}^h(t_k), k\le q)\Big[ \wdh{W}^h(t+\rho)-\wdh{W}^h(t)\Big]=0,
	\eeq
	and
	\beq{e:2513}
	\E\Psi(\wdh{H}^h(t_k), k\le q)\Big[ \big(\wdh{W}^h(t+\rho)\wdh{W}^h(t+\rho)'-\wdh{W}^h(t)\wdh{W}^h(t)'-\big(\wdh{T}^h(t+\rho)-\wdh{T}^h(t)\big)I_d-\wdh{\e}^h(\rho)\Big]=0.
	\eeq
	By the Skorokhod representation and the dominated convergence theorems, letting $h\to 0$ in \eqref{e:251}, we obtain
	\beq{e:252}
	\E\Psi(\wdh{H}(t_k), k\le q)\Big[ \wdh{W}(t+\rho)-\wdh{W}(t)\Big]=0.
	\eeq
	Since $\wdh{W}\cd$  has continuous paths with probability one,
	\eqref{e:252} implies that $\wdh{W}\cd$ is a continuous $\wdh{\mathcal{F}}\cd$-martingale. Moreover, \eqref{e:2513} gives us that
	\beq{}
	\E\Psi(\wdh{H}(t_k), k\le q)\Big[ \wdh{W}(t+\rho)\wdh{W}(t+\rho)'-\wdh{W}(t)\wdh{W}(t)'-\big(\wdh{T}(t+\rho)-\wdh{T}(t)\big)I_d\Big]=0.
	\eeq
	This implies part (c).

	(d) The proof of this part is motivated by that of \cite[Theorem 10.4.1]{Kushner92}.
	By virtue of Skorohod representation,
	\beq{e:27}
	\int_0^t  \Big[ b (\wdh{X}^h(s)) + \wdh{C}^h(s) \Big]d\wdh{T}^h(s)
	\to  \int_0^t  \Big[ b (\wdh{X}(s)) + \wdh{C}(s) \Big]d\wdh{T}(s),
	\eeq
	as $h\to 0$
	uniformly in $t$ on any  bounded time interval with probability one.

	For each positive constant $\rho$ and a process $\wdh{\nu}\cd$, define the piecewise constant process $\wdh{\nu}^\rho\cd$
	by
	$\wdh{\nu}^\rho(t)=\wdh{\nu}(k\rho)$ for $t\in [k\rho, k\rho+\rho), k\in\mathbb{Z}_{\geq 0}$. Then, by the tightness of
	$(\wdh{X}^h\cd)$,
	\eqref{e.4.22} can be rewritten as
	\beq{e:28}
	\wdh{X}^h(t) = x_0 + \int_0^t \Big[b (\wdh{X}^h(s)) + \wdh{C}^h(s)\Big] d\wdh{T}^h(s) + \int_0^t \sg(\wdh{X}^{h, \rho}(s)) d \wdh{W}^h(s)  - \wdh{Y}^h(t) + \wdh{\e}^{h, \rho}(t),
	\eeq
	where
	$\lim\limits_{\rho\to 0}\limsup\limits_{h\to 0} \E|\wdh{\e}^{h, \rho}(t)|=0.$
	Owing to the fact that
	$\wdh{X}^{h, \rho}$
	takes constant values
	on the intervals $[k\rho, k\rho+\rho)$, we have
	\beq{e:29}\int_0^t  \sg(\wdh{X}^{h, \rho}(s))d\wdh{W}^h(s)\to \int_0^t  \sg(\wdh{X}^\rho(s))d\wdh{W}(s) \quad \text{ as }\quad  h\to 0,
	\eeq
	which are well defined with probability one since they can be written as finite sums.
	Combining \eqref{e:27}-\eqref{e:29}, we have
	\beq{}
	\wdh{X}(t)=x_0 +\int_0^t \Big[b(\wdh{X}(s))+\wdh{C}(s)\Big]d\wdh{T}(s)+\int_0^t  \sg(\wdh{X}^\rho(s))d\wdh{W}(s) - \wdh{Y}(t)+\wdh{\e}^{\rho}(t),
	\eeq
	where
	$\lim\limits_{\rho\to 0}E|\wdh{\e}^{ \rho}(t)|=0.$ Taking the limit $\rho \to 0$ in the above equation yields the result.
\end{proof}

For $t<\infty$, define the inverse ${\lbar T}(t)= \inf\{s: \wdh{T}(s)>t\}$. For any process $\wdh{\nu}\cd$, define the time-rescaled process $(\lbar\nu\cd)$ by $\lbar \nu(t)= \wdh{\nu}({\lbar T}(t))$ for $t\ge 0$. Let ${\mathcal{\lbar F}}(t)$ be the $\sigma$-algebra generated by
	$\{\lbar X(s), {\lbar W}(s), {\lbar Y}(s), {\lbar C}(s), {\lbar T}(s): s\le t\}$. Let $V^h(x)$ and $V^U(x)$ be value the functions defined in \eqref{e2.4.5} and \eqref{e:VU}, respectively.

\begin{thm}\label{thm:r}
	 Under conditions of Theorem \ref{thm:thm}, the following assertions are true.
	\begin{itemize}
		\item[\rm(a)] $\lbar T$ is right continuous, nondecreasing, and $\lbar T(t)\to \infty$ as $t \to \infty$ with probability one.
		\item[\rm(b)]  The processes $\lbar Y(t)$ and $\lbar C(t)$ are $\mathcal{\lbar F}(t)$-adapted. Moreover, $\lbar Y(t)$ is right-continuous, nondecreasing, nonnegative; $\lbar{C}(t)\in [0, \lambda]$ for any $t\ge 0$.
		\item[\rm(c)]
		$\lbar W\cd$ is an $\mathcal{\lbar F}(t)$-adapted standard Brownian motion,
		and \beq{e.5.2}
		{\lbar X}(t)=x +\int_0^t \Big[b(\lbar X(s))   +\overline{C}(s)\Big] ds+\int_0^t  \sg(\lbar X(s))d\lbar W(s)-\lbar Y(t), \quad t\ge 0.
		\eeq
	\end{itemize}
\end{thm}

\begin{proof}
	(a)
	We will argue via contradiction that $\wdh{T}(t)\to \infty$ as $t\to \infty$ with probability one. Suppose $\P[\sup_{t\ge 0}\wdh{T}(t)<\infty]>0$. Then there exist positive constants $\e$ and $T_0$ such that
	\beq{e:300}
	\P[\sup\limits_{t\ge 0}\wdh{T}(t)<T_0-1]>\e.
	\eeq	
	We first observe that
	$$t+d|Y^h(t)|\ge \sum\limits_{k=0}^{n^h(t)-1}\Big(\Delta t^h_n I_{\{\pi^h_k=0\}} + h I_{\{\pi^h_k\ge 1\}}\Big).$$
	Since $\wdh{T}^h\cd$
	is nondecreasing and
	$\wdh{T}^h(\wdh{t}^h_n)=t^h_n$,
	\beq{e:301}
	\barray
	\wdh{T}^h\big(t+d|Y^h(t)|\big) \ad\ge  \wdh{T}^h \Big(\sum\limits_{k=0}^{n^h(t)-1}\big( \Delta t^h_k I_{\{\pi^h=0\}} + hI_{\{\pi^h_k\ge 1\}}\big)\Big)\\
	\ad = \wdh{T}^h(\wdh{t}^h_{n^h(t)})={t}^h_{n^h(t)}\ge t-1.
	\earray
	\eeq	
	The last inequality above is a consequence of the inequalities $t^h_{n^h(t)}\le t< t^h_{n^h(t)+1}=t^h_{n^h(t)}+\Delta t^h_{n+1} <t^h_{n^h(t)}+1$.
	
	It follows from \eqref{e.4.14} that for each fixed $t\ge 0$,
	$\sup\limits_{h}\E\big(|Y^h(t)|\big)<\infty.$
	Thus, for a sufficiently large
	$K$,
	\beq{e:302}
	\P\{ d|Y^h(T_0)| \ge 2K \} \le\dfrac{d\E\big|Y^h(T_0)\big|}{2K}< \dfrac{\e}{2}.
	\eeq
	In views of \eqref{e:301} and \eqref{e:302}, we obtain
	\beq{e:3022}
	\barray
	\P\big[\wdh{T}^h(T_0+2K) <T_0-1 \big]\ad \le \P\big[\wdh{T}^h\big(T_0+d|Y^h(T_0)\big) <T_0-1 , d|Y^h(T_0)|<2K\big]\\
	\ad \quad + \P\big[ d|Y^h(T_0)|\ge 2K\big]\\
	\ad < \dfrac{\e}{2} \qquad \text{for small }h.
	\earray
	\eeq
	Since $\wdh{T}^h$ converges weakly  to $\wdh{T}$, it follows from \eqref{e:3022} that
	$\liminf\limits_{h\to 0} \P\big[\wdh{T}^h(T_0+2K) <T_0-1 \big]\le \e/2$. This contradicts \eqref{e:300} (see \cite[Theorem 1.2.1]{B68}).
	Hence $\wdh{T}(t)\to \infty$ as $t\to \infty$ with probability one.	
	Thus ${\lbar T}(t)<\infty$ for all $t$ and ${\lbar T}(t)\to \infty$ as $t\to\infty$. Since $\wdh{T}\cd$ is nondecreasing and continuous,
	${\lbar T}\cd$ is nondecreasing and right-continuous.
	
	(b) The properties of $\lbar{Y}\cd$ follow from the fact that $\wdh{Y}\cd$ is continuous, nondecreasing, nonnegative, and ${\lbar T}\cd$ is right-continuous.  The properties of $\lbar{C}\cd$ follow from those of  $\wdh{C}\cd$.

	(c) Note that although ${\lbar T}\cd$ might fail to be continuous, $\lbar{W}\cd=\wdh{W}({\lbar T}\cd)$ has continuous paths with probability one. Indeed, consider the tight sequence $\big({W}^h\cd, \wdh{W}^h\cd, \wdh{T}^h\cd\big)$
	with the weak limit
	$\big(\wdt{W}\cd, \wdh{W}\cd, \wdh{T}\cd\big)$.
	Since $\wdh{W}^h\cd=W^h(\wdh{T}^h\cd)$, we must have that $\wdh{W}\cd=\wdt{W}(\wdh{T}\cd)$.
	It follows from the definition of ${\lbar T}\cd$ that  for each $t\ge 0$, we have $\wdh{T}({\lbar T}(t))=t$. Hence $\lbar{W}(t)=\wdh{W}({\lbar T}(t))=\wdt{W}\big(\wdh{T}({\lbar T}(t))\big)=\wdt{W}(t)$.
	Since the sizes of the jumps of $W^h\cd$ go to $0$ as $h\to 0$, $\wdt{W}\cd$ also has continuous paths with probability $1$. This shows that $\lbar{W}\cd=\wdh{W}({\lbar T}\cd)$ has continuous paths with probability $1$.
	Before characterizing $\lbar{W}\cd$, we note that for $t\ge 0$, $\{{\lbar T}(s)\le t\}=\{\wdh{T}(t)\ge s\}\in \wdh{\mathcal{F}}(t)$ since  $\wdh{T}(t)$ is
	$\wdh{\mathcal{F}}(t)$-measurable. Thus ${\lbar T}(s)$ is an $\wdh{\mathcal{F}}(t)$-stopping time for each $s\ge 0$. Since $\wdh{W}(t)$
	is an $\wdh{\mathcal{F}}(t)$-martingale with quadratic variation process $\wdh{T}(t) I_d$,
	\beq{e:3023}\barray \ad  \E\big[\wdh{W}({\lbar T}(t)\wedge n)| \wdh{\mathcal{F}}({\lbar T}(s))\big]=\wdh{W}({\lbar T}(s)\wedge n), \quad n=1,2, \dots, \\
	\ad \E \wdh{W}({\lbar T}(t)\wedge n)\wdh{W}({\lbar T}(t)\wedge n)'=\E\wdh{T}({\lbar T}(t)\wedge n)I_d,
	\earray
	\eeq
	and
	$\wdh{T}({\lbar T}(t)\wedge n)\le \wdh{T}({\lbar T}(t))=t$.
	Hence for each fixed $t\ge 0$, the family $\{\wdh{W}({\lbar T}(t)\wedge n), n\ge 1\}$ is uniformly integrable. By that uniform integrability, we obtain from
	\eqref{e:3023} that
	$E\big[\wdh{W}({\lbar T}(t))| \wdh{\mathcal{F}}({\lbar T}(s))\big]=\wdh{W}({\lbar T}(s))$, that is
	$E\big[\lbar{W}(t)| \lbar{\mathcal{F}}(s)\big]=\lbar{W}(s)$.
	This proves that
	$\lbar{W}\cd$
	is a continuous $\lbar{\mathcal{F}}\cd$ -martingale. We next consider its quadratic variation. By the Burkholder-Davis-Gundy inequality, there exists a positive constant $K$ independent of $n=1, 2, ...$ such that
	$$\E|\wdh{W}({\lbar T}(t)\wedge n)|^2\le K\E\bigg[\Big(\sup\limits_{0\le s\le {\lbar T}(t)}|\wdh{W}({\lbar T}(s)\wedge n)|^2 \Big)\bigg]\le K\E|\wdh{T}({\lbar T}(t)\wedge n)|\le Kt.$$
	Thus the families $\{\wdh{W}({\lbar T}(t)\wedge n), n\ge 1\}$
	and $\{\wdh{T}({\lbar T}(t)\wedge n), n\ge 1\}$ are uniformly integrable for each fixed $t\ge 0$.
	Combining this with
	the fact that $\wdh{W}\cd$, $\wdh{T}\cd$ have continuous paths,
	for nonnegative constants $s\le t$,
	we have
	\beq{e:3024}
	\barray
	\wdh{W}({\lbar T}(s)\wedge n)\wdh{W}({\lbar T}(s)\wedge n)'\ad -\wdh{T}({\lbar T}(s)\wedge n)I_d\\
 \ad =
	\E\big[ \wdh{W}({\lbar T}(t)\wedge n)\wdh{W}({\lbar T}(t)\wedge n)'-\wdh{T}({\lbar T}(t)\wedge n)I_d | \wdh{\mathcal{F}}({\lbar T}(s))\big]\\
	\ad  \to \E\big[ \wdh{W}({\lbar T}(t))\wdh{W}({\lbar T}(t))'-\wdh{T}({\lbar T}(s))I_d | \wdh{\mathcal{F}}({\lbar T}(s))\big]\\
	\ad = \E\big[ \lbar{W}(t)\lbar{W}(t)'-t I_d | \lbar{\mathcal{F}}(s)\big].
	\earray\eeq
	Note that the first equation in \eqref{e:3024}
	follows from the martingale property of
	$\wdh{W}\cd\wdh{W}\cd'-\wdh{T}\cd I_d$
	with respect to $\wdh{\mathcal{F}}(t).$
	Letting $n\to\infty$ in \eqref{e:3024}, we arrive at
	$$\E\big[ \lbar{W}(t)\lbar{W}(t)'-tI_d | \mathcal{F}(s)\big]=\lbar{W}(s)\lbar{W}(s)'-s I_d.$$ Therefore, $\lbar{W}\cd$
	is an $\lbar{\mathcal{F}}(t)$ - adapted standard Brownian motion.
	A rescaling of \eqref{e.5.1} yields
	\begin{equation*}
	\lbar{X}(t)=x +\int_0^t \Big[b(\lbar{X}(s))  +\lbar C(s)\Big]ds+\int_0^t  \sg(\lbar{X}(s))d\lbar{W}(s)-\lbar{Y}(t).
	\end{equation*} The proof is complete.
\end{proof}

\begin{thm}\label{thm:4.6} Under conditions of Theorem \ref{thm:thm},
	let $V^h(x)$ and $V^U(x)$ be value functions defined in \eqref{e2.4.5}  and \eqref{e:VU}, respectively. Then $V^h(x)\to V^U(x), x\in[0,U]^d$ as $h\to 0$. If \eqref{e2.5} holds, then $V^h(x)\to V(x), x\in[0,U]^d$ as $h\to 0$.
\end{thm}

\begin{proof}
	We first show that as $h\to 0$,
	\beq{e:32}
	J^h(x, u^h) \to
	J(x, \lbar{Y}\cd, \lbar C\cd),
	\eeq
	where $u^h=(\pi^h, C^h)$.
	Indeed, for an admissible strategy $u^h=(\pi^h_n, C^h_n)$, we have
	\beq{e:33}
	\barray
	J^h (x, u^h) \ad
	=  \E\bigg[\sum_{m=1}^{\infty} e^{-\delta t_m^h} f\cdot \Delta Y_{m}^{h}- \sum_{m=1}^{\infty} e^{-\delta t_m^h} g(X^h_m) \cdot C^h_m \Delta t^h_m\bigg].\\
	\ad =  \E\Big[ \int_0^{\infty} e^{-\delta\wdh{T}^h(t)} f\cdot d\wdh{Y}^{h}(t)-\int_0^{\infty} e^{-\delta\wdh{T}^h(t)} g(\wdh{X}^h(t))\cdot \wdh{C}^{h}(t)d\wdh{T}^h(t)\Big].
	\earray
	\eeq
	By a small modification of the proof in Theorem \ref{thm:r} (a), we have $\wdh{T}^h(t)\to \infty$ as $t\to \infty$ with probability $1$.
	It also follows from the representation \eqref{e.4.14} and estimates on $B^h\cd$ and $M^h\cd$ that $\{Y^h(n+1)-Y^h(n): n, h\}$ is uniformly integrable. Thus, by the definition of $\wdh{T}^h\cd$,
	\begin{equation*}
	\barray
	\E \int_{T_0}^{\infty }e^{-\delta\wdh{T}^h(t)} f\cdot d\wdh{Y}^{h}(t) \ad \le  \E \int_{\min \{t:\wdh{T}^h(t)\ge T_0 \}}^{\infty} K e^{-\delta s}\cdot d{Y}^{h}(s)\\
	\ad \le  \E \int_{T_0 }^{\infty} K e^{-\delta s}\cdot d{Y}^{h}(s)\to 0,
	\earray
	\end{equation*}
	uniformly in $h$ as $T_0\to \infty$. In the above argument, we have used that  $\wdh{T}^h(T_0)\le T_0$. Then by the weak convergence, the Skohorod representation, and uniform integrability we have for any $T_0>0$ that
	$$ \E\int_0^{T_0}  e^{-\delta\wdh{T}^h(t)} f\cdot d\wdh{Y}^{h}(t)\to
	\E \int_0^{T_0} e^{-\delta\wdh{T}(t)} f\cdot d\wdh{Y}(t).$$ Therefore, we obtain
	$$ \E\int_0^{\infty}  e^{-\delta\wdh{T}^h(t)} f\cdot d\wdh{Y}^{h}(t)\to
	\E \int_0^{\infty} e^{-\delta\wdh{T}(t)} f\cdot d\wdh{Y}(t).$$
	 Similarly,
	 $$ \E\int_0^{\infty}  e^{-\delta\wdh{T}^h(t)} g(\wdh{X}^h(t))\cdot \wdh C^h(t) d\wdh{T}^{h}(t)\to
	 \E \int_0^{\infty} e^{-\delta\wdh{T}(t)} g(\wdh{X}(t))\cdot \wdh{C}(t) d\wdh{T}(t).$$
	On inversion of the timescale, we have
	$$J^h(x, u^h)\to  \E\Big[\int_0^\infty  e^{-\delta t} f\cdot d\lbar{Y}(t)-\int_0^\infty  e^{-\delta t} g(
	 \lbar{X}(t))\cdot d\lbar{C}(t)dt\Big].$$
	Thus, $J^h(x, u^h)\to J(x, \lbar{Y}\cd, \lbar{C}\cd)$ as $h\to 0$.

	Next, we prove that
	\beq{e:34}
	\limsup\limits_{h} V^h(x) \le V^U(x).
	\eeq
	For any small positive constant $\e$, let $\{\wdt u^h\}$ be an $\e$-optimal harvesting strategy for the chain $\{X^h_n\}$;
	that is,
	$$V^h(x)=\sup\limits_{u^h} J^h(x, u^h)\le J^h(x, \wdt{u}^h) + \e.$$
	Choose a subsequence $\{\wdt{h}\}$ of $\{h\}$
	such that
	\beq{e:35}\limsup\limits_{{h}\to 0} V^{{h}}(x)=\lim\limits_{\wdt{h}\to 0}V^{\wdt{h}} (x)\le\limsup\limits_{\wdt{h}\to 0} J^{\wdt{h}}(x, {\wdt{u}}^{\wdt{h}})+\e.\eeq
	Without loss of generality (passing to an additional subsequence if needed), we may assume that
	$$\wdh{H}^{\wdt{h}}\cd
	= \Big(\wdh{X}^{\wdt{h}}\cd, \wdh{W}^{\wdt{h}}\cd, \wdh{Y}^{\wdt{h}}\cd, \wdh{A}^{\wdt{h}}\cd, \wdh{T}^{\wdt{h}}\cd\Big)$$
	converges weakly to
	$$\wdh{H}\cd= \Big(\wdh{X}\cd,  \wdh{W}\cd, \wdh{Y}\cd, \wdh{A}\cd, \wdh{T}\cd\Big),$$
	and $\lbar Y\cd=\wdh{Y}(\lbar T\cd)$, $\lbar A\cd=\wdh{A}(\lbar T\cd)$, $\lbar C\cd=\wdh{C}(\lbar T\cd)$.
	It follows from our claim in the beginning of the proof that
	\beq{e:36}
	\lim\limits_{\wdt{h}\to 0} J^{\wdt{h}}(x, {\wdt{u}}^{\wdt{h}})= J(x,  \lbar Y\cd, \lbar C\cd )\le V^U(x),
	\eeq
	where	$J(x, \lbar Y\cd, \lbar C\cd)\le V^U(x)$ since $V^U(x)$ is the maximizing performance function.
	Since $\e$ is arbitrarily small, \eqref{e:34} follows from \eqref{e:35} and \eqref{e:36}.
	
	To prove the reverse inequality
	$\liminf\limits_{h} V^h(x)\ge V^U(x)
	$, for any small positive constant $\e$,
	we choose a particular $\e$-optimal harvesting strategy for \eqref{e2.2.2} such that
	the approximation can be applied
	to the chain $\{X^h_n\}$
	and the associated reward compared with $V^h(x)$. By an adaptation of the method used by \cite{Kushner91} for singular control problems,
	for given $\e>0$, there is a $\e$-optimal harvesting strategy
	$({Y}\cd, {C}\cd)$
	for \eqref{e2.2.2} in $\mathcal{A}_x^U$ with the following properties: There are $T_\e<\infty$, $\rho>0$, and $\lambda>0$ such that $( {Y}\cd,  {C}\cd)$
	are constants on the intervals $[n\lambda, n\lambda + \lambda)$;  only one of the components of $ Y\cd$  can jump at a time and the jumps take values in the discrete set $\{k\rho
	: k=1, 2, ...\}$;  ${Y}\cd$ is bounded and is constant on $[T_\e, \infty)$; and $C\cd$ takes only finitely many values.
	
We adapt this strategy to the chain $\{X^h_n\}$  by a sequence of controls  $u^h\equiv (Y^h,C^h)$ using the same method as in \cite[p. 1459]{Kushner91}.  Suppose that we wish to apply a harvesting action of ``impulsive'' magnitude $\Delta y_i$ (that is, for species $i$) to the chain at some interpolated time $t_0$. Define $n_h=\min\{k: t^h_k\ge t_0\}$, with $t^h_k$ was defined in \eqref{typo}. Then starting at step $n_h$, apply $[\Delta y_i/h]$ successive harvesting steps on species $i$.
	Let $Y^h\cd$ denote the piecewise interpolation of the harvesting strategy just defined.
	With the observation above,
	let $({Y}^h, {C}^h)$ denote the interpolated form of the adaption.
	By the weak convergence argument analogous to that of preceding theorems, we obtain the  weak convergence $$\big(X^h\cd, W^h\cd, Y^h\cd, A^h\cd\big)\to \big({X}\cd, {W}\cd,{Y}\cd, {A}\cd\big),$$
	where $A(t)=\int_0^t C(s)ds$, and the limit solves \eqref{e2.2.2}. It follows that	
	$$J(x,   {Y}\cd, C\cd)\ge  V^U(x)-\e.$$
	By the optimality of $V^h(x)$ and the above weak convergence,
	$$V^h(x)\ge J^h(x, u^h)\to J(x, {Y}\cd, {C}\cd).$$
	It follows that
	$\liminf\limits_{h\to 0}V^h(x)\ge V^U(x) -\e$.
	Since $\e$ is arbitrarily small, $\liminf\limits_{h\to 0}V^h(x)\ge V^U(x)$. 
	Therefore, $V^h(x)\to V^U(x)$ as $h\to 0$. If \eqref{e2.5} holds,
	by Proposition \ref{prop1} we have $V^U(x)=V(x)$ which finishes the proof.
\end{proof}

\subsection{Transition Probabilities for bounded harvesting and seeding rates}\

In this case, recall that $u^h_n= (\pi^h_n, Q^h_n)$ for each $n$ and $u^h=\{u^h_n\}_n$ be a sequence of controls. It should be noted that $\pi^h_n = 0$
includes the case that we harvest nothing and also seed nothing; that is, $Q^h_n=0$. Note also that
$\mathcal{F}^h_n=\sigma\{X^h_m, u^h_m, m\le n\}$.

The sequence $u^h= (\pi^h, Q^h)$
is said to be admissible if it satisfies the following conditions:
\begin{itemize}
	\item[{\rm (a)}]
	$u^h_n$ is
	$\sigma\{X^h_0, X^h_1,\dots, X^h_{n}, u^h_0, u^h_1,\dots, u^h_{n-1}\}-\text{adapted},$
	\item[{\rm (b)}]  For any $x\in S_{h+}$, we have
	$$\P\{ X^h_{n+1} = x | \mathcal{F}^h_n\}= \P\{ X^h_{n+1} = x | X^h_n, u^h_n\} = p^h( X^h_n, x| u^h_n),$$
	\item[{\rm (c)}] Let $X^{h}_{n, j}$  be the $j$ th component of the vector $X^h_n$ for $j=1, 2, \dots, d$. Then
	$$ \P\big(
	\pi^h_{n}=\min\{j: X^{h}_{n, j} = U+h\}  | X^{h}_{n, j} = U+h \text{ for some } j\in \{1, \dots, d \}, \mathcal{F}^h_n\big)=1.
	$$
	\item[{\rm (d)}] $X^h_n\in S_{h+}$ for all $n\in\mathbb{Z}_{\geq 0}$.
\end{itemize}

Now we proceed to define transition probabilities $p^h (x, y | u)$ so that the controlled Markov chain $\{X^h_n\}$ is locally consistent with respect to the controlled diffusion $X\cd$.
 For $(x, u)\in S_{h+}\times \mathcal{U}$ with $u=(0, q)$, we define
\beq{e3.4.7}
\barray
\aad Q_h (x, u)=\sum\limits_{i=1}^d a_{ii}(x) -\sum\limits_{i, j: i\ne j}\dfrac{1}{2}|a_{ij}(x)| +h\sum\limits_{i=1}^d |b_i(x) +q_i| +h,\\
\aad p^h \(x, x+h\ei |u\) =
\dfrac{a_{ii}(x)/2-\sum\limits_{j: j\ne i}|a_{ij}(x )|/2+\big(b_{i}(x)+q_i\big)^+ h }{Q_h (x, u)}, \\
\aad p^h \(x, x-h \ei | u\) =
\dfrac{a_{ii}(x)/2-\sum\limits_{j: j\ne i}|a_{ij}(x )|/2+\big(b_i(x) +q_i)^- h}{Q_h (x, u)}, \\
\aad p^h \( x, x+h\ei +h \ej) | u\) =  p^h \( x, x-h\ei -h\ej | u\) =
\dfrac{a_{{ij}}^+(x)}{2Q_h (x, u)}, \\
\aad p^h \( x, x+h\ei -h \ej | u\) = p^h \( x, x-h\ei + h \ej |  u \)=
\dfrac{a_{{ij}}^-(x)}{2Q_h (x, u)}, \\
\aad   p^h \( x, x | u\) =\dfrac{h  }{ Q_h (x, u)},\qquad  \Delta t^h (x, u)=\dfrac{h^2}{Q_h(x, u)}.
\earray
\eeq
Set $p^h \(x, y|u=(0, q)\)=0$ for all unlisted values of $y\in S_{h+}$.
Assumption \ref{a:2} guarantees that
the transition probabilities in \eqref{e3.4.7} are well-defined. At the reflection steps, we define
\beq{e3.4.8}
\barray
\aad p^h\( x, x - h \ei| u = (i, q)\)=1 \quad\text{and} \quad \Delta t^h(x, u=(i, q))=0,\quad  i=1, 2, \dots, d.
\earray
\eeq
Thus, $p^h \(x, y|u=(i, q)\)=0$ for all unlisted values of $y\in S_{h+}$.

\subsection{Transition Probabilities for unbounded seeding and bounded harvesting rates}\

In this case, recall that $u^h_n= (\pi^h_n, R^h_n)$ for each $n$ and $u^h=\{u^h_n\}_n$ be a sequence of controls.
 It should be noted that $\pi^h_n = 0$
includes the case that we harvest nothing; that is, $R^h_n=0$.
Note also that
$\mathcal{F}^h_n=\sigma\{X^h_m, u^h_m, m\le n\}$.

The sequence $u^h= (\pi^h, R^h)$
is said to be admissible if it satisfies the following conditions:
\begin{itemize}
	\item[{\rm (a)}]
	$u^h$ is
	$\sigma\{X^h_0, X^h_1,\dots, X^h_{n}, u^h_0, u^h_1,\dots, u^h_{n-1}\}-\text{adapted},$
	\item[{\rm (b)}]  For any $x\in S_{h+}$, we have
	$$\P\{ X^h_{n+1} = x | \mathcal{F}^h_n\}= \P\{ X^h_{n+1} = x | X^h_n, u^h_n\} = p^h( X^h_n, x| u^h_n),$$
	\item[{\rm (c)}] Let $X^{h}_{n, j}$  be the $j$ th component of the vector $X^h_n$ for $j=1, 2, \dots, d$. Then $$ \P\big(
	\pi^h_{n}=\min\{j: X^{h}_{n, j} = U+h\}  | X^{h}_{n, j} = U+h \text{ for some } j\in \{1, \dots, d \}, \mathcal{F}^h_n\big)=1.
	$$
	\item[{\rm (d)}] $X^h_n\in S_{h+}$ for all $n\in\mathbb{Z}_{\geq 0}$.
\end{itemize}

Now we proceed to define transition probabilities $p^h (x, y | u)$ so that the controlled Markov chain $\{X^h_n\}$ is locally consistent with respect to the controlled diffusion $X\cd$. We use the notations as in the preceding case. For $(x, u)\in S_{h+}\times \mathcal{U}$ with $u=(0, r)$, we define
\beq{e4.4.7}
\barray
\aad Q_h (x, u)=\sum\limits_{i=1}^d a_{ii}(x) -\sum\limits_{i, j: i\ne j}\dfrac{1}{2}|a_{ij}(x)| +h\sum\limits_{i=1}^d |b_i(x) -r_i| +h,\\
\aad p^h \(x, x+h\ei |u\) =
\dfrac{a_{ii}(x)/2-\sum\limits_{j: j\ne i}|a_{ij}(x )|/2+\big(b_{i}(x)-r_i\big)^+ h }{Q_h (x, u)}, \\
\aad p^h \(x, x-h \ei | u\) =
\dfrac{a_{ii}(x)/2-\sum\limits_{j: j\ne i}|a_{ij}(x )|/2+\big(b_i(x) -r_i)^- h}{Q_h (x, u)}, \\
\aad p^h \( x, x+h\ei +h \ej) | u\) =  p^h \( x, x-h\ei -h\ej | u\) =
\dfrac{a_{{ij}}^+(x)}{2Q_h (x, u)}, \\
\aad p^h \( x, x+h\ei -h \ej | u\) = p^h \( x, x-h\ei + h \ej |  u \)=
\dfrac{a_{{ij}}^-(x)}{2Q_h (x, u)}, \\
\aad   p^h \( x, x | u\) =\dfrac{h  }{ Q_h (x, u)},\qquad  \Delta t^h (x, u)=\dfrac{h^2}{Q_h(x, u)}.
\earray
\eeq
Set $p^h \(x, y|u=(0, r)\)=0$ for all unlisted values of $y\in S_{h+}$.
Assumption \ref{a:2} guarantees that
the transition probabilities in \eqref{e4.4.7} are well-defined. At the reflection steps, we define
\beq{e4.4.8}
\barray
\aad p^h\( x, x - h \ei| u = (i, r)\)=1 \quad \text{and}\quad  \Delta t^h(x, u=(i, r))=0,\quad i=1, 2, \dots, d.
\earray
\eeq
As a result, $p^h \(x, y|u=(i, r)\)=0$ for all unlisted values of $y\in S_{h+}$.
At the seeding steps, we define
$$ p^h\( x, x + h \ei| u = (-i, r)\)=1 \quad \text{and}\quad  \Delta t^h(x, u=(-i, r))=0,\quad  i=1, 2, \dots, d.
$$
Thus, $p^h \(x, y|u=(-i, r)\)=0$ for all unlisted values of $y\in S_{h+}$.

\end{document}